\documentclass[11pt]{article}
\usepackage{fullpage}
\usepackage{setspace}
\usepackage{natbib}
\setcitestyle{authoryear,open={((},close={))}} 
 \bibpunct[, ]{(}{)}{,}{a}{}{,}%
\makeatletter
\renewcommand\section{\@startsection {section}{1}{\z@}%
	{-3.5ex \@plus -1ex \@minus -.2ex}%
	{2.3ex \@plus.2ex}%
	{\normalfont\fontfamily{phv}\fontsize{16}{19}\bfseries}}
\renewcommand\subsection{\@startsection{subsection}{2}{\z@}%
	{-3.25ex\@plus -1ex \@minus -.2ex}%
	{1.5ex \@plus .2ex}%
	{\normalfont\fontfamily{phv}\fontsize{14}{17}\bfseries}}
\renewcommand\subsubsection{\@startsection{subsubsection}{3}{\z@}%
	{-3.25ex\@plus -1ex \@minus -.2ex}%
	{1.5ex \@plus .2ex}%
	{\normalfont\normalsize\fontfamily{phv}\fontsize{14}{17}\selectfont}}
\makeatother
\usepackage{amsmath}
\usepackage{graphicx}
\usepackage{enumerate}
\usepackage{natbib} %comment out if you do not have the package
\usepackage{url} % not crucial - just used below for the URL
\usepackage{xcolor}
\usepackage{amsfonts}
\usepackage{amsthm}
\usepackage{fixltx2e}
\usepackage{accents}
\MakeRobust{\underaccent} % make \underaccent not fragile in moving arguments

\usepackage{xcolor}
\usepackage{floatrow}
\usepackage[ruled,lined,linesnumbered,commentsnumbered,longend]{algorithm2e}
\usepackage{tikz}
\usepackage{wrapfig}
\usepackage{enumitem}
\usepackage{lineno}
\usepackage{enumitem}
\usepackage{subcaption}
\usepackage{comment}
\usepackage{multirow}
\usepackage{verbatim}

\usepackage{float}
\floatstyle{plaintop}
\restylefloat{table}
\usepackage{tablefootnote}
\usepackage{wrapfig}
\usepackage{bbm}

\usetikzlibrary{decorations.pathreplacing}
\tikzset{
	font={\fontsize{10pt}{10}\selectfont}}
\newcommand{\ubar}[1]{\underaccent{\bar}{#1}}
\newcommand{\R}{\mathbb{R}}

\newcommand{\USB}{\mathcal{U}_{\text{SB}}}

\newcommand{\set}[2]{\left\{#1 \; \left|\;\; #2 \right.\right\}}
\newcommand{\abs}[1]{\left|#1\right|}
\newcommand{\card}[1]{\abs{#1}}
\newcommand{\ext}{\text{ext}}

\newcommand{\shimrit}[1]{\textcolor{black}{#1}}

\newcommand{\noam}[1]{\textcolor{black}{#1}}

\newcommand{\noamr}[1]{\textcolor{black}{#1}}
\newcommand{\shimritr}[1]{\textcolor{black}{#1}}
\newcommand{\noamrr}[1]{\textcolor{black}{#1}}

\newcommand{\vectornorm}[1]{\left|\left|#1\right|\right|}
\newcommand{\bu}{{\bf u}}
\newcommand{\bv}{{\bf v}}
\newcommand{\bx}{{\bf x}}
\newcommand{\bbd}{{\bf D}}
\newcommand{\bd}{{\bf d}}

\newcommand{\bphi}{\boldsymbol{\phi}}
\newcommand{\norm}[1]{\|#1\|}
\DeclareMathOperator{\argmax}{argmax}
\usepackage{authblk}

\newtheorem{theorem}{Theorem}
\newtheorem{observation}[theorem]{Observation}
\newtheorem{assumption}[theorem]{Assumption}
\newtheorem{lemma}[theorem]{Lemma}
\newtheorem{proposition}[theorem]{Proposition}

\DeclareMathOperator{\adjacent}{adjacent}
\DeclareMathOperator{\Infs}{Infs}
\DeclareMathOperator{\Cols}{Cols}
\DeclareMathOperator{\dist}{{dist}}
\newcommand{\loc}{a}

\definecolor{myblue1}{RGB}{35,119,189}
\definecolor{myblue2}{RGB}{95,179,238}
\definecolor{myblue3}{RGB}{129,168,207}
\definecolor{myblue4}{RGB}{26,89,142}
\definecolor{deepred}{RGB}{180,0,0}
\definecolor{darkred}{RGB}{139,0,0}
\definecolor{deepgreen}{RGB}{0,140,0}
\definecolor{darkgreen}{RGB}{0,100,0}
\definecolor{deepblue}{RGB}{0,30,200}
\definecolor{darkblue}{RGB}{0,15,100}
\definecolor{Red}{rgb}{1, 0, 0} % Red of svgnames
\definecolor{Green}{rgb}{0.2, .8, 0} % Red of svgnames
\definecolor{Blue}{rgb}{.255,.41,.884} % 

\begin{document}
%%%%%%%%%%%%%%%%%%%%%%%%%%%%%%%%%%%%%%%%%%%%%%%%%%%%%%%%%%%%%%%%%%%%%%%%%%%%%%
\def\spacingset#1{\renewcommand{\baselinestretch}%
	{#1}\small\normalsize} \spacingset{1}
%%%%%%%%%%%%%%%%%%%%%%%%%%%%%%%%%%%%%%%%%%%%%%%%%%%%%%%%%%%%%%%%%%%%%%%%%%%%%%	
	
%\TITLE
\title{Robust Radiotherapy Planning with Spatially Based Uncertainty Sets}
%\RUNTITLE{Radiotherapy Planning with Spatial Uncertainty}
%\RUNAUTHOR{N. Goldberg, M. Langer, and S. Shtern} 
%\ARTICLEAUTHORS{%
%\AUTHOR
\author{}
%\author[1]{Noam Goldberg}
%\affil[1]{Department of Management, Bar-Ilan University}
%\author[2]{Mark P. Langer}
%\affil[2]{School of Medicine, Indiana University}
%\author[3]{Shimrit Shtern}
%\affil[3]{Faculty of Data and Decision Sciences, Technion - Israel Institute of Technology}

\author{Noam Goldberg $^a$, Mark Langer $^b$, and Shimrit Shtern $^c$ \\
	$^a$ Department of Management, Bar-Ilan University, Ramat Gan, Israel \\
	$^b$ School of Medicine, Indiana University, Indianapolis, IN, USA\\ 
	$^c$ Faculty of Data and Decision Sciences, Technion - Israel Institute of Technology, Haifa, Israel }
\date{}

%\begin{comment}
\maketitle

\begin{abstract}
Radiotherapy treatment planning is a challenging large-scale optimization problem plagued by uncertainty. Following the robust optimization methodology, we propose a novel,
spatially based uncertainty set for robust modeling of radiotherapy
planning, producing solutions that are immune to unexpected changes in biological conditions. Our proposed uncertainty set realistically captures biological radiosensitivity patterns that are observed using recent advances in imaging, while its parameters can be personalized for individual patients. We exploit the structure of this set to devise a compact reformulation of the robust model. We develop a row-generation scheme to solve real, large-scale 
instances of the robust model. This 
method is then extended to a relaxation-based scheme for enforcing challenging, yet clinically important, dose-volume cardinality constraints. 
The computational performance of 
our algorithms, as well as the quality and robustness of the computed treatment plans, are demonstrated on simulated and real imaging data. Based on accepted performance measures, such as minimal target dose and homogeneity, these examples demonstrate that the spatially robust model achieves almost the same performance as the nominal model in the nominal scenario, and otherwise, the spatial model outperforms both the nominal and the box-uncertainty models. 
\end{abstract}
%\end{comment}
%\KEYWORDS
	\noindent%
\spacingset{1.5}

\section{Introduction\label{sec:INTRO}}
Radiotherapy treatment, and in particular intensity modulated radiation therapy (IMRT), involves setting up an array of beams with variable intensities to irradiate a tumor.  
Planning such a treatment entails selecting the intensity of each beam or beam unit (called a beamlet). To facilitate such planning, the effect of the beam intensities -- that is, the resulting radiation dose (whose unit of measure is 1 Gy = 1 J/kg) – on both the tumor and the healthy surrounding organs must be  
analyzed. 
The dose is measured over a three-dimensional grid of spatial cube units called voxels. The size of a voxel, and the granularity of the partition into voxels, are determined by the radiation beamlet size and the imaging resolution. Recent technology allows for voxel edge lengths of 2-3 mm. 
The dose at each voxel is generally assumed to be linearly related to the chosen beamlet intensities.  
In general, this is the theme of radiotherapy treatment planning (RTP), which may prescribe the treatment over both time and space. The particular optimization problem for setting up the beamlet intensities is known as fluence map optimization (FMO).  
A variety of formulations have been proposed for this problem. Typically, such formulations require that the tumor voxels are sufficiently irradiated,   
while healthy tissue voxels are spared from receiving 
an unhealthy dose.  
The formulations differ in the way that these requirements are handled.  Additional clinical requirements can also be factored into the model. 
Some formulations aim to maximize the average tumor voxel dose~\citep{PreciadoWalters2006}, while constraining the dose received by healthy organs. Other formulations minimize a measure of the dose received by healthy tissue voxels, or by all voxels, subject to ensuring that target (i.e., tumor) voxels receive at least the prescribed dose~\citep{Romeijn2006}. 

Our approach is to maximize the tumor's ``weakest link'': the voxels receiving the minimum dose of all tumor voxels. This is a widely accepted objective for evaluating radiotherapy planning~\shimritr{(\citealp{deasy1997multiple,yan1997adaptive,ICRU2010}, and \citealp{PreciadoWalters2004})}, subject to 
constraints on the healthy organ dose and on tumor dose homogeneity. {In particular, the objective of maximizing the minimum dose objective is argued to better support biological optimization, as ``a conservative surrogate for biological effect'', compared with minimization of deviations from prescribed doses~\citep{deasy1997multiple}. The importance of combining this objective with homogeneity constraints is outlined by~\cite{yan1997adaptive,Allen2012}, \shimritr{and} \cite{Khairi2021}. 
%~\cite{yan2000off}.
}

Advances in imaging have exposed uncertainties in the effect of the prescribed radiation dosage. {It is now known that tumors are not homogeneous, but rather, the sensitivity of cells varies across the tumor and changes over time~\shimritr{(\citealp{bentzen2011,Saka2014,Titz2008}, and \citealp{Nohadani2017})}. 
In an effort to improve patient outcomes, there is a need for the model to account for the uncertain problem data. Modeling uncertainty in an optimization problem usually results in added complexity and scale. 
%In this paper, we consider two main sources of complexity: the uncertainty involved in FMO data and dose-volume constraints. 
Uncertainty is involved in all radiotherapy treatment planning stages, starting with the target volume definition, but also in relation to the dose required to control the tumor, healthy tissue tolerances, and the actual dose delivered to the patient.
Uncertainty in radiotherapy planning has been extensively addressed in the literature, mainly in the context of geometric uncertainty that results from the patient's movement and from imaging inaccuracies; see, for example,~\cite{Bortfeld2008}, the recent survey by \cite{Unkelbach2018}, and the references therein. 

Approaches for dealing with uncertainty in FMO include stochastic programming, robust optimization and distributionally robust optimization. The stochastic approach assumes that the distribution of the uncertainty is known and its parameters can be reasonably estimated. 
In practice, estimating the exact distribution is difficult, and incorrectly estimating the distribution may lead to solutions that are not clinically viable. 
In contrast, the robust optimization methodology \citep{BenTal09} does not  
assume any knowledge of the distribution, but rather, defines a set that contains all possible uncertainty realizations that need to be hedged against. 
By making few or no assumptions 
about the distribution, the robust approach optimizes the treatment for the worst-case realization of the uncertainty. Motivated by motion uncertainty, \cite{Chu2006} propose a robust formulation with uncertain influence matrices. 
The distributionally robust approach combines the stochastic and robust approaches -- 
it protects against the worst-case distribution within a family of possible probability distributions of the uncertain parameters.  For example, \cite{Bortfeld2008}  
model motion uncertainty by defining families of distributions that correspond to possible motion patterns of patients.  \cite{unkelbach2007} consider a (discrete) uncertainty set for beamlet range uncertainties in proton therapy (IMPT). This uncertainty set is constructed so that beamlets of the same beam must reach the same range. In this work they consider the objective of minimizing a weighted sum of the deviations from a prescribed dose.  However, the model being solved may be overly conservative as the worst case is applied to each of the voxels separately, rather than to the objective. A similar scenario-based approach to account for range uncertainty and setup uncertainty in IMPT is also used in \cite{Pflugfelder2008}.

Another source of uncertainty, also mentioned in \cite{Unkelbach2018}, arises from the biological conditions of the patients, which the authors consider a new frontier of robust RTP optimization. These conditions may affect the radiosensitivity of the tumor to the dose. The resulting effective dose is referred to as the biologically-adjusted dose. 
{We note that this adjustment is different from the biological effective dose (BED) often modeled by the linear-quadratic (LQ) model \citep{jones2001use}, which aims to estimate the biological effectiveness in terms of cell-kill probability for 
a given physical dose. In contrast, the biologically adjusted dose corrects the physical dose to account for the impact of a biological condition on the effective dose at a particular voxel. For the case of oxygen levels (or the lack thereof also known as hypoxia), this is described by the oxygen enhancement ratio (OER), which is the ratio between doses required to obtain the same biological effect \shimritr{(\citealp{hill2015hypoxia} and \citealp{mcmahon2018linear})}, or its normalized variant known as the oxygen {modification} factor (OMF)~\citep{Titz2008}.
Information about a living tissue's (e.g., a tumor's) biological conditions is available through bio-markers 
and imaging data. Specifically, the tumor's oxygenation level, which strongly affects its radiosensitivity, is measured by imaging techniques such as fluoromisonidazole (FMISO)~\citep{toma2012}. 
Translating the imaging data into radiosensitivity values is a crude estimation that is subject to uncertainty. Further, the oxygenation level, and accordingly the radiosensitivity, may change between the time of the image scan and the commencement of treatment. 
In fact, biological uncertainty in radiation therapy, and the effect of oxygenation on radiosensitivity in particular, is an active area of research in the medical physics and other research communities~\shimritr{(\citealp{Titz2008,toma2012}, and \citealp{Shang2021})}.  
Biological conditions and radiosensitivities are considered in the context of FMO by~\cite{Saka2014} {but their proposed methods do not attempt to model and protect from the associated uncertainty. Robust static (e.g., cumulative) dose planning under biological uncertainty has been proposed by~\cite{Li2015} and \cite{Eikelder2020}. \cite{Eikelder2020} propose a tumor control probability (TCP) based model assuming voxel independence and simple box uncertainty for its parameters. Simple box uncertainty here implies conservative solutions where all voxels are assumed to be simultaneously insensitive. \cite{Li2015} propose a dose-based model with budgeted uncertainty (see~\cite{Bertsimas2004}) on the deviation of the voxels' prescribed doses from their nominal value. In this model, the conservatism is controlled by a single budget parameter that %indicates 
\noamr{controls} how many \shimritr{voxels  have  prescribed doses that (maximally) deviate from %maximally deviates from 
%its 
their nominal values.}. However, in this model the budget is allocated linearly to the change in the prescribed doses, so it might not accurately capture measurement errors whose effect on radiosensitivity may be highly nonlinear~\shimritr{\citep{Titz2008}}. Also, this model does not limit the differences in radiosensitivities among neighboring voxels, \shimritr{and thus may consider unrealistic scenarios where the dose in one voxel radiosensitivity is significantly lower than its neighboring voxels.}}
%and~\cite{Eikelder2020}, but the methods applied {either do not attempt to model the uncertainty at all} or do not use the robust optimization methodology.  {To the best of our knowledge, the only work that considers robust FMO under biological uncertainty at the voxel level is that of \cite{Nohadani2017}. 

{Recent papers that consider radiosensitivity uncertainty in adaptive radiation therapy planning include~\citet{Nohadani2017,roy2022}, \shimritr{and} \cite{jeyakumar2023}. These papers assume a perscribed dose approach where the objective is to minimize the total dose administered while satisfying simple lower and upper dose bound constraints on target voxels and OARs, respectively. Consequently, the {biological uncertainty in this formulation appears only on the
right-hand side}. %In their work, 
\cite{Nohadani2017} utilize simple box uncertainty sets 
 containing the measured radiosensitivities, and %put an emphasis on 
{model} the evolution of this uncertainty through time. In the current paper, we also formulate and solve a robust FMO that accounts for biological uncertainty at the voxel level where the main focus is to model the spatial connectivity or dependence of the voxel uncertainties. 
%Note that dose limits may be appropriate for well studied cancer types but may be less appropirate for medical research purposes and for planning in case of less studied cases.
%For the models proposed by \cite{roy2022,jeyakumar2023} the static problem would amount to a box uncertainty model similar to~\cite{Nohadani2017}.
Both~\cite{roy2022} \shimritr{and} \cite{jeyakumar2023} propose models similar to~\cite{Nohadani2017}. \cite{roy2022} studies an extension of the model to a multistage adaptive problem, while~\cite{jeyakumar2023} consider adaptivity to uncertainty in  influence matrix  resulting in a semidefinite program. } 

%However, 
{I}n contrast to previous work, our approach is to model and explicitly account for the {spatial} interdependence among individual voxel {radiosenstivity} uncertainties. Accounting for such interdependence is expected to %yield 
improve %ment 
over %alternative, 
simpler robust approaches that tend to result in overly conservative solutions. {Spatial dependence has been considered in other applications; for example by estimating covariance matrices and applying these using ellipsoid uncertainty sets in transportation~\shimritr{\citep{Chassein2019}}, or general polyhedral uncertainty sets in energy~\shimritr{\citep{Lorca2014}}. These previous approaches may be less appropriate for biological uncertainty in radiotherapy because there may not be sufficient data to reliability estimate the correlations between voxel radiosensitivities. Also, applying both types of uncertainty sets to large-scale radiotherapy optimization problems, may be computationally intractable in practice. In radiotherapy, the scenario-based approaches to range and setup uncertainties in IMPT appear to resemble robust optimization with spatial uncertainty (for example \cite{unkelbach2007} \shimritr{and} \cite{Pflugfelder2008}
discussed above). %Resembling spatial uncertainty in robust radiotherapy planning of~\cite{unkelbach2007}, discussed in the above in the context of range and setup uncertainties. 
However, the uncertainty applies to the beamlets rather than to voxels and it is addressed by scenario-based approaches, so the dependency is not explicitly modeled.} %\noam{To this end}
\shimrit{{In contrast}, the current \noam{paper} focuses on %the 
%introducestion of a 
spatial uncertainty mode\noam{l}ling,  reformulation of
\noam{the corresponding robust FMO problem} and solution \noam{methods for solving this problem}. %In this work we do not address how to administer the overall dose within the multiple treatment treatment sessions (known as fractions). The decision on 
\noam{Dividing the administered dose into fractions and accounting for treatment sessions over time \noam{(also known as fractions)}} %the fraction sizes 
may %require to also address 
{involve a more elaborate uncertainty model over both space and time} %temporal uncertainty and which require more involved models and reformulation techniques which are outside the scope of the current paper and 
{that} is left for future work. %While this may also be an interesting line of research, addressing both the spatial and temporal uncertainty 
%This work  whereas its interaction with temporal uncertainty  
}
 
To make our proposed model applicable to clinical settings, we also consider important dose-volume constraints. These constraints limit the percentage of healthy-organ voxels that can receive more than a specified radiation dose. These are hard combinatorial constraints, which are virtually impossible to handle exactly given the problem size. 
Hence, many solution approaches have been proposed to approximately solve such constraints in the literature~\shimritr{(\citealp{ferris2003,PreciadoWalters2006}, and \citealp{Romeijn2006})}. However, effectively %formulating and 
solving models that incorporate such constraints remains a challenging problem, especially in the context of the large scale problems considered herein. 

 To summarize our current contribution, we  
 propose a novel FMO formulation that accounts for voxel radiosensitivity and its associated uncertainty, in both the objective and in the homogeneity constraints.  
 To this end, Section~\ref{sec:modeling} presents our novel spatially dependent uncertainty set with justification based on publicly available real biomarker (FMISO imaging) data. In Section~\ref{sec:LSRO}, we develop a scalable constraint generation algorithm for solving the resulting 
 robust formulation.  Section~\ref{sec:tuning_beta} presents a method for relaxing the dose-volume constraints -- a method that can be solved and tuned in order to determine a feasible solution to {the robust formulation when augmented with a 
 dose-volume constraint}. 
 Finally, Section~\ref{sec:computational} includes a numerical study that demonstrates the scalability of our method, shows how to personalize the treatment plan, and presents a study of the effectiveness of our method under different scenarios.

\section{Modeling Robust Radiotherapy Planning}\label{sec:modeling}
 We now present our model for robust radiotherapy planning with biological uncertainty. In Section~\ref{sec:nominal},  
the nominal model, which is the planning optimization model in the absence of uncertainty, is discussed.   In Section~\ref{sec:uncertainty}, the robust approach to hedging against uncertainty is presented. This approach includes the development of a novel biological uncertainty set and its incorporation into the nominal problem.

\subsection{Model Introduction - Nominal Model}\label{sec:nominal}

In the following, we use the notation $[n]=\{1,\ldots,n\}$, for any integer $n\geq 1$. Additionally, for a vector $a\in \R^n$ and index set  $I\subset [n]$, $a_I\in \R^{\card{I}}$ denotes the subvector of $a$ with coordinate set $I$. 
The set of all voxel indices is $[m]$ for some positive integer $m$, which is 
partitioned into the planning target volume (PTV) index set $T\subset [m]$, containing tumor voxel indices, and index sets corresponding to the $K$ organs at risk (OARs), $H_k\subset [m]$, for each $k\in [K]$, such that 
$[m]=T\cup \bigcup_{k=1}^K H_k$.  

FMO involves determination of the  $n$ beamlet intensities. For  $i\in [n]$,  $x_i$ is beamlet $i$'s  intensity decision variable. Given the intensity vector $\bx\in\R^n$ and an influence matrix $\bbd\in \R^{m\times n}$,  for each voxel $v\in  [m]$,  $d_v(x)=\sum_{i=1}^n D_{vi}x_i$ is \emph{the physical dose} applied to voxel $v$.

The objective is to maximize a concave function of the dose associated with the PTV voxels, $f(\bd_T)$, while constraining the OAR dose. For convenience, throughout most of the paper, we will consider a simple OAR
bound of the form  
$d_v(x)\leq \bar{d}_{k}$ for each $k\in[K]$ and $v\in H_k$, where $\bar{d}_k$ is a predefined upper bound on the dose for OAR $k\in[K]$. In Section~\ref{sec:tuning_beta}, we consider more elaborate dose-volume constraints.  
An important 
aspect of the problem that affects the definition of the function $f$, which is also the main focus of the model considered in the current paper, is that the effect of the dose on a given PTV voxel $v\in T$ depends on biological factors, such as the PTV's oxygenation level. This information is 
modeled by a radiosensitivity vector $\phi\in [0,1]^{\card T}$.   Accordingly, {\emph{the biologically adjusted dose}  (or simply \emph{adjusted dose})} of voxel $v$ is given by $\phi_v{d}_v$. For example, if considering the effect of oxygen levels then $\phi$ would be a vector of OMF values, as discussed in Section~\ref{sec:INTRO}. The radiosensitivity is usually computed based on biomarker  
and imaging data. Notably, 
$\phi$ is highly uncertain due to imaging and conversion  inaccuracies and  variation in the oxygen level over time. 
Yet, the radiotherapy planning goal is generally to administer as much {adjusted}  dose to the PTV voxels as possible,  
which is accomplished here by setting $f(\bd)=\min_{v\in T} \phi_vd_v(x)$. % deleted subvector set subscript - does not work

{We also consider homogeneity constraints that limit the dose range across the PTV \citep{PreciadoWalters2004}. 
This is intended to prevent the creation of ``hot spots'' in the tumor tissue, which cause damage to the healthy tissue in which the tumor is embedded.} 
Specifically, 
our homogeneity constraints are of the form $\mu f(\bd)-\phi_v{d}_v\geq 0$, for all $v\in T$, where $\mu>1$ is a given homogeneity parameter. 
Homogeneity has been traditionally applied to the physical dose~\shimritr{(\citealp{Allen2012,kataria2012}, and \citealp{Khairi2021})}. 
The possibility of applying it to the adjusted dose has been enabled by recent technological advances in estimating necessary biological information, specifically %in our model 
the radiosensitivity vector $\phi$. The main goal of the current paper is to address the challenge  %associated with describing the 
of modeling the uncertainty in $\phi$ and using this information to improve radiotherapy treatment planning. The optimization problem of determining the intensity of the beams given the radiosensitivity parameter $\bphi$ is formulated as 
\begin{subequations}\label{prob:nominal}
    \begin{align}
        &\underset{\substack{
        \bx\in\R^n_+,\ubar{d}\in\R}}{\text{maximum}} &&  \ubar{d}\\
        &\text{subject to}&&
        \phi_{v}d_v(x)\geq \ubar{d}, && v\in T\\
        &&& \mu \ubar{d}-\phi_{v}d_v(x)\geq 0,&& v\in T \label{constr:HOMOGEN}\\
        &&&d_v(x)\leq \bar{d}_{k},&& k\in [K],v\in H_k,\label{OARSIMPLECONS} 
    \end{align}
\end{subequations}
which will be referred to as \emph{the nominal formulation}.

Note that $\bphi$ is derived from imaging data through a series of crude transformations. Thus, it is potentially highly inaccurate. Ignoring this inaccuracy in practice may lead to poor clinical outcomes in terms of controlling the adjusted dose and {adjusted homogeneity}. The next section 
extends the model to address the uncertainty in $\bphi$. 

\subsection{Robust Optimization Model with Spatially {Bound}
%{Based} 
Uncertainty Set}\label{sec:uncertainty}

We now present a modeling scheme to deal with the uncertainty in the parameter vector $\bphi$. 
The radiosensitivity uncertainty is caused by multiple factors, including imaging noise, biomarker conversion error, and variation over time.
Inferring a distribution of the radiosensitivity parameter based on limited patient-specific (personalized) data 
is nontrivial and prone to error. Moreover,  
maximizing the clinical outcome (in our model, maximizing the minimum PTV adjusted dose) while using a {high-dimensional multivariate} 
probability distribution to model the biological uncertainty,  
would result in chance-constrained optimization problem~\citep{luedtke2010integer}, which is notorious for being computationally hard.  
Instead, our approach to modeling uncertainty   
follows the \emph{robust optimization (RO)} methodology. In this modeling framework, the uncertain parameters are assumed to lie in a predetermined \emph{uncertainty set} denoted by $\mathcal{U}$. 
{ An RO problem formulation is designed to protect 
its solutions against all realizations in the uncertainty set.  
In particular,} an optimal solution $(\bx,\ubar{d})$ must remain feasible for every possible realization $\bphi\in\mathcal{U}$ while maximizing the worst-case (minimum) PTV adjusted dose. 

Formulating the robust counterpart (RC) of~\eqref{prob:nominal}, in which variable $\underbar d$ now depends on the uncertain $\phi$ subject to   uncertain homogeneity constraints~\eqref{constr:HOMOGEN}, would lead to an adjustable {RO} problem. Such problems are computationally intensive \citep{Ben-Tal04}, and although it may be possible to solve these adjustable formulations when the problem is moderately sized, we find it more effective to avoid the adjustable robust model by reformulating the homogeneity constraints according to the equivalence 
\begin{align}\label{eq:homogeneity} \mu\phi_u{d}_u- \phi_v{d}_v\geq 0, \text{ for all } v,u\in T && \Leftrightarrow && \mu f(\bd)-\phi_v{d}_v\geq 0 \text{, for all } v\in T,\end{align} at the cost of
increasing the number of constraints by a factor of $O(\card{T})$. 
The resulting robust problem is given by

\begin{subequations}\label{prob:general_robust}
    \begin{align}
        &\underset{\substack{%\bdelta_k\in \R^{|H_k|\times M_k},k\in[K],\\
        \bx\in\R^n_+,\ubar{d}\in\R}}{\text{maximum}}&& \ubar{d}%-\sum_{k=1}^q\sum_{j=1}^{M_k}\beta_{kj}\sum_{v\in H_{k}} y_{kvj}&
        \\
        &\text{subject to}&&
        \min_{\bphi\in\mathcal{U}}\phi_{v}d_v(x)\geq \ubar{d}, && v\in T\label{constr:ROBUSTMIN}\\
        &&&\max_{\bphi\in \mathcal{U}}\left\{ \phi_{v}d_v(x)- \mu \phi_{u} d_u(x)\right\}\leq 0,&& u,v\in T\label{constr:GENROBUSTHOMOGEN}\\
        &&& d_v(x)%-y_{kvj}
        \leq \bar{d}_{k}
        ,&& k\in [K],v\in H_k.
    \end{align}
\end{subequations}
%The challenge is therefore to 
Constructing the uncertainty set poses some challenges, in particular, trying to satisfy and balance the following set of requirements for $\mathcal{U}$:
\begin{enumerate}[label=(\roman*)]
    \item \label{REQ1} Personalized for the patient. 
    This requirement relies on patient specific information, such as imaging data acquired at the start of the treatment.
    \item \label{REQ2} 
    Computationally tractable. The nominal problem may already be computationally intensive given the large number of voxels. 
        The computational effort of solving the robust problem should not be excessive; in particular, it should not significantly increase the running time relative to that required to solve the %original 
        nominal problem.
    \item \label{REQ3} Not too conservative. 
    A goal of our %the 
    modeling approach %presented herein 
    is to determine a set that captures the essence of the uncertainty, thereby providing ``reasonable'' protection, while not being too conservative, and thus tackling one of the major criticisms of RO.
\end{enumerate}
In the next section, we present a novel uncertainty set 
that models the spatial dependence between  
voxels 
while addressing these requirements.

\subsubsection{The Spatially {Bound} Uncertainty Set.}
Given a radiosensitivity vector $\widehat{\bphi}$ that is estimated based on a patient's  
imaging data at the start of the treatment, the  goal is to construct an  
uncertainty set that satisfies our three modeling requirements, described in Section~\ref{sec:uncertainty}. Perhaps the simplest uncertainty set that can be applied is a box (in particular, a hypercube) 
uncertainty set, which allows the radiosensitivity of each PTV voxel $v$ to deviate from its calculated value $\widehat{\phi}_v$ by at most $\delta>0$, where $\delta$ is a given parameter. That is,
$\mathcal{U}_{\text{B}}=\{\bphi\in \R^{|T|}: |\phi_{v}-\widehat{\phi}_v|\leq \delta, \; \forall v\in T\}$. 
This uncertainty set, used for example in \cite{Nohadani2017}, admits a straightforward robust formulation. However, if $\delta$ is chosen to be sufficiently large to capture the actual uncertainty in the tumor biological conditions, then the resulting set may be overly conservative, thereby 
violating one of our key modeling requirements (see Section~\ref{sec:uncertainty}). To restrict the uncertainty further, a key observation to be made is that voxels that are situated close to each other should have similar radiosensitivities. In Figure~\ref{fig:Patient1VoxelRadioSensitive}, both plots demonstrate that the radiosensitivity differences are small for neighboring voxels %(distance of one), 
and increase with the distance. %up to a distance of ten.

\begin{figure}
\begin{subfigure}[t]{0.45\textwidth}
\includegraphics[scale=0.38]{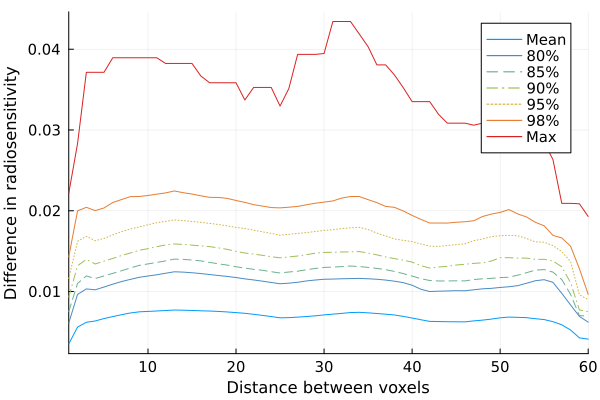}
\caption{
Pairwise radiosensitivity difference percentiles vs. pairwise distances.\label{fig:Patient1_VoxDistVSRSDist}}
\end{subfigure}
\begin{subfigure}[t]{0.45\textwidth}
\includegraphics[scale=0.38]{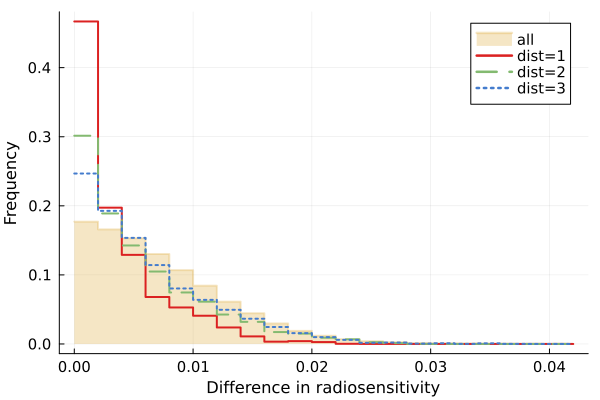}
\caption{
Histogram of the pairwise radiosensitivity differences for pairwise distances of 1, 2, 3, and all.
\label{fig:Patient1_RSDistHist}}
\end{subfigure}
\caption{Relationship between radiosensitivity difference and pairwise PTV voxel distance in the first-visit data of Patient 1 from the TCIA brain dataset~\citep{Clark2013}.
\label{fig:Patient1VoxelRadioSensitive}}
\end{figure}

Incorporating this observation into the uncertainty model leads to a \emph{spatially r{bound} (SB) uncertainty set}; in particular, we define 
\begin{equation}\label{eq:unc_SR}\USB=\{\bphi\in \R^{|T|}: |\phi_{v}-\widehat{\phi}_v|\leq \delta, \; |\phi_v-\phi_u|\leq \gamma_{vu},\; \forall v,u\in T\},\end{equation}
where $\gamma_{vu}$ are parameters of the uncertainty set that may depend on characteristics of the voxel pair $v$ and $u$. Specifically, the $\gamma_{uv}$ parameters are assumed to depend on the pairwise distances, and in particular, they are given by a metric-preserving function of the distance. %, which rests on the %following 
%next assumption.  
\begin{assumption}\label{ass:gamma_struct}
For any $u,v\in T$,
 $\gamma_{uv}=\Gamma(\dist(u,v))$,
where $\dist(u,v)$ is the  
distance in voxels between $u$ and $v$, and
$\Gamma:\R\rightarrow [0,1]$ is a nondecreasing subadditive function satisfying $\Gamma(x)=0$ if and only if $x=0$. \end{assumption}
Defining $\gamma_{uv}$ as a function of the distance between $u$ and $v$ in this way captures the interdependence of the radiosensitivities of nearby voxels and limits the conservatism associated with the box uncertainty set. To the best of our knowledge, such spatially dependent uncertainty sets have not been previously investigated in the context of robust radiotherapy planning, nor it seems in the context of {RO} in general.

The $\gamma_{uv}$ parameters for $u,v\in T$ can be personalized to a particular patient’s data, thereby addressing requirement~\ref{REQ1}.  Requirement~\ref{REQ3} may be satisfied by imposing the spatial constraints associated with the parameters $\gamma_{uv}$, for $u,v\in T$, to augment the box constraints in the definition of the spatially {bound} uncertainty set. In the absence of these spatial constraints, the uncertainty set would reduce to the overly conservative box uncertainty set.
The following section addresses the question of whether optimal solutions of~\eqref{prob:general_robust}  
for uncertainty set $\USB$ can be computed efficiently in practice (requirement~\ref{REQ2}).

\subsubsection{Compactly Reformulating The Robust Problem}
In this section, we present an exact reformulation of \eqref{prob:general_robust} for the case where $\mathcal{U}=\USB$ with a polynomial number of constraints with respect to the size of the nominal formulation~\eqref{prob:nominal}.

The following result, which will be used in the proofs that  follow, establishes  that $\gamma_{uv}$ is a
metric.
\begin{lemma}[{\citealt[][Proposition~3.2]{Corazza1999}}]\label{lem:properties_of_gamma}
Suppose that Assumption~\ref{ass:gamma_struct} holds. Then, 
$\gamma_{\cdot}=\Gamma(d(\cdot)):T\times T\rightarrow [0,1]$ is a metric. 
\end{lemma}

For $S\subseteq \R^n$, denote the projection of set $S$ onto index set $I\subseteq [n]$ by $P_I(S)$, so that    
$P_I(S)=\set{v_I}{v\in S}$. 
For simplicity, for the singleton $I=\{i\}$, the projection will be denoted by $P_i(S)$.
%Next, we state 
The following is a straightforward observation that is used in %deriving the 
reformulating~\eqref{prob:general_robust}. 
\begin{observation}\label{lem:proj_reform}
Let $\mathcal{U}\subseteq\R^n$ be a compact set. 
Then, for any {continuous function} $f:\R^{|I|}\rightarrow \R$,  
 $   \min_{\bu\in \mathcal{U}} f(\bu_I)=\min_{\bv\in P_I(\mathcal{U})} f(\bv)$.
\end{observation}

Problem \eqref{prob:general_robust} is now specialized and reformulated for the uncertainty set $\USB$. Henceforth, it is assumed that $\USB\neq\emptyset$. 
Starting with the the first set of constraints~\eqref{constr:ROBUSTMIN}, and using Observation~\ref{lem:proj_reform}, for each $v\in T$, the inner minimization in this constraint is given by
\begin{equation}\label{eq:const1}
    \min_{\bphi\in\USB}\phi_{v}d_v(x)=\min_{\phi_v\in P_v(\USB)}\phi_vd_v(x).
\end{equation}
Since matrix $\bbd$ and vector $\bx$ are both nonnegative, 
the solution of \eqref{eq:const1} 
is given by the minimal value of $\phi_v$. For convenience, in the following, for each $u\in T$, define $\ubar{\phi}^0_u=\max\{0,\widehat{\phi}_u-\delta\}$, $\bar{\phi}^0_u=\min\{1,\widehat{\phi}_u+\delta\}$, $\ubar\phi_u= \min_{\bphi\in \USB}\phi_u$, and $\bar\phi_u=\max_{\bphi\in \USB}\phi_u$.  
The following proposition characterizes the   
one-dimensional projection of $\USB$.

\begin{figure}[h!]
\centering
\begin{tikzpicture}[scale=0.3]
		\draw[help lines, color=gray!30, dashed] (-1,-2) grid (30,11);
		\draw[->,ultra thick] (-1,0)--(30,0) node[right]{};
		\draw[->,ultra thick] (0,-1)--(0,11) node[above]{$\phi$};
		\node [myblue1] at (5,5){\textbullet};
		{\node [] at (6,5) {$\hat{\phi}_v$};}
        %\draw [decorate,decoration={brace,amplitude=6pt},xshift=-0.2pt,yshift=0pt](5,3) -- (5,5) node [black,midway,xshift=-10pt] {\footnotesize $\delta$};
         %\draw [decorate,decoration={brace,amplitude=6pt},xshift=-0.2pt,yshift=0pt](5,5) -- (5,7) node [black,midway,xshift=-10pt] {\footnotesize$\delta$};
          \node [black] at (5,0) {\textbar};
        \node [black] at (5,-1) {$v$};
		\node [myblue1] at (15,3.5){\textbullet};
          \node [black] at (15,0) {\textbar};
        \node [black] at (15,-1) {$u$};
		{\node [] at (16,3.5) {$\hat{\phi}_u$};}
		\node [myblue1] at (25,7){\textbullet};
        \node [black] at (25,0) {\textbar};
        \node [black] at (25,-1) {$w$};
		{\node [] at (26,7) {$\hat{\phi}_w$};}
		\draw[|-|,thick,myblue1] (5,7)--(5,3);
		\draw[|-|,thick,myblue1] (15,5.5)--(15,1.5);
		\draw[|-|,thick,myblue1] (25,9)--(25,5);
	
 %\node [myblue1] at (4,6) {$+\delta$};
%	\node [myblue1] at (4,4) {$-\delta$};
    \node[myblue1] at (6.2,7) {$\bar{\phi}^0_v$};
    \node[myblue1] at (6.2,3) {$\ubar{\phi}^0_v$};
    \node[myblue1] at (16.2,5.5) {$\bar{\phi}^0_u$};
    \node[myblue1] at (16.2,1.5) {$\ubar{\phi}^0_u$};
    \node[myblue1] at (26.2,9) {$\bar{\phi}^0_w$};
    \node[myblue1] at (26.2,5) {$\ubar{\phi}^0_w$};
    \draw[|-|,thick,deepred] (15,5.5)--(15,6.5);
		\draw[|-|,thick,deepred] (15,0.5)--(15,1.5);
		\draw[|-|,thick,deepred] (25,9)--(25,10.7);
		\draw[|-|,thick,deepred] (25,3.3)--(25,5);

    \node[deepred] at (13,6) {$+\bar{\gamma}_{uv}$};
    \node[deepred] at (13,1.5) {$-\bar{\gamma}_{uv}$};
    \node[deepred] at (23,9.5) {$+\bar{\gamma}_{wv}$};
    \node[deepred] at (23,4.5) {$-\bar{\gamma}_{wv}$};
    \draw[deepgreen,thick,dashed] (15,6.5)--(5,6.5);
		\draw[deepgreen,thick,dashed] (25,10.7)--(5,10.7);
		\draw[deepgreen,thick,dashed] (15,0.5)--(5,0.5);
		\draw[deepgreen,thick,dashed] (25,3.3)--(5,3.3);	
		\draw[|-|,ultra thick,deepgreen] (5,3.3)--(5,6.5);
		\node [deepgreen] at (3.7,6.5) {$\bar{\phi}_v$};
		\node [deepgreen] at (3.7,3.3) {$\ubar{\phi}_v$};
		\end{tikzpicture}
  \caption{{An illustration of the one-dimensional projection result of Proposition~\ref{prop:one_dim_proj}.} {Assuming only three voxels, $v, u,$ and $w$, the bounds on the radiosenstivity of voxel $v$ are determined by tightening its original interval ${[{\ubar{\phi}}^0_v,\bar{\phi}^0_v]}$ by the intersection with
  $[\ubar{\phi}^0_u - \gamma_{uv},\bar{\phi}^0_u + \gamma_{uv}]$ 
  and $[\ubar{\phi}^0_w-\gamma_{wv},\bar{\phi}^0_w+\gamma_{wv}]$, imposed by voxels $u$ and $w$, respectively.
  }\label{fig:onedimproj}}
\end{figure}
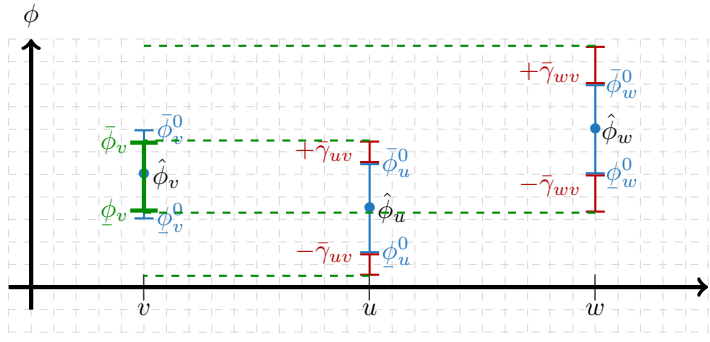 

\begin{proposition}\label{prop:one_dim_proj} 
For every $v\in T$: 
\begin{enumerate}[label=(\roman*)]
\item\label{prop:one_dim_proj_part0} \mbox{$P_{v}(\USB)=[\ubar\phi_v,\bar\phi_v]$}.
\item\label{prop:one_dim_proj_part1} %\begin{align*}
    $\ubar\phi_v\geq \max_{u\in T\setminus\{v\}}\{\ubar\phi_u- \gamma_{uv}\}$ and  
    $\bar\phi_v \leq \min_{u\in T\setminus\{v\}}\{\bar \phi_u+\gamma_{uv}\}$.
\item\label{prop:one_dim_proj_part2} Suppose that Assumption~\ref{ass:gamma_struct} holds. Then,   
\begin{align}\ubar\phi_v=\max_{u\in T}\{\ubar\phi^0_u- \gamma_{uv}\} && \text{ and } &&  \bar\phi_v=\min_{u\in T}\{\bar\phi^0_u+\gamma_{vu}\}\label{eq:ONEDIMPROJDEF}.   %for $v\in T$ let
\end{align}
\end{enumerate}
\end{proposition}
\begin{proof}%{Proof.}
\underline{Part \ref{prop:one_dim_proj_part0}:}
 $P_{v}(\USB)=[\ubar\phi_v,\bar\phi_v]$ 
follows from that fact that $\USB$ is a convex and compact set 
and from Observation~\ref{lem:proj_reform}, which together imply that the projection of $\USB$ onto $v$ is a closed interval  $[\ubar{\phi}_v,\bar{\phi}_v]$. 

\underline{Part \ref{prop:one_dim_proj_part1}:} 
For the sake of contradiction, suppose  that $\ubar\phi_v<\ubar\phi_u-\gamma_{uv}$ for some $u\in T\setminus\{v\}$, and let $\phi^*$ be a minimizer (vector) that attains $\ubar\phi_v$. Then, by optimality of $\phi^*_v=\ubar\phi_v$,
$\phi^*_v-\phi^*_u \leq \ubar\phi_v - \ubar\phi_u < -\gamma_{uv}$, 
%\] 
thereby contradicting the feasibility of $\phi^*$. The proof of the upper bound on $\bar\phi_v$ is similar. 

\underline{Part \ref{prop:one_dim_proj_part2}:} 
Let $\phi^*\in\argmax_{\phi\in\USB}\{\phi_v\}$, so that $\bar{\phi}_v=\phi^*_v$. 
First, observe that by the optimality of $\phi^*_v$, there is at least one binding constraint, $\bar{\phi}_v\leq \bar\phi^0_v$ or $\bar{\phi}_v-\phi^*_{u_1}\leq \gamma_{vu_1}$, for some $u_1\in T$. Thus, $\bar{\phi}_v=\min\{\bar\phi^0_v,\phi^{*}_{u_1}+\gamma_{v{u_1}}\}$ for some $u_1\neq v$. If $\bar{\phi}_v=\bar\phi^0_v$ then $\bar{\phi}_v=\bar\phi^0_v+\gamma_{vv}$ from Lemma~\ref{lem:properties_of_gamma}, 
and the claim follows. 
Otherwise, $\bar{\phi}_v=\phi^{*}_{u_1}+\gamma_{v{u_1}}$.
Applying this argument repeatedly to components of $\phi^*$, it follows that there exist $u_2,\ldots,u_k\in T$, for some $k\leq \card{T-1}$, satisfying \begin{align}\phi^{*}_{u_l}=\phi^{*}_{u_{l+1}}+\gamma_{u_lu_{l+1}}, && \text{ for } &&  l=1\ldots, k-1, && \text{ and } && \phi^{*}_{u_k}=\bar{\phi}^0_{u_k}.\label{eq:PATHDEF}\end{align}
Lemma~\ref{lem:properties_of_gamma} (triangle inequality) implies that 
%\begin{align*}
    $\bar{\phi}_v  %&= 
    \bar\phi^0_{u_k} +\sum_{l=1}^{k-1} \gamma_{u_lu_{l+1}}+\gamma_{vu_1}\geq  \bar\phi^0_{u_k}+\gamma_{vu_k}$. 
%\end{align*}
Next, from the feasibility of $\phi^*$ and~\eqref{eq:PATHDEF}, 
$\bar{\phi}_v= \phi^*_v\leq \phi^*_{u_k}+\gamma_{vu_k}= \bar\phi^0_{u_k}+\gamma_{vu_k}$, so it actually holds as an equality.  
A similar line of reasoning is used to prove the formula for $\ubar{\phi}_v$.
%\Halmos  
\end{proof}

The one-dimensional projection established by Proposition~\ref{prop:one_dim_proj} is graphically depicted by Figure~\ref{fig:onedimproj}.
%\begin{wrapfigure}[12]{r}{0.6\textwidth}
Note that following Proposition~\ref{prop:one_dim_proj}, $\USB$ can be alternatively written as 
\[\USB=\{\bphi\in \R^{|T|}: \ubar\phi_v\leq \phi_{v}\leq \bar\phi_v,\; \forall v\in T, \;\;\quad  |\phi_v-\phi_u|\leq \gamma_{vu},\; \forall v,u\in T\}.\]{Two dimensional examples of $\USB$ are graphically illustrated in Figure~\ref{fig:uncertaintyset}. The two-dimensional case will bear a special importance compactly reformulating the robust problem~\eqref{prob:general_robust}. }
Next, we reformulate the second set of constraints~\eqref{constr:GENROBUSTHOMOGEN}. Again, due to Observation~\ref{lem:proj_reform}, for each $v,u\in T$, the inner maximization is 
\begin{equation}\label{eq:reform_second_const}
    \max_{\bphi\in\USB}\left\{ \phi_{v}d_v(x)- \mu \phi_{u} d_u(x)\right\}=\max_{(\phi_v,\phi_u)\in P_{\{v,u\}}(\USB)} \left\{\phi_{v}d_v(x) - \mu \phi_{u} d_u(x)\right\}.
\end{equation}
The next proposition establishes that the projection onto a pair of indices $u,v\in T$,  $P_{\{v,u\}}(\USB)$, can be expressed 
as a polyhedral set defined by  
six constraints.

\begin{figure}[b]
\begin{subfigure}%[10]{R}
[t]{0.45\textwidth}
\centering
\begin{tikzpicture}[scale=3]
%\node at (0.5,1.1) {f=1};
\draw[->] (0,0) -- (1,0) node[anchor=north] {$\phi_{v}$};
\draw[->] (0,0) -- (0,1) node[anchor=east] {$\phi_{u}$};
\foreach \x in {0,0.2,0.4,0.6,0.8}{
	\draw (\x,0) node[anchor=north] {\x};};
\foreach \y in {0,0.2,0.4,0.6,0.8}{
	\draw (0,\y) node[anchor=east] {\y};};
\draw[dotted] (0.3,-0.1) -- (0.3,1);
\draw (0.3,-0.1) node[anchor=north] {$\ubar{\phi}_{v}$};
\draw[dotted] (0.7,-0.1) -- (0.7,1);
\draw (0.7,-0.1) node[anchor=north] {$\bar{\phi}_{v}$};
\draw[dotted] (-0.15,0.1) -- (1,0.1);
\draw (-0.15,0.1) node[anchor=east] {$\ubar{\phi}_{u}$};
\draw[dotted] (-0.15,0.5) -- (1,0.5);
\draw (-0.15,0.5) node[anchor=east] {$\bar{\phi}_{u}$};
\fill [fill=red,fill opacity=0.1] (0.3,0.1)--(.3,.5)--(.7,.5)--(.7,.1)--(0.3,0.1);
\draw[green, thick] (0.3,0.2)--(.6,.5)--(.4,.5)--(.3,.4)--(0.3,0.2);
\fill [fill=blue,fill opacity=0.1] (0,0)--(.1,0)--(1,.9)--(1,1)--(0.9,1)--(0,.1);
\draw[dashed] (0,0) -- (1,1);
\draw [decorate,decoration={brace,amplitude=6pt},xshift=-0.2pt,yshift=0pt]
(0.3,0.3) -- (0.3,0.4) node [black,midway,xshift=-15pt] {\footnotesize
	$\gamma_{vu}$};
\draw [decorate,decoration={brace,amplitude=6pt},xshift=0pt,yshift=0.2pt]
(0.4,0.5) -- (0.5,0.5) node [black,midway,yshift=10pt] {\footnotesize
	$\gamma_{vu}$};
\filldraw[red] (0.3,0.2) circle (0.5pt);
\filldraw[red] (0.6,0.5) circle (0.5pt);
\end{tikzpicture}%
%\vspace{-16pt}
\end{subfigure}
\begin{subfigure}%[10]{R}
[t]{0.45\textwidth}
\centering
%\vspace{-3em}
\begin{tikzpicture}[scale=3]
%\node at (0.5,1.1) {f=1};
\draw[->] (0,0) -- (1,0) node[anchor=north] {$\phi_{v}$};
\draw[->] (0,0) -- (0,1) node[anchor=east] {$\phi_{u}$};
\foreach \x in {0,0.2,0.4,0.6,0.8}{
	\draw (\x,0) node[anchor=north] {\x};};
\foreach \y in {0,0.2,0.4,0.6,0.8}{
	\draw (0,\y) node[anchor=east] {\y};};
 
\draw[dotted] (0.3,-0.1) -- (0.3,1);
\draw (0.3,-0.1) node[anchor=north] {$\ubar{\phi}_{v}$};
\draw[dotted] (0.7,-0.1) -- (0.7,1);
\draw (0.7,-0.1) node[anchor=north] {$\bar{\phi}_{v}$};
\draw[dotted] (-0.15,0.3) -- (1,0.3);
\draw (-0.15,0.3) node[anchor=east] {$\ubar{\phi}_{u}$};
\draw[dotted] (-0.15,0.7) -- (1,0.7);
\draw (-0.15,0.7) node[anchor=east] {$\bar{\phi}_{u}$};
\fill [fill=red,fill opacity=0.1] (0.3,0.3)--(.3,.7)--(.7,.7)--(.7,.3)--(0.3,0.3);
\draw[green, thick] (0.4,0.3)--(.7,.6)--(.7,.7)--(.6,.7)--(0.3,0.4)--(0.3,0.3)--(0.4,0.3);
\fill [fill=blue,fill opacity=0.1] (0,0)--(.1,0)--(1,.9)--(1,1)--(0.9,1)--(0,.1);
\draw[dashed] (0,0) -- (1,1);
\draw [decorate,decoration={brace,amplitude=6pt},xshift=-0.2pt,yshift=0pt]
(0.3,0.3) -- (0.3,0.4) node [black,midway,xshift=-14pt] {\footnotesize
	$\gamma_{vu}$};
\draw [decorate,decoration={brace,amplitude=6pt},xshift=0pt,yshift=0.2pt]
(0.6,0.7) -- (0.7,0.7) node [black,midway,yshift=10pt] {\footnotesize
	$\gamma_{vu}$};
\filldraw[red] (0.4,0.3) circle (0.5pt);
\filldraw[red] (0.7,0.6) circle (0.5pt);
\end{tikzpicture}%
%\vspace{-16pt}
\end{subfigure}
\caption{{An illustration of two different two-dimensional projections of $\USB$ for voxels $v$ and $u$. 
The shaded square depicts the $\delta$-bounds on the deviation from the measured $\widehat{\phi}$. The gray $45^0$ band depicts the $\gamma_{vu}$ bound on the radiosensitivity difference between the voxels. The green polygons are 
  the projected uncertainty sets. The red points indicate the vertices that potentially maximize~\eqref{eq:reform_second_const}. }\label{fig:uncertaintyset}}
\end{figure}
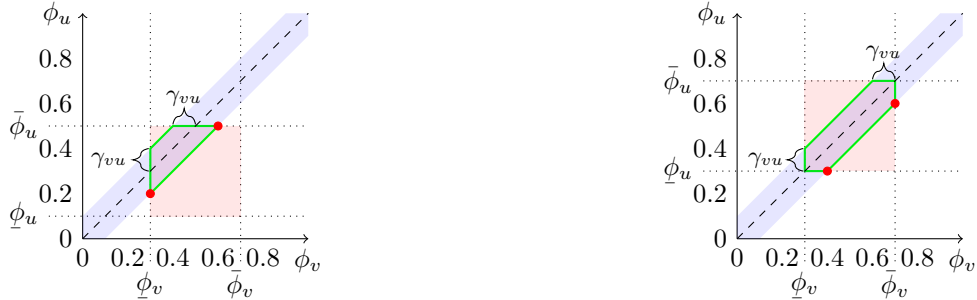

\begin{proposition}\label{prop:two_dim_proj}
Suppose some $v,u\in T$, $\ubar{\bphi}\in\R^{|T|}$, and $\bar{\bphi}\in\R^{|T|}$ 
whose components are given by~\eqref{eq:ONEDIMPROJDEF}, and that Assumption~\ref{ass:gamma_struct} holds. Then,  
$$P_{\{v,u\}}(\USB)=\mathcal U_{u,v}\equiv\{(\phi_v,\phi_u): \quad \ubar{\phi}_v \leq  \phi_v\leq \bar{\phi}_v,\;\quad
\ubar{\phi}_u \leq \phi_u\leq \bar{\phi}_u,\;\quad
\abs{\phi_v-\phi_u}\leq \gamma_{vu}
\}.$$
\end{proposition}

\begin{proof}%{Proof.}
Let $u,v\in T$, and for convenience, define also  $\tilde{\mathcal{U}}_{u,v}=\{(\phi_v,\phi_u): -\gamma_{vu}\leq \phi_v-\phi_u\leq \gamma_{uv}\}$,  
to show that $P_{\{u,v\}}(\USB)\subseteq \mathcal{U}_{u,v}$. 
The contractive property of the projection onto a plane, as a convex set,  implies that $P_{\{u,v\}}(\USB)\subseteq 
P_{u}(\USB)\times P_{v}(\USB)$. Then, by Proposition~\ref{prop:one_dim_proj}-\ref{prop:one_dim_proj_part0} and the definition of $\USB$, 
it follows that $$P_{\{u,v\}}(\USB)\subseteq   
(P_{u}(\USB)\times P_{v}(\USB)) 
\cap  \tilde{\mathcal{U}}_{uv} = ([\ubar\phi_u,\bar\phi_u]\times [\ubar\phi_v,\bar\phi_v])\cap \tilde{\mathcal{U}}_{uv}=\mathcal U_{u,v}.$$

To show the reverse inclusion, 
let $(\tilde{\phi}_u,\tilde{\phi}_v)\in {\mathcal{U}}_{u,v}$. 
Consider extending this two dimensional vector to 
a $\card T$-dimensional vector in $\USB$ in order to prove that $(\tilde{\phi}_u,\tilde{\phi}_v)\in P_{\{u,v\}}(\USB)$.   
Now, consider some $w\in T\setminus\{u,v\}$ and define  
$\tilde\phi_w=\max\{\tilde{\phi}_v-\gamma_{vw},\tilde{\phi}_u-\gamma_{uw}, \ubar{\phi}_w\}$. This choice of $\tilde\phi_w$ satisfies the lower bounds imposed by $\tilde\phi_u,\tilde\phi_v$ and $\ubar\phi_w$ in the definition of~$\USB$. 
Also, to prove the corresponding upper bounds, first note that  
$$
\tilde\phi_w\leq \max\{\tilde{\phi}_v-\gamma_{vw},\tilde{\phi}_v+\gamma_{vu}-\gamma_{uw},\tilde{\phi}_v+\gamma_{vw}\}\leq \tilde{\phi}_v+\gamma_{vw},$$ 
where the first inequality is due to  
$(\tilde{\phi}_u,\tilde{\phi}_v)\in {\mathcal{U}}_{u,v}$ and $\tilde\phi_v\geq \ubar{\phi}_v\geq \ubar{\phi}_w-\gamma_{vw}$ (the latter inequality following from Proposition~\ref{prop:one_dim_proj}-\ref{prop:one_dim_proj_part1}) 
 and the second inequality is due to Assumption~\ref{ass:gamma_struct} and  
Lemma~\ref{lem:properties_of_gamma} (triangle inequality).
Similarly,
$%\max\{\tilde{\phi}_v-\gamma_{vw}, \tilde{\phi}_u-\gamma_{uw} , \ubar{\phi}_w\}
\tilde\phi_w\leq \max\{\tilde{\phi}_u+\gamma_{vu}-\gamma_{vw},\tilde{\phi}_u-\gamma_{uw},\tilde{\phi}_u+\gamma_{uw}\}\leq \tilde{\phi}_u+\gamma_{uw}$.
Finally, the fact that $(\tilde\phi_u,\tilde\phi_v)\in \mathcal U_{u,v}$, 
 together with the definitions of $\ubar\phi$ and $\bar\phi$, imply that 
$\tilde\phi_w=\max\{\tilde{\phi}_v-\gamma_{vw},\tilde{\phi}_u-\gamma_{uw},\ubar{\phi}_w\}\leq \max\{\bar{\phi}_v-\gamma_{vw},\bar{\phi}_u-\gamma_{uw},\ubar{\phi}_w\}\leq \bar{\phi}_w$. 
Next, consider (arbitrary) $w'\in T\setminus\{u,v,w\}$ in order to verify that $\tilde{\phi}_w$ and similarly defined~$\tilde\phi_{w'}$ together  satisfy\begin{equation}\label{eq:wRSDist}
-\gamma_{ww'}\leq \tilde{\phi}_{w}-\tilde{\phi}_{w'}\leq \gamma_{ww'}.
\end{equation}
Starting with the lower bound, 
\begin{align*}
\tilde{\phi}_{w}-\tilde{\phi}_{w'}&= \max\{\tilde{\phi}_{v}-\gamma_{vw},\tilde{\phi}_u-\gamma_{uw},\ubar{\phi}_{w}\}-\max\{\tilde{\phi}_v-\gamma_{vw'},\tilde{\phi}_u-\gamma_{uw'},\ubar{\phi}_{ w'}\}\\
&\geq 
\min\{\gamma_{vw'}-\gamma_{vw},\gamma_{uw'}-\gamma_{uw},\ubar{\phi}_{w}-\ubar{\phi}_{w'}\}.
\end{align*}
By definition of {$\ubar{\phi}$}, 
$\ubar{\phi}_{w}-\ubar{\phi}_{w'}\geq -\gamma_{ww'}$. 
By Lemma~\ref{lem:properties_of_gamma} (triangle inquality) 
we also have that $\gamma_{vw'}-\gamma_{vw},\gamma_{uw'}-\gamma_{uw}\geq -\gamma_{ww'}$. 
Proving the upper bound  
is analogous {(by interchanging the roles of $w$ and $w'$ and since $\gamma_{ww'}=\gamma_{w'w}$)}. %Since the assumed $w$ and $w'$ are arbitrary it follows that
Thus,~\eqref{eq:wRSDist} holds for all $w,w'\in T$.  
Finally, we have established that $\tilde\phi\in \USB$, and accordingly, it follows that 
${\mathcal{U}}_{u,v}\subseteq P_{\{u,v\}}(\USB)$.%\Halmos
\end{proof}

\bigskip

Following the results of Propositions~
\ref{prop:one_dim_proj}-\ref{prop:two_dim_proj}, formulation
 \eqref{prob:general_robust}  
 with $\mathcal{U}=\USB$ can now be compactly reformulated with a finite number of constraints. In particular, in this reformulation, no variables are added and the number of constraints of this formulation is $2|T|$ times that of the nominal problem \eqref{prob:nominal}, thereby remaining polynomial in the size of the nominal formulation. 
Consider the formulation
    \begin{subequations}\label{prob:SR_robust}
    \begin{align}
        &\underset{\substack{
        \bx\in\R^n_+,\ubar{d}\in\R}}{\text{maximum}}\hskip-10pt&& \ubar{d}\\
        &\text{subject to}&&
        \ubar\phi_{v} d_v(x)\geq \ubar{d}, & v\in T\label{constr:ROB_MIN}\\
        &&& \bar\phi_{v}d_v(x)- \mu \max\{\bar\phi_{v}-\gamma_{vu},\ubar\phi_{u}\} d_u(x)\leq 0,\hskip-10pt& (u,v)\in T^2:u\neq v\label{constr:ROB_HOMOGEN1}\\
        &&& \min\{\ubar\phi_{u}+\gamma_{vu},\bar{\phi}_v\}d_v(x) - \mu  \ubar\phi_u d_u(x)\leq 0,& (u,v)\in T^2:u\neq v\label{constr:ROB_HOMOGEN2}\\
        &&&\ d_{v}(x) %-y_{kvj}
        \leq \bar{d}_{k},& k\in [K],v\in H_k.\label{constr:OAR_UB}%,j\in[M_k]
    \end{align}
    \end{subequations}
The following theorem establishes the correctness of this formulation, and specifically, demonstrates that its optimal solution set coincides with that of~\eqref{prob:general_robust}.

\begin{theorem}
    Suppose that Assumption~\ref{ass:gamma_struct} holds and let $\mathcal{U}=\USB$. Then, the optimal solution sets of \eqref{prob:general_robust} and~\eqref{prob:SR_robust} coincide.
\end{theorem}

\begin{proof}%{Proof.}
    The reformulation of %the first constraint 
    constraints~\eqref{constr:ROBUSTMIN} is a straightforward consequence of~\eqref{eq:const1} together %with Lemma~\ref{lem:proj_reform} 
    and Proposition~\ref{prop:one_dim_proj}.
   %The second constraint 
   Constraints~\eqref{constr:GENROBUSTHOMOGEN} %is 
   are reformulated based on  \eqref{eq:reform_second_const} to a maximization over $P_{\{v,u\}}(\USB)$ in place of a maximization over $\USB$. Due to Assumption~\ref{ass:gamma_struct}, this set is given by the result of Proposition~\ref{prop:two_dim_proj}. Further, since the objective is linear, 
   and since, for each $u,v\in T$, $P_{\{u,v\}}(\USB)$ is a {bound} 
   polyhedral set, {then by the fundamental theorem of linear programming, there}  is an optimal $(\phi_{u},\phi_{v})$ 
   that is an extreme point of the set $P_{\{u,v\}}(\USB)$. 
   Moreover, 
{since $\phi_v$ and $\phi_u$ have, respectively, a positive and a negative coefficient in~\eqref{constr:GENROBUSTHOMOGEN}, this implies that the   points in $\ext(P_{\{u,v\}}(\USB))$ 
are dominated by the points in}
\begin{align*} \mathcal U^*_{u,v}\equiv \Big\{
(\max\{\bar{\phi}_v-\gamma_{vu},\ubar{\phi}_u\},\bar{\phi}_v),
(\ubar{\phi}_u,\min\{\ubar{\phi}_u+\gamma_{vu},\bar{\phi}_v\})\Big\}\subset\ext(P_{\{u,v\}}(\USB));\end{align*}
that is, for each pair $v,u\in T$, constraint~\eqref{constr:GENROBUSTHOMOGEN} is satisfied 
if and only if its left-hand side is nonpositive for these two points in $\mathcal U^*_{u,v}$.
%\Halmos
\end{proof}
In Section~\ref{sec:LSRO}, we will address the computational aspect of solving~\eqref{prob:SR_robust}, specifically, how to deal with the potentially very large number of homogeneity constraints~\eqref{constr:ROB_HOMOGEN1} and~\eqref{constr:ROB_HOMOGEN2}.

\section{Solving Large-scale RO}\label{sec:LSRO}

The radiation therapy problems that we consider involve tens of thousands of tumor voxels and millions of healthy tissue voxels. Accordingly, formulation~\eqref{prob:SR_robust} may have hundreds of millions of constraints. Therefore, it may prove difficult to solve or even store this problem on most computers. In order to address this challenge, we propose an algorithm that dynamically generates constraints of~\eqref{prob:SR_robust} as needed.  

A schematic description is outlined by Algorithm~\ref{alg:RGAschem}. 
\begin{algorithm}[]
\caption{Constraint generation algorithm sketch}\label{alg:RGAschem}
\LinesNumbered
\SetKwInOut{Input}{Input}
\SetKwInOut{Output}{Output}
\SetKwInOut{Init}{Initialize}
\Input{Initial beamlet configuration {$x_0$}, number of initial constraints for each OAR $n_0$, number of homogeneity constraints per iteration $n_S$,  thresholds $\tau_{\Infs,1}>\tau_{\Infs,2}$ and $\tau_z$}
\Init{%{$x=x_0$}, 
$\hat H_1\subset H_1,\ldots,\hat H_K\subset H_K$ each with {$n_0$} constraints violated by $x_0$, $S\leftarrow\emptyset$\;}
\label{step:phase_i} Phase I: {Alternate between iteratively generating at most $n_H$ constraints ~\eqref{constr:OAR_UB} to each OAR until a feasibility tolerance $\tau_{\Infs,1}$ is achieved and the objective function value does not {decrease} by more than a factor of $\tau_z$, and appending at most $n_S$ homogeneity constraints \eqref{constr:ROB_HOMOGEN1}~\eqref{constr:ROB_HOMOGEN2} to set $S\subseteq T^2$. If none of the latter %homogeneity 
constraints are generated switch to Phase II}\;
%Iteratively solve~\eqref{prob:SR_robust} with subset of constraints ~\eqref{constr:ROB_HOMOGEN1},\eqref{constr:ROB_HOMOGEN2} for $S^2\subset T^2$ and $H_k$ replaced by $\hat H_k$ for each $k\in [K]$, generating OAR constraints~\eqref{constr:OAR_UB} as long as a loose feasibility tolerance {$\tau_{\Infs,1}$ is exceeded or the objective decreases by more than a factor of $\tau_z$}. It is followed by generation of at most $n_S$ homogeneity constraints~\eqref{constr:ROB_HOMOGEN1},~\eqref{constr:ROB_HOMOGEN2}, and then alternating between OAR and homogeneity constraint generation.\;
Phase II: {The same as Phase I, with feasibility tolerance of $\tau_{\Infs,2}<\tau_{\Infs,1}$ for constraints ~\eqref{constr:OAR_UB}. Exit once no additional constraints can be generated.}
%Tighten the violation threshold {of OAR constraints to} $\tau_{\Infs,2}$ and continue to resolve and generate both types of constraint until none is violated {[this is not true]}. 
\end{algorithm}
A detailed version of the algorithm in pseudocode is given {as} Algorithm~\ref{alg:RGA} in Appendix~\ref{appx:algo}. 
The algorithm is initialized with beamlet configuration $x_0$, typically a part of a previously computed solution or an all-ones vector multiplied by some normalized dose. Based on this initial configuration, each OAR is initialized with the most violated $n_0$ constraints. This initialization, and the initial emphasis on generating OAR constraints~\eqref{constr:OAR_UB} %in %Steps~\ref{INLOOPSTART}-\ref{INLOOPEND} {[This refers to Algorithm 3 which is now in the appendix and is not mentioned at all in algorithm 1, maybe should be added?]} 
%of 
in the algorithm, ensures that the optimization problem is bounded from the start or that it becomes bounded after a few iterations in phase I of the algorithm. The algorithm consists of two phases. Each one of them alternates between generation of OAR constraints~\eqref{constr:OAR_UB} and~\eqref{constr:ROB_HOMOGEN1}~\eqref{constr:ROB_HOMOGEN2}. The first phase involves an OAR constraint infeasibility tolerance $\tau_{\Infs,1}$ that is larger than the tolerance applied in the second phase, $\tau_{\Infs,2}$. Intuitively, the larger tolerance in the first phase it intended to ensure both that the subproblems become bounded fairly quickly and to delay generation of OAR constraints that may be inactive at optimality until the homogeneity constraints are nearly satisfied. 
%step~\ref{step:phase_i}.
%The algorithm proceeds in two phases: In the first phase, constraints \eqref{constr:OAR_UB} involving the healthy organs are %satisfied 
%generated until the objective function does not improve by more than $\tau_z$ {and} none of the constraints is violated by more than $\tau_{\Infs,1}$ followed by generation  %%{[This again refers to quantities we do not mention in Algorithm 1]}
%. The second {phase} of the algorithm alternates between adding these constraints and adding homogeneity constraints \eqref{constr:ROB_HOMOGEN1} and \eqref{constr:ROB_HOMOGEN2}, until all constraints are satisfied within an additive tolerance of $\tau_{\Infs,2}<{\tau_{\Infs,1}}$. {[not a correct description both phases alternate]} 

Next we consider additional constraints that pertain to healthy organs and we develop an extension of the proposed computational method to address such constraints. 

\section{Dose-Volume Constraints}\label{sec:tuning_beta}

In many cases, clinical guidelines stipulate, in addition to or in place of 
strict bounds on the OAR voxel dose, 
\emph{dose-volume constraints}. This refers to the practice of allowing some proportion  $\alpha\in (0,1)$, usually less than one half, of the voxels to receive a larger dose than the upper bound $\bar d_K$ applying to $1-\alpha$ of the voxels. %of radiation than the complementary proportion,
%\begin{equation}
%\sum_{v\in H_k}\mathbbm{I}(d_v>\bar{d}_{k})\leq \lfloor\alpha|H_k|\rfloor.\label{DVCONSTRAINT}
%\end{equation} 
In its most general form, a planning optimization problem could involve several dose-volume constraints for each organ. We develop a method for approximately solving large-scale planning problems in the case of a single dose-volume constraint.  
The single dose-volume constraint setting is clinically important, as illustrated by the brain-tumor dataset examined in the computational section. 

In general, dose-volume constraints can be formulated by introducing additional auxiliary binary variables into formulation~\eqref{prob:SR_robust}, turning it into the mixed-integer program, %{[changed $z$ to $w$ since $z$ is used later in a different cotext]} %shown as \eqref{prob:DV_MIP}. 
\begin{subequations}\label{prob:DV_MIP}
\begin{align}
& \max_{\bx\in\R^n_+,\ubar{d}\in\R,y\in\R^{|H_K|},w\in\{0,1\}^{\card{H_K}}} && \underbar d\\ %- \beta \sum_{v\in H_{K}}y_v\\
& \text{subject to} && \eqref{constr:ROB_MIN},\eqref{constr:ROB_HOMOGEN1},\eqref{constr:ROB_HOMOGEN2}\label{constr:ALL_ROBUST}\\
& && d_v(x) \leq \bar d_k && k\in [K-1],v\in H_k\label{constr:Hard_HK_BOUNDS2}\\
& && d_{v}(x) -y_v \leq \bar d_K\label{constr:SOFT_HK_BOUNDS}\\
& && y_v \leq (\hat d_K- \bar d_K)w_v && v\in H_K \label{constr:BOUND_DELTA}\\
& && \sum_{v\in H_K}w_v\leq \theta \label{DVCONSTRAINT}.
\end{align}
\end{subequations}
Here, %$\beta$ is a tunable penalty parameter that is used to control the deviations, and 
without losing generality, it is assumed that the dose-volume constraint applies to the $K$th OAR, auxiliary continuous deviation variables $y_{v}$ are defined  for each $v\in H_K$, and  
$\hat{d}_K$ is an ``absolute'' upper bound on the dose of each voxel in $H_K$. Setting $\theta=\lfloor\alpha|H_k|\rfloor$ in~\eqref{DVCONSTRAINT} would ensure compliance with a dose-volume guideline specified as a proportion $\alpha$.  

However, solving such a formulation to optimality 
involves a hard combinatorial problem that is neither theoretically nor practically tractable for the high-dimensional problems %such as those 
considered in the current paper.  The standard continuous relaxation of~\eqref{prob:DV_MIP} replaces %decision 
{the binary decision variables} $w\in\{0,1\}^{\card {H_K}}$ by {continuous ones} $w\in[0,1]^{\card {H_K}}$. Therefore, it is straightforward to see that replacing constraint~\eqref{DVCONSTRAINT} with $\sum_{v\in H_K} y_v \leq \theta\cdot (\hat d_K-\bar d_K)$, and fixing $w_v=1$ in~\eqref{constr:BOUND_DELTA}, for all $v\in H_K$, %preserves all 
maintains the optimality of all solutions that are optimal to the continuous relaxation of~\eqref{prob:DV_MIP}. 
%bounding the sum (or $L_1$-norm) of the continuous deviation variables $y_v$ in place of  the binary deviation indicator variables $z_v$, for $v\in H_K$. 
%Problem \eqref{prob:SR_robust} is then reformulated as
%For carefully selected values of the parameter $\beta$, the penalized problem~\eqref{prob:PEN_ROB_FORM} can be viewed as a relaxation {of a deviation-constrained problem with a right-hand side upper bound $\Theta\in \R_+$,}
Hence, fixing %$z=\mathbbm{1}$ {this is not defined}
{$w_v=1$ for all $v\in H_K$}, for $\Theta = \theta\cdot (\hat d_K-\bar d_K)$, the formulation  
\begin{subequations}\label{prob:BOUND_ROB_FORM}
\begin{align}
& \max_{\bx\in\R^n_+,\ubar{d}\in\R,y\in\R^{|H_K|}} && \underbar d\\
& \text{subject to} && \eqref{constr:ALL_ROBUST},\eqref{constr:Hard_HK_BOUNDS2},\eqref{constr:SOFT_HK_BOUNDS}&&\label{constr:bound_prev_constr}\\
&&& \sum_{v\in H_K} y_v \leq \Theta \label{constr:BUDGET}\\
& &&  y_v \leq \hat d_K - \bar d_K && v\in H_K\label{constr:YBND}
\end{align}
\end{subequations}
is a continuous relaxation of~\eqref{prob:DV_MIP}. As a continuous relaxation, solutions of this formulation~\eqref{prob:BOUND_ROB_FORM} with $\Theta = \theta\cdot (\hat d_K-\bar d_K)$ may provide bounds but tend to be infeasible for~\eqref{prob:DV_MIP}. On the other hand, solutions of~\eqref{prob:BOUND_ROB_FORM} for sufficiently small values of $\Theta$, for example  $\Theta=0$, together with a corresponding indicator variable vector $w$, is feasible for~\eqref{prob:DV_MIP}. This leads us to consider parametric solutions of~\eqref{prob:BOUND_ROB_FORM} and in particular solving an equivalent a penalty-based formulation (similar to a Lagrangian relaxation of~\eqref{prob:BOUND_ROB_FORM}) that relaxes constraint~\eqref{constr:BUDGET}, excluding constant terms from the objective. In particular, denote the optimal solution set of this penalty-based problem variant %{(with $w_v=1$)} 
as   
\begin{equation}
\mathcal{S}(\beta)\equiv \argmax\{\ubar d-\beta\sum_{v\in H_K}y_v:\eqref{constr:bound_prev_constr}-\eqref{constr:YBND}\}   \label{prob:PEN_ROB_FORM}
\end{equation}

Of course, for every value of $\Theta$ in problem \eqref{prob:BOUND_ROB_FORM}, 
there exists a corresponding value of $\beta$ in problem \eqref{prob:PEN_ROB_FORM} such that the optimal values of $\ubar{d}$ for the two problems {coincide}. 
%Indeed, $\beta$ can be viewed as  
%an optimal dual variable value associated with constraint \eqref{constr:BUDGET} in problem \eqref{prob:BOUND_ROB_FORM}.
%For every $\beta\geq 0$, let $\mathcal{S}(\beta)$ be the optimal solution set of problem \eqref{prob:PEN_ROB_FORM} with penalty parameter $\beta$. Also
For every $\beta\geq 0$, the notations $s(\beta)\equiv(x(\beta),\ubar{d}(\beta),y(\beta))$ and $z(\beta)$ are used, respectively, to denote an optimal solution (an arbitrary element of the optimal solution set $\mathcal S(\beta$)) and the optimal objective value of  \eqref{prob:PEN_ROB_FORM} with penalty parameter $\beta$. 
The next 
lemma establishes that {$\ubar{d}(\beta)$} and  $\sum_{v\in H_K} y(\beta)_v$
are monotonically non-increasing in $\beta$ (see proof in Appendix \ref{appx:proof_lemma_monotome}).

\begin{lemma}\label{lemma:monotone}
For %some $\beta_u$ {[why is $\beta_u$ needed?]}, and 
every $0\leq \beta_1< \beta_2<\beta_u$, 
every pair of optimal solutions $(x,\underbar d,y)\in \mathcal S(\beta_1)$ and $(\tilde x,\tilde{\underbar d},\tilde y)\in \mathcal S(\beta_2)$ satisfies 
\begin{align*}
&\sum_{v\in H_K}\tilde y_v\leq 
\sum_{v\in H_K}y_v, &&  \tilde{\ubar{d}}\leq \ubar{d}, &&\text{ and } && \tilde{\ubar{d}}-\beta_2\sum_{v\in H_K}\tilde y_v\leq \ubar{d}-\beta_1\sum_{v\in H_K}y_v.\end{align*}
\end{lemma}

Lemma~\ref{lemma:monotone}, together with the fact that there exists a value of $\beta>0$ for which $\sum_{v\in H_K}y_v(\beta)=0$,  
implies that for every value of $\Theta\geq 0$, it is possible to find a minimal value of $\beta$ such that~\eqref{constr:BUDGET} is satisfied.  
%{[This sentence seems disconnected]} 
\noam{Motivating our method is both monotonicity of $\vectornorm{y(\beta)}_1=\sum_{v\in H_K}y(\beta)_v$, following Lemma~\ref{lemma:monotone}, and that} the one-norm is a convex approximation of 
the zero-``norm'' {($\vectornorm{y(\beta)}_0$)} \noam{, which is the number of nonzero components of the deviation variable vector $y$}.
%$\vectornorm{y}_0=\sum_{v\in H_K}w_v$ in~\eqref{DVCONSTRAINT} that 
\noam{The one-norm} is effectively used to promote sparsity in large-scale machine learning and compressed sensing applications (see, for example,~\cite{Donoho2003}).  We expect that solving \eqref{prob:PEN_ROB_FORM} with appropriate $\beta$ will approximate the problem with dose-volume constraints \eqref{DVCONSTRAINT} reasonably well.
Our goal is to find the minimal value of $\beta$ for which the optimal solution of \eqref{prob:PEN_ROB_FORM} satisfies \eqref{DVCONSTRAINT}. To achieve this , %we solve the 
LP \eqref{prob:PEN_ROB_FORM} is solved parametrically,  starting from a lower bound on $\beta$ and 
traversing optimal bases for the given subsets of constraints defining the reduced problem. This algorithm is presented as 
Algorithm~\ref{alg:PARAMETRICschem}.  

An additional difficulty in implementing our algorithm lies in the fact that, in practice, \eqref{prob:PEN_ROB_FORM} is often too large to solve with all constraints present, and the constraints are generated on the fly as in Algorithm~\ref{alg:RGAschem}. %in an ad-hoc manner
 Consequently, some of the constraints and variables may be absent from the formulation, and accordingly the reduced cost of these variables, as well as the full basis%, both of which are heavily used in standard sensitivity analysis, 
  are not readily available.
  {However, both the reduced cost and full basis are needed for determining the minimal increase in $\beta$ for which the current optimal basis will no longer be optimal, and consequently $\vectornorm{y}_0$ may change.}
  %This is while both may be needed by Simplex  pivoting, which in turn is needed in order to step to an adjacent basic feasible solution. 
  Our parametric algorithm takes this into account by possibly retracting and allowing the value of the parameter $\beta$ to decrease if 
the computed solution is not adjacent to the preceding basic feasible solution in the polyhedron that incorporates the newly generated constraints.  %{[what about variables?]}
%includes newly generated variables %or constraints not present in the basis computed in 
%the previous. 
%% NG - the previous argument is incorrect - may have generated a constraint whose slack is now in the basis so it is not active and basis may still be adjacent to previous one in the full polyhedron
%The algorithm is designed to use the partial basis and reduced-cost information while solving the large-scale formulation~\eqref{prob:PEN_ROB_FORM} using only subsets of active constraints and a subset of variables $y$. 
%This is shown in Algorithm~\ref{alg:PARAMETRICschem}, which is a sketch of our parametric algorithm that steps through basic optimal solutions of~\eqref{prob:PEN_ROB_FORM} parameterized by $\beta$ given a subset of constraints generated thus far. 
\begin{algorithm}
	\caption{Parametric penalty algorithm sketch\label{alg:PARAMETRICschem}}
	\LinesNumbered
	\SetKwInOut{Input}{Input}
	\SetKwInOut{Output}{Output}
	\SetKwInOut{kwInit}{Initialize}
	\SetKwRepeat{Do}{do}{while}
	\Input{Initial bounds $0\leq \beta_l < \beta_u$, BFS optimal for lower bound $\beta=\beta_l$, initial constraint subsets $(S,\hat H_1,\ldots,\hat H_K)$
		%where $\hat{H}_K=\Cols(\hat H_K, s,\tau_{\Infs,2}/2)$.
	}
	\While{$\vectornorm{y_{\hat H_K}}_0>\theta$ or current BFS is not adjacent to the previous one in the full polyhedron with constraints $(T,H_1,\ldots,H_K)$}
	{
		%Depending on whether 
  {If} the current BFS satisfies (or violates) $\vectornorm{y_{\hat H_K}}_0>\theta$ then minimally increase (respectively decrease) $\beta$ so that current BFS is no longer optimal\;
	%	Determine an adjacent BFS in direction of updated $\beta$ along Pareto frontier of partial polyhedron with constraints $S,\hat H_1,\ldots,\hat H_K$\;
		Solve restricted LP $(S,\hat H_1,\ldots,\hat H_K,\beta)$ to obtain optimal
		BFS with updated constraint sets $(S,\hat H_1,\ldots,\hat H_K)$\label{RCInvokeStepSchem}
	}
	\Output{$\beta$ and corresponding BFS.}
\end{algorithm}

A detailed version of this algorithm is given in Appendix~A.3%\ref{appx:paramalgo}
. The algorithm applies row and column generation in Step~\ref{RCInvokeStepSchem} in order to solve $\Phi(S,\hat H,\ldots,\hat H_K,\beta)$; the details of this procedure are presented as Algorithm~\ref{alg:RCGA} in Appendix~\ref{sec:RCGenAlg}. This algorithm is a subroutine for solving \eqref{prob:PEN_ROB_FORM} that is similar to Algorithm~\ref{alg:RGAschem} in that it solves the  problem for a subset of the constraints (rows), but it also generates columns corresponding to $y$ variables as needed (only those components of $y$ corresponding to constraints~\eqref{constr:SOFT_HK_BOUNDS} and~\eqref{constr:BOUND_DELTA} that have already been generated). 

Let $\beta^*=\min\{\beta:\exists (x,\ubar{d},y)\in \mathcal{S}(\beta), \norm{y}_0\leq \lfloor\alpha |H_K|\rfloor\}$; that is, $\beta^*$ is the minimal penalty parameter for which the optimal solution to \eqref{prob:PEN_ROB_FORM} satisfies
~\eqref{DVCONSTRAINT}. Bounds on $\beta^*$ can be established using Lagrange multipliers of constraint~\eqref{constr:BUDGET} for solutions that are optimal to~\eqref{prob:BOUND_ROB_FORM} for different values of $\Theta$, specifically, $\beta_l$ and $\beta_u$ 
for $\Theta =  \lfloor\alpha |H_K|\rfloor(\hat d_K -\bar d_K)+\epsilon$ (for some $\epsilon>0$) and $\Theta =  0$, respectively. 
These bounds are proven in Appendix~%\ref{appx:proof_lem_bounds}
B.2. The following proposition establishes the correctness of Algorithm~\ref{alg:PARAMETRICschem} given such bounds.

\begin{proposition}\label{prop:param_finite_steps}  Suppose that LP~\eqref{prob:PEN_ROB_FORM} has a unique optimal basis for all $\beta\in[\beta_l,\beta_u]$. Then, 
%Suppose that~\eqref{prob:PEN_ROB_FORM} is nondegenerate. 
Algorithm~\ref{alg:PARAMETRICschem} determines $\beta^*$ in a finite number of steps.
\end{proposition}
\noindent The proof of the proposition is deferred to Appendix~\ref{appx:proof_param_finite}. 

\section{Experiments}\label{sec:computational}

In our experiments, we focus on a particular case study based on real brain data of Patient 4 from the TCIA archive~\citep{Clark2013}, referred to as the Brain dataset in Table~\ref{table:datasets}.
Due to the very large number of constraints in~\eqref{prob:nominal}, and especially in~\eqref{prob:SR_robust}, for this dataset, we also use a smaller liver imaging dataset~\citep{Craft2014}
to evaluate the computational performance of our dynamic generation of OAR and homogeneity constraints in Algorithm~\ref{alg:RGA}. 
FMISO-PET images are not available for this smaller dataset, so {for this particular dataset,} the radiosensitivity data are created synthetically.
Oxygen levels tend to \noamrr{deccrease}, and accordingly radiosensitivity values tend to decrease, %accordingly hypoxia tends to be more prevalent, 
as we approach the center of the PTV. Thus, to create sensible radiosensitivity values, we first define vector $f$ %over PTV voxels, 
such that, for each PTV voxel $v\in T$, \noamrr{$f_v=(2\pi)^{-3/2}\det(\Sigma)^{-1/2}
%\exp(-(\loc_v-\overline{\loc})^\top\Sigma_\text{PTV}^{-1}(\loc_v-\overline{\loc })
e^{-\frac{1}{2}(\loc_v-\overline{\loc})^\top\Sigma^{-1}(\loc_v-\overline{\loc })}
$},
where $\loc_v$ is the vector of x-y-z coordinates of voxel $v$'s position,   $\overline{\loc}=\frac{1}{|T|}\sum_{v\in T}\loc_v$ is the mean voxel position, and the \noamrr{covariance is given by a} ($3\times 3$) \noamrr{scaled} covariance matrix of the  PTV voxel positions, 
 % is given by 
$\Sigma=\frac{\noamrr{50}}{|T|}\sum_{v\in T}(\loc_v-\overline{\loc})(\loc_v-\overline{\loc})^\top$. The radiosensitivity for voxel $v\in T$ is then set to be $\phi_v = 0.85+0.15\times \frac{f_{\max}-f_v}{f_{\max}-f_{\min}}$, where $f_{\max}=\max_{u\in T}\{f_u\}$ and $f_{\min}=\min_{u\in T}\{f_u\}$. This choice of $\phi_v$ implies 
that %the values of $\phi_v$ lie in  
it is contained in $[0.85,1]$, while %for voxels $v$ closer to the center of the PTV,
as voxel v’s position is closer to the center of
the PTV, the radiosensitivity, $\phi_v$,  approaches $0.85$.
%{using the following formula for $v\in T$: $\phi_v = 0.85+0.15\times \frac{\max_{u\in T}\{f_u\}-f_v}{\max_{u\in T}\{f_u\}-\min_{u\in T}\{f_u\}}$, 
%where {vector} {$f$} follows a multivariate normal density over the PTV voxels, with the mean set as the average voxel x-y-z coordinates, 
%and the covariance proportional to the covariance of the x-y-z coordinates (scaled by 0.2).}   
Details of both datasets are given in Table~\ref{table:datasets}.

\begin{table}[H]
  \caption{Imaging data  summary\label{table:datasets}}
    \centering
    \begin{tabular}{lcrlrr}
    Dataset&  PTV   & \multicolumn{4}{c}{OAR}\\
    \cline{3-6}
    &$\card{T}$& $k$& Name & $\card{H_k}$ & $\bar{d}_k$ bounds\tablefootnote{Based on \citep{ICRU2010}}\\
    %& \multicolumn{1}{c|}{$\card{T}$} & \multicolumn{1}{c}{$\sum_{k=1}^K\card{H_k}$}\\
    \hline 
    \hline
Brain (Patient 4) & 32,367 & 1 & Exterior & 1,870,179 & 100\\
&& 2& Brain & 1,360,182 & 62\\
&& 3 & Chiasm & 2,003 & 54 \\
\hline
Liver & 6,954 &1 & Liver & 77,657 & 60 \\
&&2 & Heart & 32,983 & 50\\
&&3 & Entrance & 146,462 & 60\\
    \hline
    \end{tabular}
\end{table}

\subsection{Computational Performance}\label{sec:compperform}
Our implementation uses the Julia programming language, version 1.6.2, the JuMP package~\citep{Lubin2015}, and the Gurobi solver, version 9.1.2. 
The experiments are run on 40 core Linux servers each equipped with Intel Xeon {E5-2650}, 
25 MB cache, 2.3 GHz base frequency CPUs, and 128 GB of RAM. \noam{Algorithm~\ref{alg:RGAschem} (appearing in greater detail as Algorithm~\ref{alg:RGA} in %Appendix~\ref{appx:algo}
	Appendix A.1)} is invoked for solving the  robust model~\eqref{prob:SR_robust} and the nominal model~\eqref{prob:nominal}, with parameter values $\tau_{\Infs,1}=10$, $\tau_{\Infs,2}=0$, $\tau_z=10^{-2}${, and $n_H=400$}. %{[$\tau_{\Infs,3}$ never mentioned before]} 
The homogeneity parameter is set to $\mu=1.15$ for Liver instances and {$\mu=1.1250$ for Brain instances for the nominal experiments, and $\mu=1.1875$, which obtains a similar homogeneity and biologically-adjusted dose in the nominal scenario.}   Table~\ref{table:runningtime} shows the running times achieved when solving the nominal model~\eqref{prob:nominal}, which involves only $\card T$ homogeneity constraints. 
The results show that dynamic constraint generation is not needed to solve the nominal model %(given the hardware specification), 
but with careful tuning of the constraint generation algorithm, it can improve on the running time required to solve the full model (also yielding significant savings in terms of memory consumption), especially for larger data instances such as 
Brain. 
Table~\ref{table:runningtime2} shows the running times when applying the robust model to both datasets. The parameters of the uncertainty set are set to $\delta=0.04$ and $\gamma=0.04$ for Liver, and {$\delta=0.08$ and $\gamma=0.04$} for Brain. %{[what $mu$?]} 
Here, it is evident that dynamic constraint generation is necessary in order to be able to solve both models when the number of homogeneity constraints is $O(\card{T}^2)$. {However, dynamically generating all violated constraints may result in excessive running times. In the example shown in Table~\ref{table:runningtime2}, dynamically generating all violated constraints in each iteration results in a running time of nearly four hours for the Brain instance. In  preliminary experiments, with a different choice of parameters than experimented with here, when generating all violated constraints, we have observed cases where the running times took more than 48 hours to solve, and cases memory consumption exceeded its limits.} 

\begin{table}[th]
\caption{\label{table:runningtime}Running time in seconds and number of constraints generated in solving the nominal model~\eqref{prob:nominal}. Here, constraint generation, if indicated, implies that all violated constraints are generated in each iteration.}
    \centering
        {%\resizebox{0.7\textwidth}{!}{
\begin{tabular}{l|r|r|r|r|r|}
    
      \textbf{Hom cons} & \textbf{OAR cons}& \multicolumn{2}{c|}{\textbf{Liver}} & \multicolumn{2}{c|}{\textbf{Brain}}\\ 
     \textbf{generation} & \textbf{initial - $n_0$} & \textbf{Time} & \textbf{\# generated} & \textbf{Time} & \textbf{\# generated}\\
    \hline
    \hline
     No & 2,000 & 126.3 & 19,908 & \textbf{586.7} & 70,737\\
     No & 8,000 &  \textbf{82.3} & 37,908 & 731.7 & 82,737 \\
      No & all & 453.4 & 271,010 & 5207.2 & 3,297,098\\
      Yes & 2,000 & 189.5 & 22,858  & 3,235.1 & 72,160\\
      Yes & 8,000 &    175.9 & 37,531 & 3,458.7 & 77,730\\
       Yes & all & 563.7 & 265,977 & 5,117.7 &  3,264,928\\
      \end{tabular}}
      \end{table}

   \begin{table}[H]
\caption{\label{table:runningtime2}Running time in seconds and number of constraints generated in solving~\eqref{prob:SR_robust} using Algorithm~\ref{alg:RGA}.  
%LIMIT indicates that the memory limit has been exceeded.
}
     {\centering
     {
      \begin{tabular}{r|r|r|r|r|r}      
           \textbf{Hom cons} & \textbf{OAR cons} & \multicolumn{2}{c|}{\textbf{Liver}} & \multicolumn{2}{c}{\textbf{Brain}}\\
           \textbf{per iter - $n_s$} & \textbf{initial - $n_0$}& \textbf{Time} & \textbf{\# generated}& \textbf{Time} &\textbf{\# generated}\\
    \hline
    \hline
       %2,000 & 4,000 & 285.6 & 31,350 & 2889.1 & 104,781\\
       2,000 & 2,000 & \textbf{3,718.6} & 92,890  & \textbf{4,940.4} & 82,160\\
       2,000 & all & 10,215.3 & 335,528 & 12,861.4 & 3,271,931\\%6,380.0 & 368,802\\
       8,000 & 2,000 & 3,857.8 & 130,343 &  5,255.7 & 106,261 \\
      8,000 & all & 10,033.0 & 385,506 & 12,029.5 & 3,286,449\\%3,280,731\\ 
    all & all & 8,133.5  & 408,747 & 14,356.1 & 3,424,511 \\      
    \end{tabular}}
    }
\end{table}

\subsection{Robust Model Evaluation}
%From this section on ward we focus on the Brain dataset (described above). 
Next, the spatially {bound} uncertainty set $(\USB)$, introduced in Section~\ref{sec:uncertainty}, will be constructed and customized for our patient data. Subsequently, the benefits of both the robust model~\eqref{prob:SR_robust} based on this uncertainty set,  and the dose-volume methodology of Section~\ref{sec:tuning_beta}, will be examined and discussed.

\subsubsection{Uncertainty Set Construction}

The method of constructing our uncertainty set is justified in Section~\ref{sec:uncertainty} using the data of brain cancer Patient 1 from the TCIA archive in~\cite{Clark2013}, illustrated in Figure~\ref{fig:Patient1VoxelRadioSensitive}.  The graphs in Figure~\ref{fig:Patient4VoxelRadioSensitive}, which are based on the Brain data of Table~\ref{table:datasets} (Patient 4 from the archive), demonstrate a similar pattern to the one that emerges from Figure~\ref{fig:Patient1VoxelRadioSensitive}. Both figures clearly illustrate that the radiosensitivity distance increases in the geometric voxel distance, up to a geometric distance of approximately ten. {To estimate the difference in radiosensitivity between voxel pairs as a function of the voxel distance $\Delta$, we adopt the following model with parameters $\alpha_0$, $\alpha_1$, $\alpha_2$, and $\gamma$ (initially 0),  
$$\Gamma(\Delta) = 
\begin{cases}\gamma+\alpha_0 + \alpha_1\Delta + \alpha_2 \log(\Delta)& \Delta \leq 10\\
\Gamma(10) & \text{otherwise.}
\end{cases}
$$ 
 The estimated $\Gamma$ is personalized to Patient 4's PET data. It is estimated through a {constrained} least squares fit of the radiosensitivity difference $p$-th percentile curve (see Figure~\ref{fig:Patient4VoxelRadioSensitive}), for voxel distances up to {$\Delta=10$. The least squares problem solved, is constrained in our case, so that the fitted curve upper bounds the 98th percentile graph}. %Specifically, in the current experiment, the 98th percentile is used.
 Having obtained the best fit curve, the constant $\alpha_0$ is then slightly increased so that the curve upper bounds the same $p$-th percentile for all $\Delta\geq 1$. The constant $\gamma\geq 0$ is used to 
adjust the model to be even more conservative.  
The final estimated parameter values are {$\alpha_0=0.0292761$, $\alpha_1=-0.0013514$, and $\alpha_2=0.0128265$}.
\begin{figure}
%\begin{subfigure}[t]{0.45\textwidth}
\includegraphics[scale=0.33]{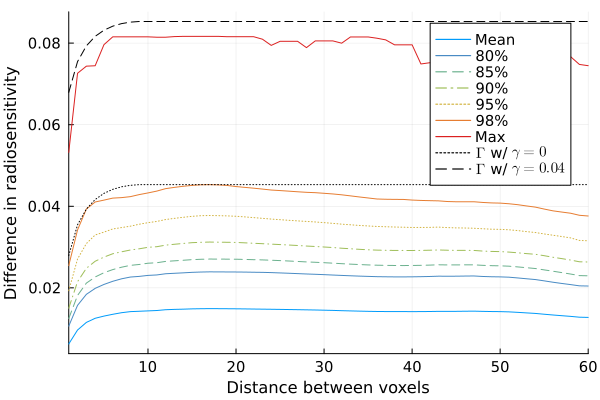}
%\caption{Radiosensitivity diff. vs. voxel dist.}\label{fig:Patient4_VoxDistVSRSDist}
%\end{subfigure}
%\hspace{30pt}
%\begin{subfigure}[t]{0.45\textwidth}
%\includegraphics[scale=0.33]{Paper/Patient4_Visit1_refpoint5_PhiDistHist.png}
%\caption{Histograms of radiosensitivity diff.}
%\end{subfigure}
\caption{
{Radiosensitivity difference vs. voxel distance}
%The relationship between radiosensitivity difference and voxel separation 
for PTV voxel pairs {of Patient 4 in the  TCIA archive} Brain dataset.
}\label{fig:Patient4VoxelRadioSensitive}
\end{figure}
Evidently, $\Gamma$ satisfies Assumption~\ref{ass:gamma_struct} for all $\gamma\geq 0$. Moreover, for $\gamma=0$,  by construction, 
 $\Gamma$ upper bounds the $98$th
 percentile graph, and for $\gamma=0.04$, $\Gamma$ 
 upper bounds the graph of the maximum, % ($100$th percentile) graph, 
 as seen in Figure~\ref{fig:Patient4VoxelRadioSensitive}(A).}

{The magnitude of the uncertainty set parameter $\delta$, which bounds the deviation of a voxel's radiosensitivity from its calculated value can be estimated by a sensitivity analysis of underlying theoretical equations in the uncertain constants; see for example the discussion in Appendix~C%\ref{appx:suv}
	.}

\subsubsection{Robust Formulation Experiments}
We now evaluate optimal solutions of {\emph{spatially robust}  model}~\eqref{prob:SR_robust} and compare them with optimal solutions of the nominal formulation~\eqref{prob:nominal} for the Brain dataset. 
We also compare with {the \emph{robust box} model that is }a robust model~\eqref{prob:general_robust} under a box uncertainty set ($\mathcal U_B$). {The solutions to the robust box model} can also be obtained by %setting  
solving~\eqref{prob:SR_robust} with $\ubar\phi{=\ubar\phi^0}$ and $\bar\phi{=\bar\phi^0}$ {by setting
$\gamma_{uv}$ in the $\USB$ uncertainty set} {to a large value (it suffices to set $\gamma_{uv}=1$)} for all $u,v\in T$. {Thus, the robust box case can be viewed as a special (simple) case of our suggested approach.}

Optimal solutions of~\eqref{prob:SR_robust} (with $\USB$) are computed for a range of uncertainty set parameter values, {$\delta=0.02,0.04,\ldots,0.14$ and $\gamma=0.01,\ldots,0.05$}, and a range of homogeneity parameter values, {$\mu=1.0625, 1.0750,\ldots,1.4000$}. Note that for some combinations of uncertainty set and homogeneity parameters, formulation  \eqref{prob:SR_robust} 
has $x=0$ as the only feasible solution; such solutions are omitted from the graphs. 
For the box uncertainty set ($\mathcal U_B$), {the solution were computed with the same values for $\delta$, however,} a larger range of homogeneity parameters, {$\mu=1.0625, 1.0750,\ldots,1.4750$}, is explored to be able to obtain nonzero solutions {for all $\delta$ values. } {Moreover, in order to obtain a more accurate value of $\mu$ below which there is no nonzero dose solution, for each $(\gamma,\delta)$ pair, the $\mu$ values between the first nonzero dose solution and the last zero dose solution were refined, first in increments of $0.0025$, then in incremenets of $0.0005$, and finaly in increments of $0.0001$.} {
To evaluate a plan involving beamlet decision variable vector $x^*$ 
under uncertainty set $\mathcal{U}$, which could be a singleton when considering  the nominal case,  empirical measures of worst-case dose and homogeneity are defined as follows
\begin{align*}
\hat{\ubar{d}} = \hat{\ubar{d}}(x^*,\mathcal{U}) = \min_{\hat\phi\in \mathcal{U}}\min_{v\in T}\hat\phi_vd_v(x^*) && \text{ and } &&
\hat \mu = \hat\mu(x^*,\mathcal{U}) = \max_{\hat\phi\in \mathcal{U}}\max_{u,v\in T}\frac{\hat \phi_ud_u(x^*)}{\hat \phi_vd_v(x^*)}; 
\end{align*} %Herein, 
$\hat{\ubar{d}}$ and $\hat \mu$ are used as abbreviations when $x^*$ and $\mathcal{U}$ are unambiguous in the given context.}

Before comparing the performance of solutions that are optimal to~\eqref{prob:SR_robust} under different scenarios, it should be noted that the optimal solution value of~\eqref{prob:SR_robust} decreases in the uncertainty set parameters $\delta$ and $\gamma$ (see Figure~\ref{fig:OPTVSDELTA} in Appendix~\ref{appx:numerical}). Note that combinations of {$\mu$}, $\delta$, and $\gamma$ resulting in a zero beamlet intensity configuration ($x={\mathbf 0}$) are omitted from all figures and are interpreted as ``infeasible'' plans. 

\begin{wrapfigure}[12]{R}{0.7\textwidth}
%\vskip-40pt
    \centering \includegraphics[scale=0.75]{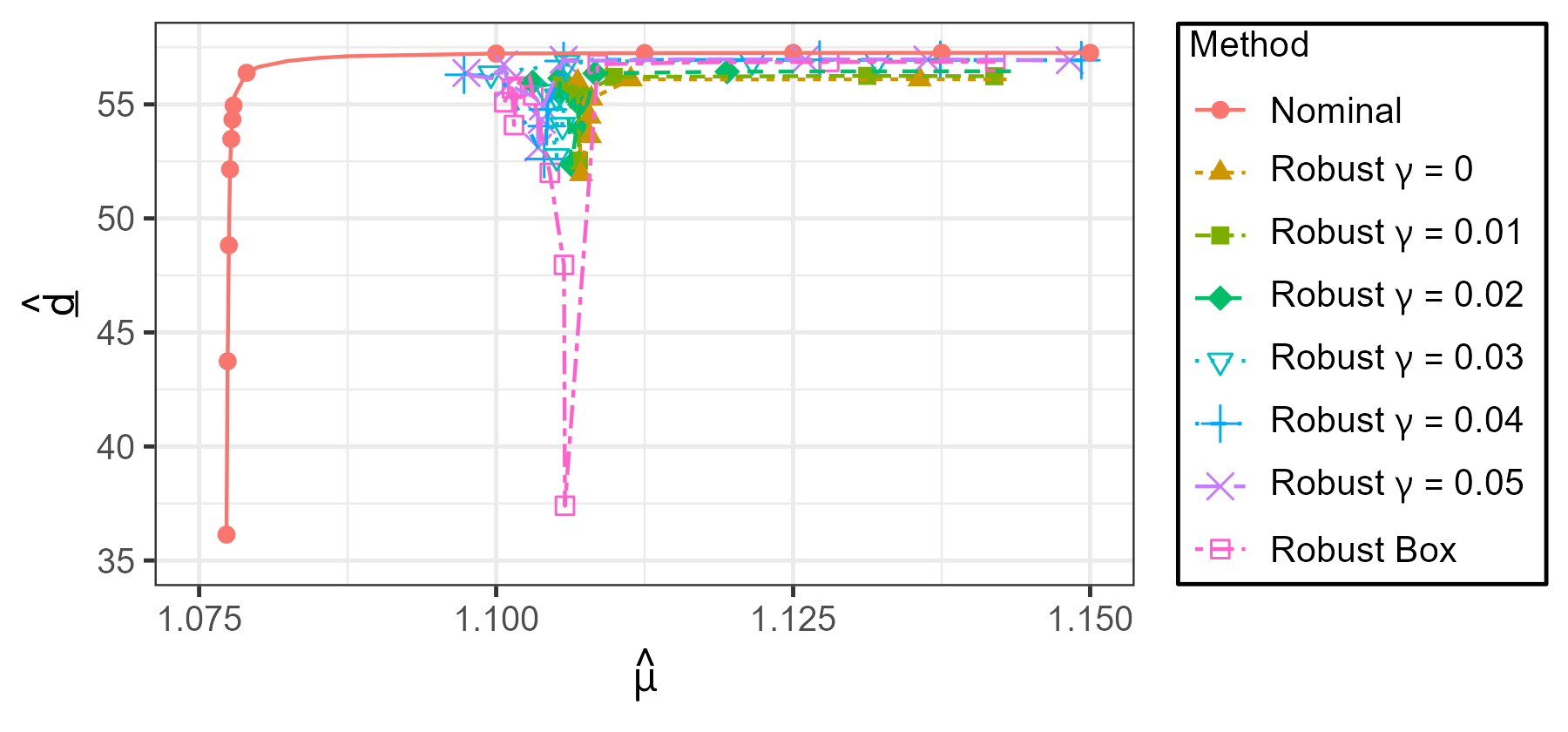}
    \caption{Nominal {biologically-adjusted} performance of  solutions optimal 
   to {model~\eqref{prob:SR_robust} for  $\delta=0.14$} and different values of $\gamma$, compared to the nominal solution. 
    \label{fig:nominal_delta0.14}}
\end{wrapfigure}

A ``cost-of-robustness'' can be measured by evaluating a robust solution in the nominal 
scenario where there is no uncertainty. Figures~\ref{fig:nominal_delta0.14} %\ref{fig:nominal_delta0.07}
and~\ref{fig:OPTandTESTvsDELTA} illustrate  this phenomenon. The nominal formulation~\eqref{prob:nominal} corresponds to  setting $\delta=0$ and $\gamma$ sufficiently large in $\USB$; for the particular dataset being considered (Patient 4), 
the nominal formulation is equivalent to a setting of $\delta=0$ and 
$\gamma\geq 0.04$ in formulation~\eqref{prob:SR_robust}). On the other hand, uncertainty sets with $\gamma\leq 0.03$ exclude the nominal scenario. Consequently, models with smaller $\gamma$ values tend to perform poorly, at least in terms of the minimum  biological dose, in the nominal scenario. Figure~\ref{fig:nominal_delta0.14} 
shows the performance in the nominal scenario, for {$\delta=0.14$}, with different values of $\gamma$ and $\mu$ in \eqref{prob:SR_robust} (a curve is presented for each value of $\gamma$, obtained by varying the values of $\mu$).  
In general,  
the ``cost of robustness''  
%becomes 
{appears} more evident as $\gamma$ decreases: {the minimal} $\hat\mu$ {that {is} attained with a nonzero dose solution}  
increases,  
%as
{and} $\hat{\ubar{d}}$ {somewhat} decreases. 
This can be explained by the fact that the nominal  
scenario is excluded from the uncertainty set when $\gamma\leq 0.03$, and the distance between this scenario and the uncertainty set increases as $\gamma$ decreases. 
The figure also shows that the robust box solution and the robust spatially {bound} solution with 
$\gamma=0.04$ exhibit fairly 
similar performance in the nominal scenario {but the robust box appears less stable and demonstrates a notable drop in dose for biological homogeneity values slightly larger than 1.1. Note that a lack of stability stems from the fact that the measured $\hat\mu$ and $\hat{\ubar d}$ values typically differ from the values of $\mu$ and $\hat{\ubar d}$, respectively, that are used in optimizing~\eqref{prob:SR_robust}. Consequently, even slight changes in the parameter $\mu$ may in some cases induce a different optimal solution that is associated with significantly different values of the measured $\hat \mu$ and $\hat{\ubar d}$.}
{In this case, when increasing $\mu$ from $1.3603$ to $1.3607$, the nominal performance of the robust box model solution improves with respect to both homogeneity and dose, from $\hat{\mu}=1.105$ and $\hat{\ubar{d}}=37.4$ to $\hat{\ubar{\mu}}=1.1$ and $\hat{\ubar{d}}=55.1$.}
%However, as $\delta$ increases, the spatially bounded solution with $\gamma=0.04$ becomes slightly superior (see Figure~\ref{fig:nominal_deltaall} in Appendix~\ref{appx:numerical}, which illustrates the results for the full set of $\delta$ values).

Figure~\ref{fig:OPTandTESTvsDELTA} shows nominal performance curves for $\gamma=0.02, 0.04, 1$,   
$\delta=0.02,0.04,\ldots,0.1{4}$, and a range of $\mu$ values (each curve is obtained by varying the value of $\mu$). For each of the $\gamma$ values,  the ``cost of robustness'' becomes more evident as $\delta$ increases: the curves for larger $\delta$ values appear to be {nearly} dominated in the $\hat \mu-\hat{\ubar{d}}$ space by the curves obtained for smaller $\delta$ values. %{For small values of $\gamma$, and in particular $\gamma=0.02$, }
%r{the uncertainty set does not include the nominal scenario, and therefore larger deviations from the optimal nominal dose (``cost of robustness''). }
%as expected the ``cost of robustness'' appears to be larger as the spatially robust curves appear lower than the nominal curve.  
{%However, 
The robust box results appear much less stable than the corresponding  %r{spatially robust [we never use this term before maybe we should avoid it]} 
{SR model} 
{results for $\gamma\in\{0.02,0.04\}$}, especially for values of $\hat\mu\leq 1.11$,} 
%Interestingly, when $\gamma=1$ {(i.e., the box model)}, a worst-case dose of $\hat{\ubar{d}}=43.3$ Gy (with $\hat \mu\approx1.105$) is obtained for  $\delta=0.1$, which is $13.5$ Gy lower than the optimal nominal $\ubar d$ for that homogeneity level. 
%, while 
%{When} $\gamma=0.04$, $\hat{\ubar d}= 55.3$ Gy (with $\hat\mu\approx1.106$) is obtained for $\delta=0.1$, which is only $1.5$ Gy lower than the corresponding nominal solution (with $\mu=\hat\mu\approx 1.106$). 
%{For $\gamma=0.04$ it can be noted that graphs appear more stable with less variability than the corresponding ones of the robust box. 
{with doses 25Gy lower than the optimal nominal solution compared to at most 5Gy lower for $\gamma=0.02$ and $0.04$.}
%Although, for some $\hat\mu$ values in $(1.10,1.11)$ the biologically-adjusted dose can decrease as much as 5Gy compared with the nominal.}
%Finally, it appears from the figure 
%{that smaller values of $\gamma$ such as $\gamma=0.02$ seem to} improve on the best homogeneity ($\hat\mu$) values relative to the robust box model ($\gamma=1$) for %some 
%{most} values of $\delta$. 
Recall that solutions of~\eqref{prob:SR_robust} with $\gamma=0.02$ are not designed to protect against the nominal realization (which is excluded from the uncertainty sets
with %such a 
{$\gamma < 0.04$}). % for the data of Patient 4). 
%Still, it can be observed that 
Hence, the solutions obtained in the case where $\gamma=0.02$ %are  even more stable, 
%having a 
{have a biologically-}adjusted dose of  {approximately $0.81$ Gy less than the corresponding optimal nominal dose for all }{$\hat\mu > 1.11$. When increasing $\gamma$ to $0.04$, the uncertainty set captures the nominal solution and so the dose obtained for $\hat{\mu}\geq 1.11$ become equal to those of the nominal model, albeit the minimal $\hat{\mu}$ for which there was a nonzero dose is higher than the nominal model, even for small values of $\delta$. This trend persists when moving to the robust box model, which results in almost optimal nominal performance for $\delta\leq 0.04$.} %Although, the spatially robust solutions appear to obtain a biologically-adjusted dose that is reasonably close (within 5 Gy) to the maximum over all values of %$\mu$, for 
%$\hat\mu\geq 1.09$, the box robust model may obtain a solution that has a biologically-adjusted dose that is as much as 20Gy less than the nominal model for some values of $\hat\mu$. } 
%This is while the nominal model determines solutions having a biologically-adjusted dose within 1Gy with $\hat\mu \geq 1.078$ implying.} %, and  homogeneity of at least $\hat \mu=1.114$, which is $2.7\%$ more than the 
%homogeneity of every (nonzero) optimal nominal solution.]}

\begin{figure}[t]
    %\centering
     \hskip-35pt\begin{subfigure}[t]{0.33\textwidth} 
     \hskip-8pt\includegraphics[scale=0.8]{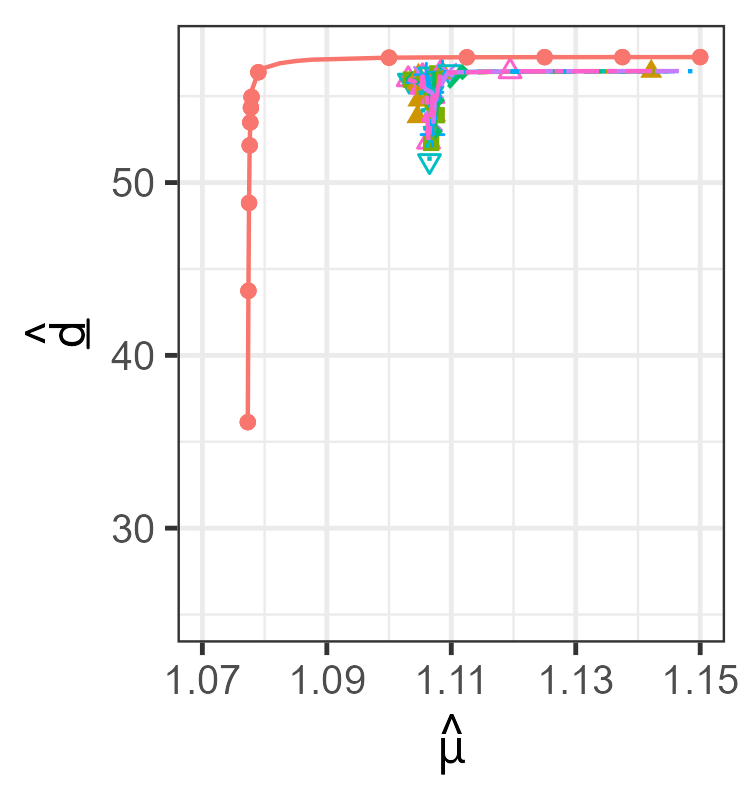}
    \vskip-10pt\caption{$\hat{\ubar{d}}$ vs. $\hat\mu$, $\gamma=0.02$}\label{fig:OPTandTESTvsDELTA_A}
    \end{subfigure}\hskip-20pt
    \begin{subfigure}[t]{0.33\textwidth}
    \includegraphics[scale=0.8]{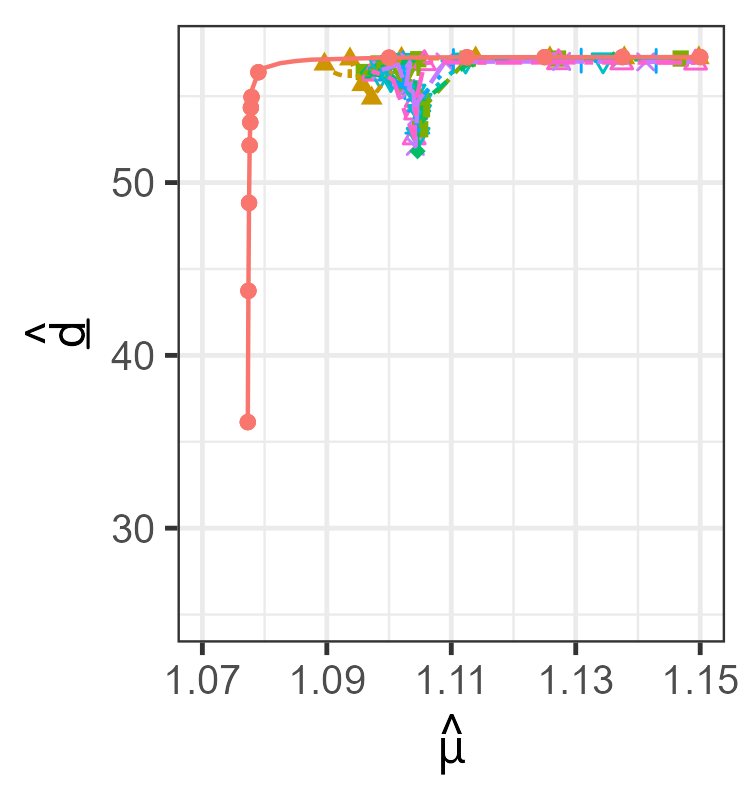}
   \vskip-10pt\caption{$\hat{\ubar{d}}$ vs. $\hat\mu$, $\gamma=0.04$}\label{fig:OPTandTESTvsDELTA_B}
    \end{subfigure}
    \hskip-15pt\begin{subfigure}[t]{0.34\textwidth}\includegraphics[scale=0.8]{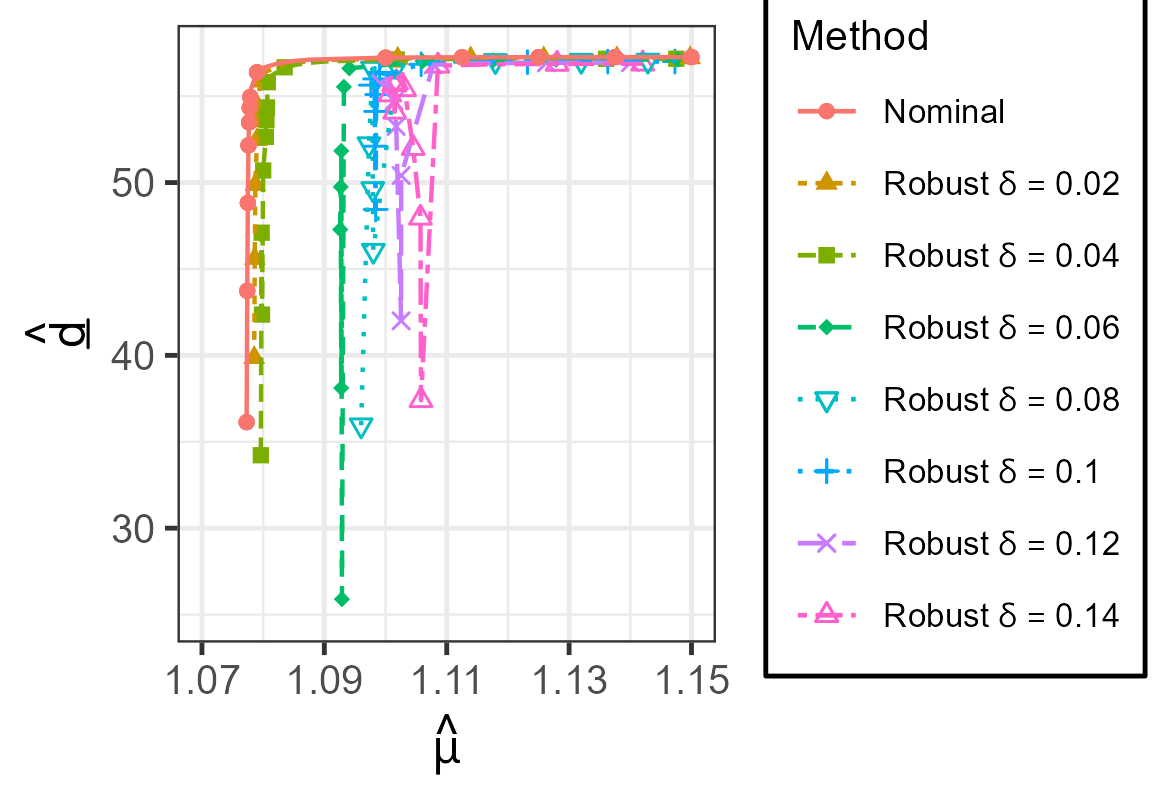}
   \vskip-10pt\caption{$\hat{\ubar{d}}$ vs. $\hat\mu$, robust box  ($\gamma=1$)}\label{fig:OPTandTESTvsDELTA_C}
    \end{subfigure}
   \caption{Nominal {biologically-adjusted} performance of solutions optimal
   to {models}~\eqref{prob:nominal} and~\eqref{prob:SR_robust}. ({for $\gamma=0.02$}, $\delta=0$ is infeasible.) 
   }
   \label{fig:OPTandTESTvsDELTA}
\end{figure}

%\begin{wrapfigure}[10]{r}{0.55\textwidth}
%    \centering
%    \vskip-5pt
  %  \includegraphics[scale=0.6]{Paper/Patient4_Visit1_physical_performance_delta0.14.png}
  %      \vskip-10pt
 %   \caption{Physical performance of solutions optimal
 %  to~\eqref{prob:SR_robust} with {$\delta=0.14$} and a range of $\gamma$ values, compared to the nominal solution.}\label{fig:physical_delta0.14}
%\end{wrapfigure}

Next, we consider the physical dose performance of the obtained solutions. Although physical performance measures do not account for the variable radiosensitivity, and planning based on such measures may not utilize the information that has been recently made available by advanced biomarkers such as FMISO, they are still important as the ongoing  { and }current basis for treatment guidelines and clinical practice. Keeping in mind that physical performance is not explicitly maximized in the optimization formulations~\eqref{prob:nominal} or~\eqref{prob:SR_robust}, it may be interesting to examine how the resulting planning solutions fare with respect to the physical performance measures. {Figure~\ref{fig:physical_performance_delta0.14}} compares solutions of \eqref{prob:SR_robust} with {$\delta=0.14$} and different values of $\gamma$, the robust box solution (also with {$\delta=0.14$}), and the nominal solution. {Figure~\ref{fig:physical_delta0.14} depicts the physical homogeneity vs. minimal physical dose of the different models, and Figure~\ref{fig:EUD_delta0.14} shows the EUD (equivalent uniform  dose). The EUD measure was introduced as a single number representing the dose throughout the PTV \shimritr{(\citealp{niemierko1997reporting} and \citealp{choi2002generalized})}, and is given by the formula $EUD=\left(\frac{1}{|T|} \sum_{v\in T} d_v^\alpha\right)^{\frac{1}{\alpha}},$
where, as in \cite{Nohadani2017} we use $\alpha=-10$.}
Surprisingly, both the optimal nominal solution and the robust box solution 
exhibit poor performance compared with the solutions of %{[we never call it that]}
{SR model} \eqref{prob:SR_robust} with respect to both physical homogeneity and physical minimal PTV dose{, for physical $\mu$ values in the interval $(1.07,1.10)$}. r{Specifically, the solutions of the nominal and robust box models are less stable and, with both physical dose and EUD up to 20Gy lower than for $gamma\leq 0.05$, and a minimal physical homogeneity value of around $1.07$  compares to $1.02$ for $\gamma\leq 0.05$.}
Similar behavior can be observed for other values of $\delta$ tested; see Figure~\ref{fig:physical_deltaall} in  Appendix~\ref{appx:numerical}.  
 A possible explanation for this behavior is that under both the nominal and the box uncertainty plans, the dose is allocated in a way that is focused mainly on the PTV voxels with the lowest radiosensitivities (following the max-min formulation). The spatially {bound} uncertainty may smooth out the PTV radiosensitivity uncertainty, resulting in a worst case that considers a greater number of voxels, including voxels for which the radiosensitivity is nominally higher, thus increasing the minimal physical dose.

\begin{figure}[t]
    %\centering
    \hspace{-35pt}
\begin{subfigure}[t]{0.44\textwidth}
  \includegraphics[scale=0.65]{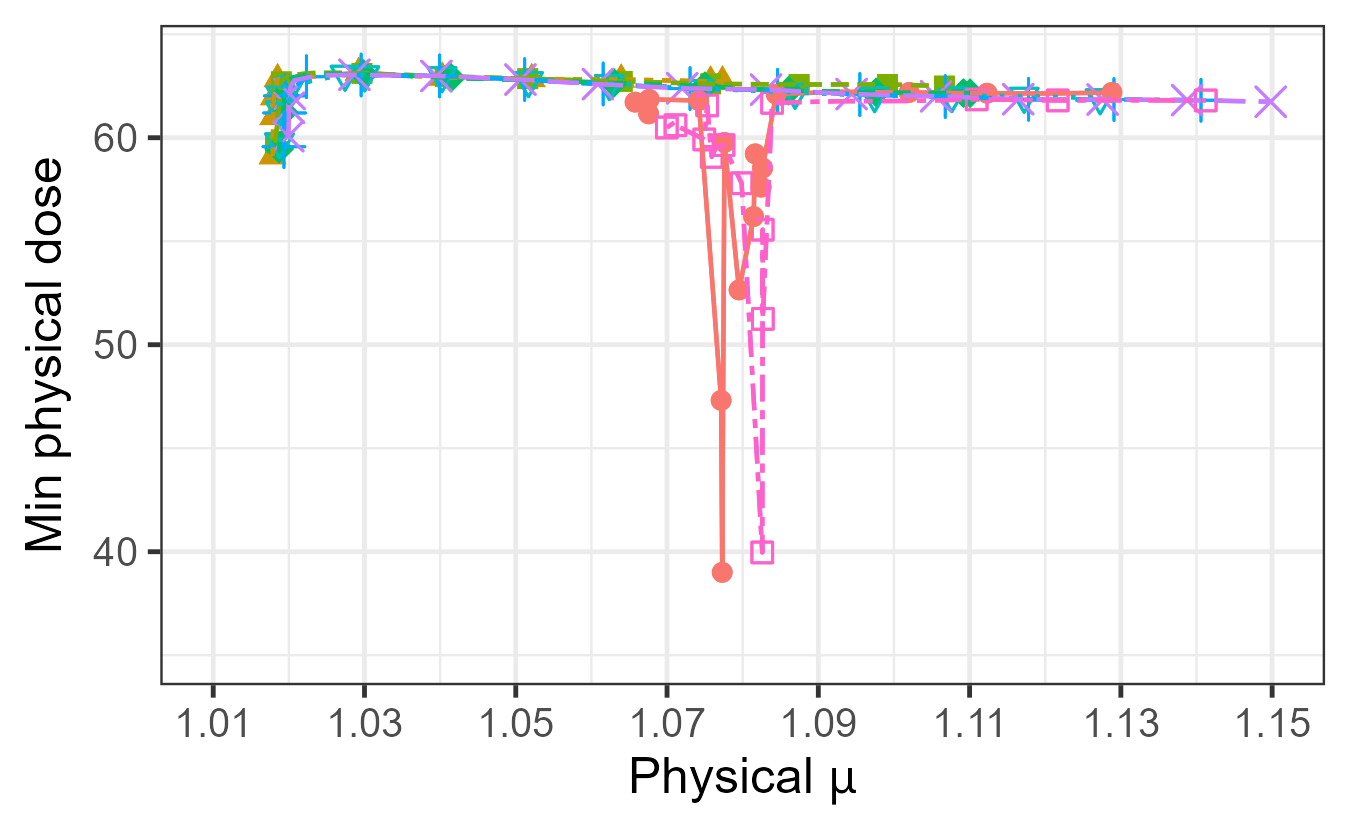}
  \caption{Minimum physical dose}%Physical performance of solutions optimal
  % to~\eqref{prob:SR_robust} with {$\delta=0.14$} and a range of $\gamma$ values, compared to the nominal solution.}
  \label{fig:physical_delta0.14}
    \end{subfigure}
\begin{subfigure}[t]{0.5\textwidth}
\includegraphics[scale=0.65]{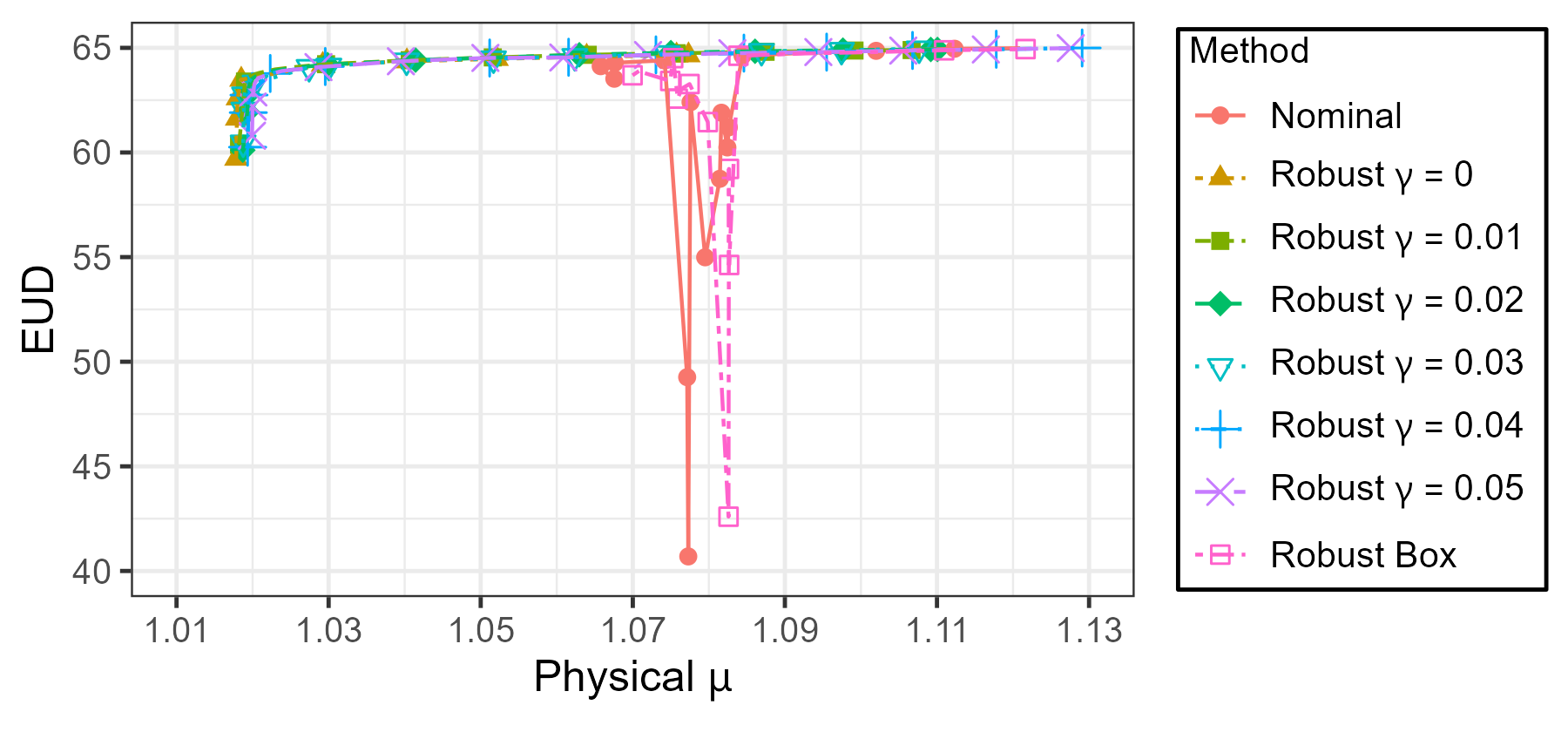}
      %  \vskip-10pt
    \caption{EUD}
    %performance of solutions optimal
   %to~\eqref{prob:SR_robust} with {$\delta=0.14$} and a range of $\gamma$ values, compared to the nominal solution.}
   \label{fig:EUD_delta0.14}
   \end{subfigure}
   \caption{Physical performance of solutions optimal
   to~\eqref{prob:SR_robust} with {$\delta=0.14$} and a range of $\gamma$ values, compared to the nominal solution.\label{fig:physical_performance_delta0.14}}
\end{figure}

\begin{figure}[b]
    \centering
     \begin{subfigure}[t]{0.24\textwidth}
    \centering\includegraphics[width=\textwidth]{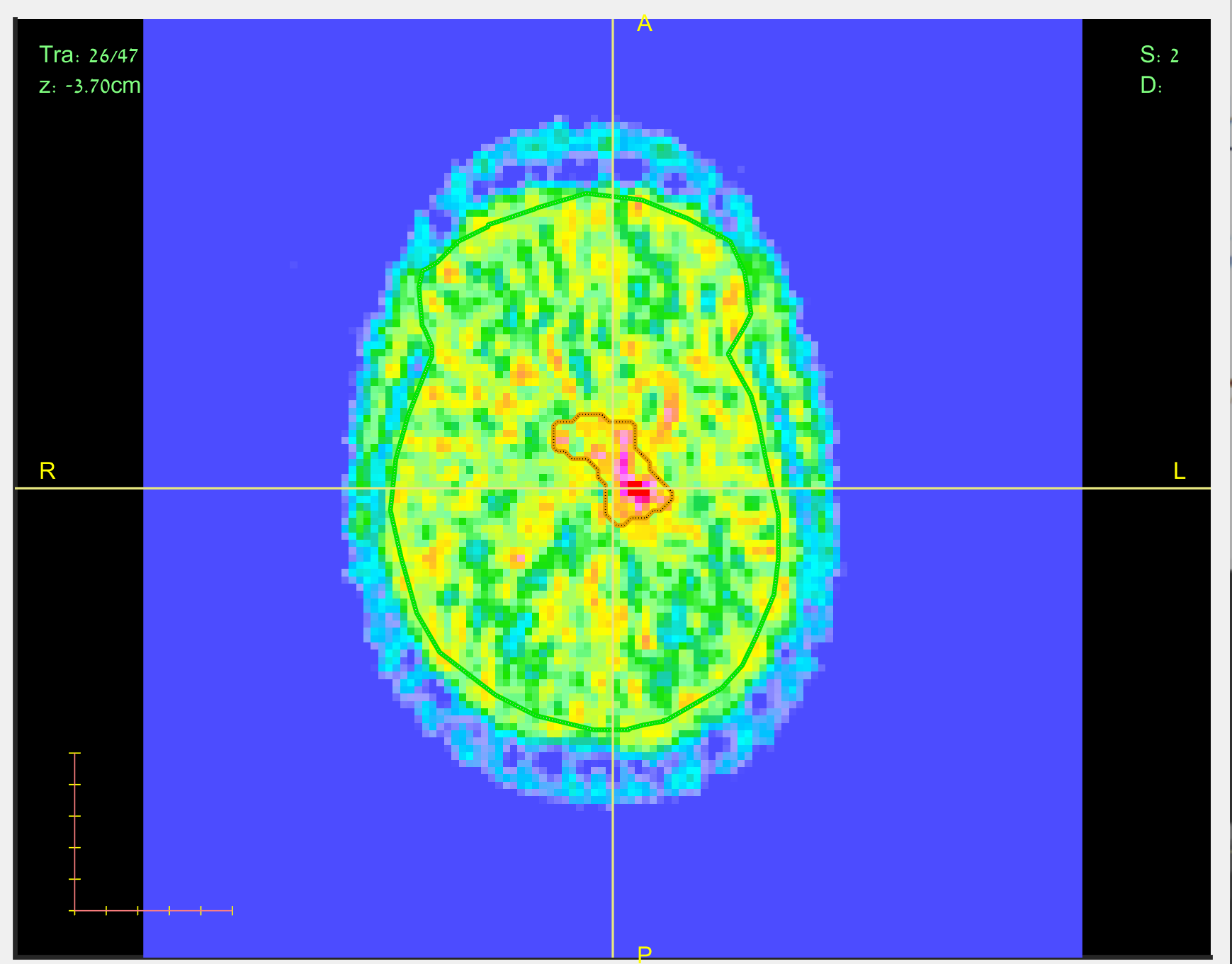}
    \caption{FMISO PET\label{fig:FMISO}}
    \end{subfigure}
    \begin{subfigure}[t]{0.24\textwidth}
    \centering\includegraphics[width=\textwidth]{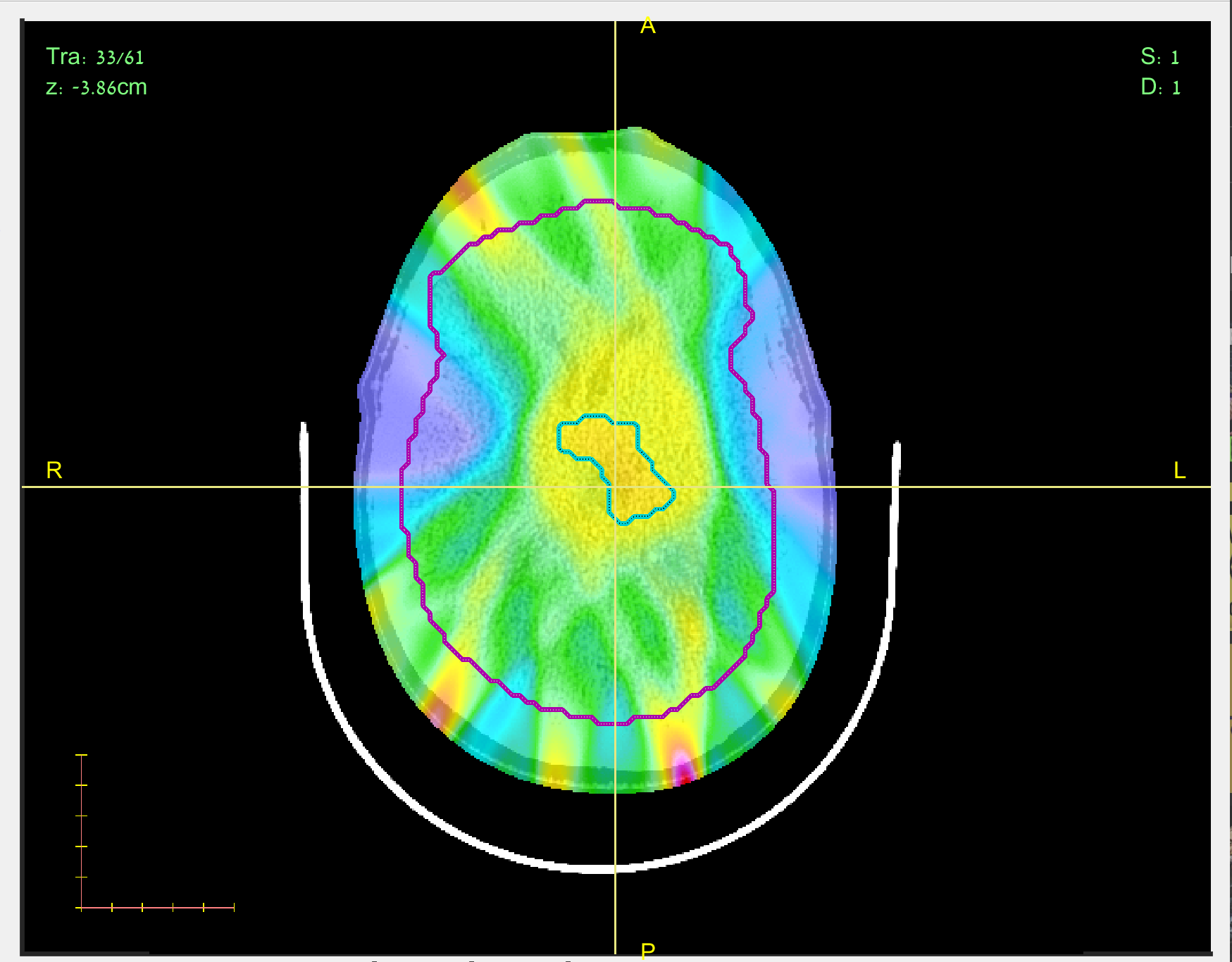}
   \caption{Nominal%Solution of~\eqref{prob:nominal}
   \label{fig:nominal_dose}}
  \end{subfigure} \hskip-4pt\begin{subfigure}[t]{0.24\textwidth}
    \centering\includegraphics[width=\textwidth]{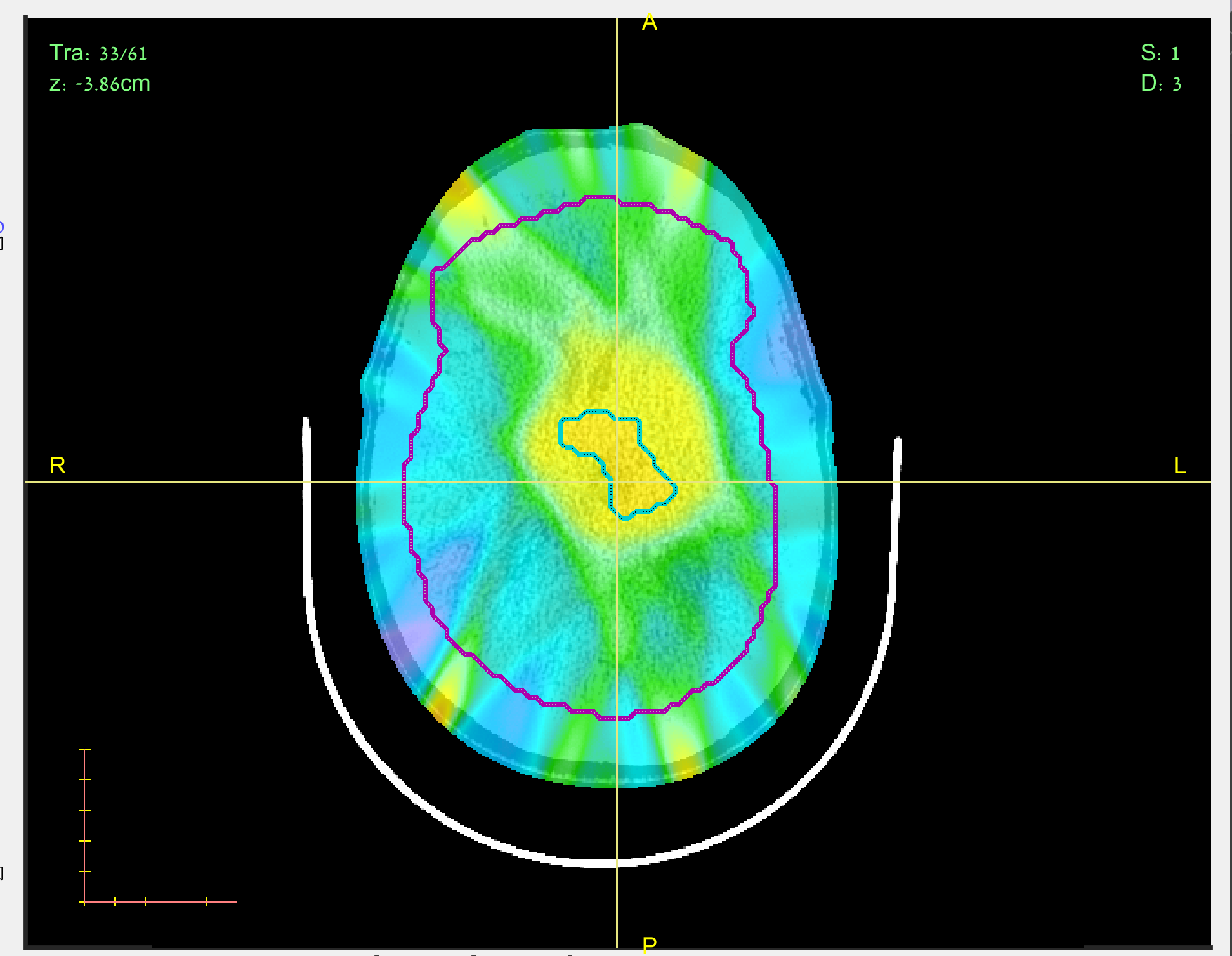}
   \caption{Spatial %Solution of~\eqref{prob:SR_robust}
   }\label{fig:robust_spatial_dose}
  \end{subfigure}%\hspace{5pt}
    \begin{subfigure}[t]{0.24\textwidth}
    \centering\includegraphics[width=\textwidth]{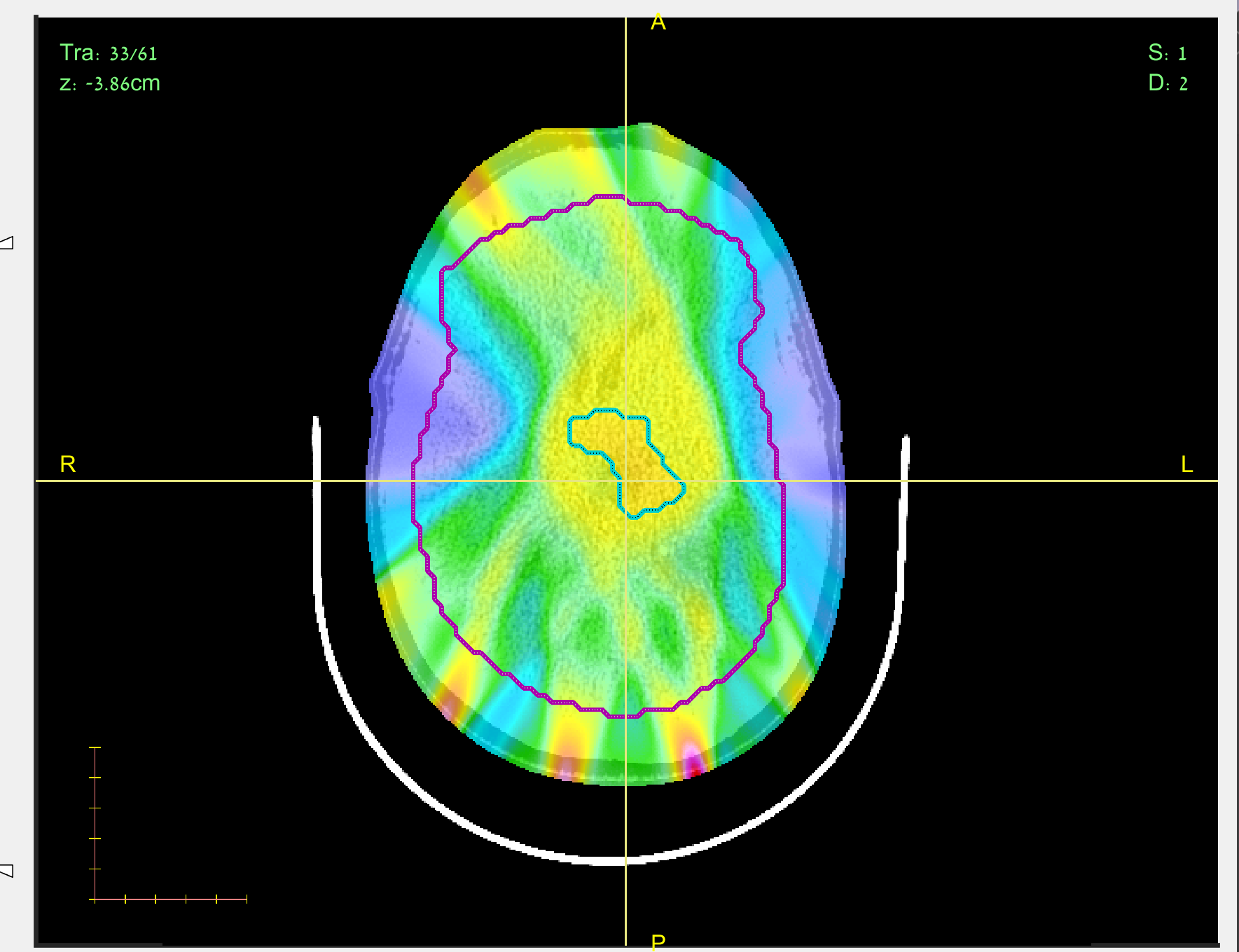}
   \caption{Box %uncertainty}
   }\label{fig:robust_box_dose}
  \end{subfigure}
   \begin{subfigure}[t]{0.02\textwidth}
    \centering\includegraphics[width=1.14\textwidth]{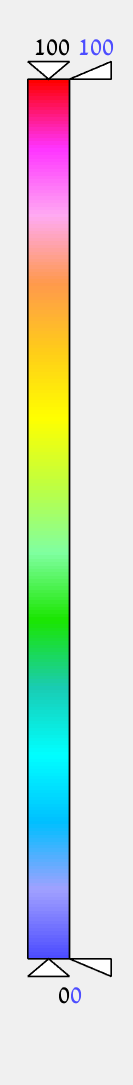}
  \end{subfigure}

%   \begin{subfigure}[t]{0.49\textwidth}
 %   \centering\includegraphics[width=0.96\textwidth]{Paper/Patient4_Visit1_PTVDoseHist.png}
 %  \caption{PTV dose-volume \label{fig:ptv_dose_hist}}
 %  \end{subfigure}
   \caption{(\subref{fig:FMISO}) -- PET, pink and red levels  indicating low oxygen. (\subref{fig:nominal_dose})-(\subref{fig:robust_box_dose}) -- physical dose. 
  %and for the PTV (\subref{fig:ptv_dose_hist}). 
   }
   \label{fig:physical_dose_cerr}  %\vskip-5pt
\end{figure}

\begin{wrapfigure}{R}{0.65\textwidth}
%\vskip-10pt
%   \begin{subfigure}[t]{0.49\textwidth}
    \centering\includegraphics[width=0.95\textwidth]{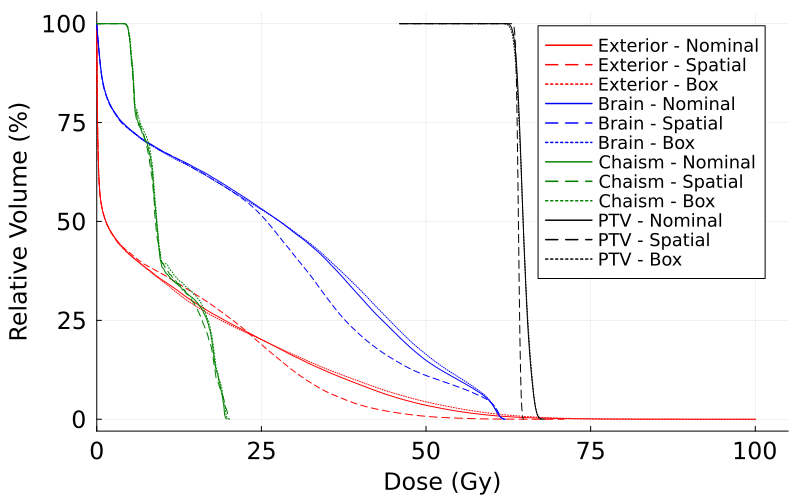}
   \caption{ Comulative dose-volume histograms.\label{fig:cdf_OAR}}
 %  \end{subfigure}
\end{wrapfigure}

Figures~\ref{fig:physical_dose_cerr}-\ref{fig:cdf_OAR} display the physical dose 
obtained by the nominal solution, the robust box solution with {$\delta=0.14$}, and the robust spatial solution of~\eqref{prob:SR_robust} with {$(\delta,\gamma)=(0.14,0.04)$}. For each model, the value of the homogeneity parameter $\mu$  is chosen {so that it yields the highest minimal physical dose while maintaining physical homogeneity less than 1.1 (which is a standard clinical requirement).} %is the one resulting in the best (lowest) physical homogeneity. %where 
All three dose images in Figure~{\ref{fig:nominal_dose}}-\ref{fig:robust_box_dose} are aligned with the same PET/CT FMISO scan slice~\ref{fig:FMISO}. %For clarity, only doses exceeding 40 Gy are depicted in these images. 
%The robust box solution applies the lowest doses to the PTV of all three models (doses ranging from 46-50 Gy, range of about 61-66 Gy for the nominal solution and 63-65 Gy for the robust spatial solution). 
{The robust box and nominal solutions apply wider physical dose ranges to the PTV, 61.5-67.5 Gy for the robust box and 62-68 Gy for the nominal, compared to the robust spatial solution, with a range of 63-65 Gy.}
Furthermore, the nominal and robust box solutions create dose ``hot spots" of {approximately 100} Gy in the perimeter of the exterior OAR, 
 while the %solution of the 
 robust spatial model solution  
applies a maximal physical dose of less than {71} Gy {to} this OAR. 
%to the same this OAR. 
This %is due to the fact that 
%is
{may be due to the fact that} the nominal and robust box solutions use fewer beams, with some of the beams set to a higher intensity, than the robust spatial solution. Thus, while the robust spatial model retains approximately the same adjusted dose as the nominal solution, 
it %is able to adjust to 
accounts for the variability and uncertainty in the radiosensitivities. %, and 
It also appears to avoid
excessively high OAR doses. % in the OARs. 

\begin{figure}[!b]
    %\hskip-10pt
    %\centering
    \hskip-45pt     \begin{subfigure}[t]{0.45\textwidth}
    % \hskip-20pt
    \includegraphics[scale=0.75]{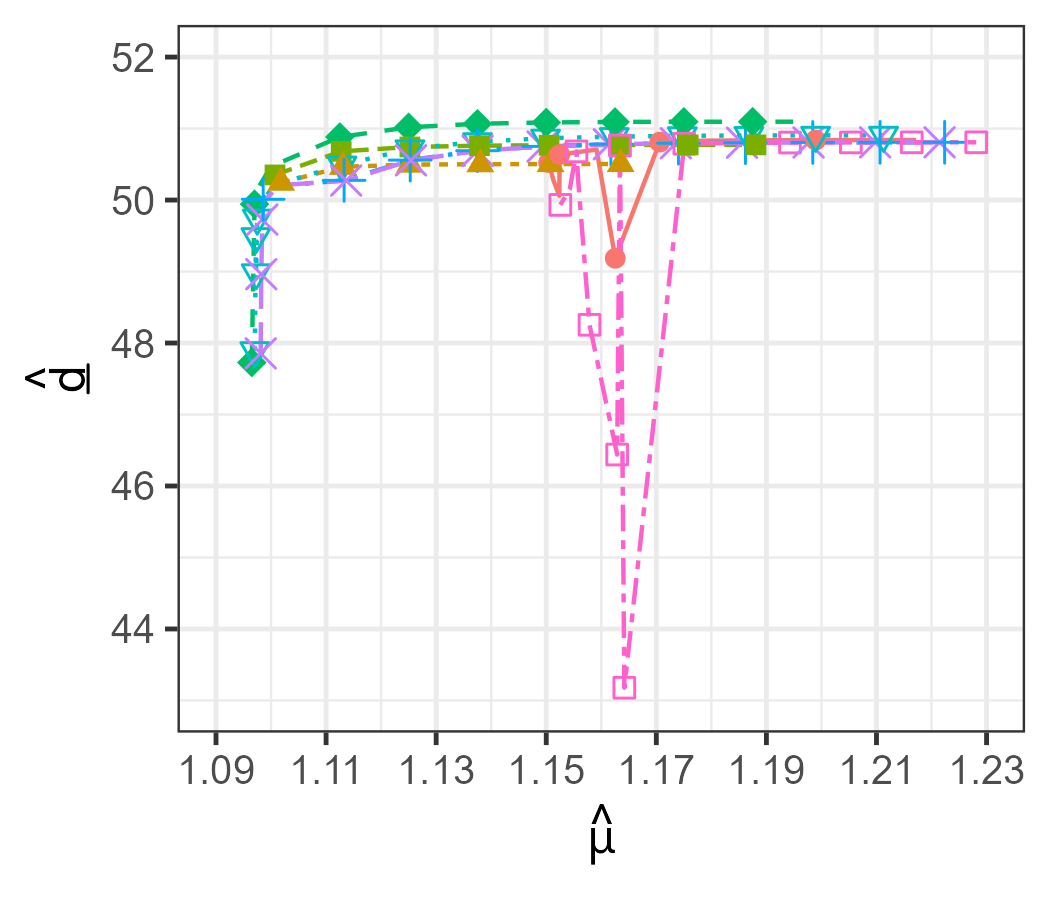}
    \vskip-10pt\caption{``True model" with $\gamma=0.02$}\label{fig:preformance_gamma0.02_delta0.1}
    \end{subfigure}\hskip-15pt
    \begin{subfigure}[t]{0.5\textwidth}
    \includegraphics[scale=0.75]{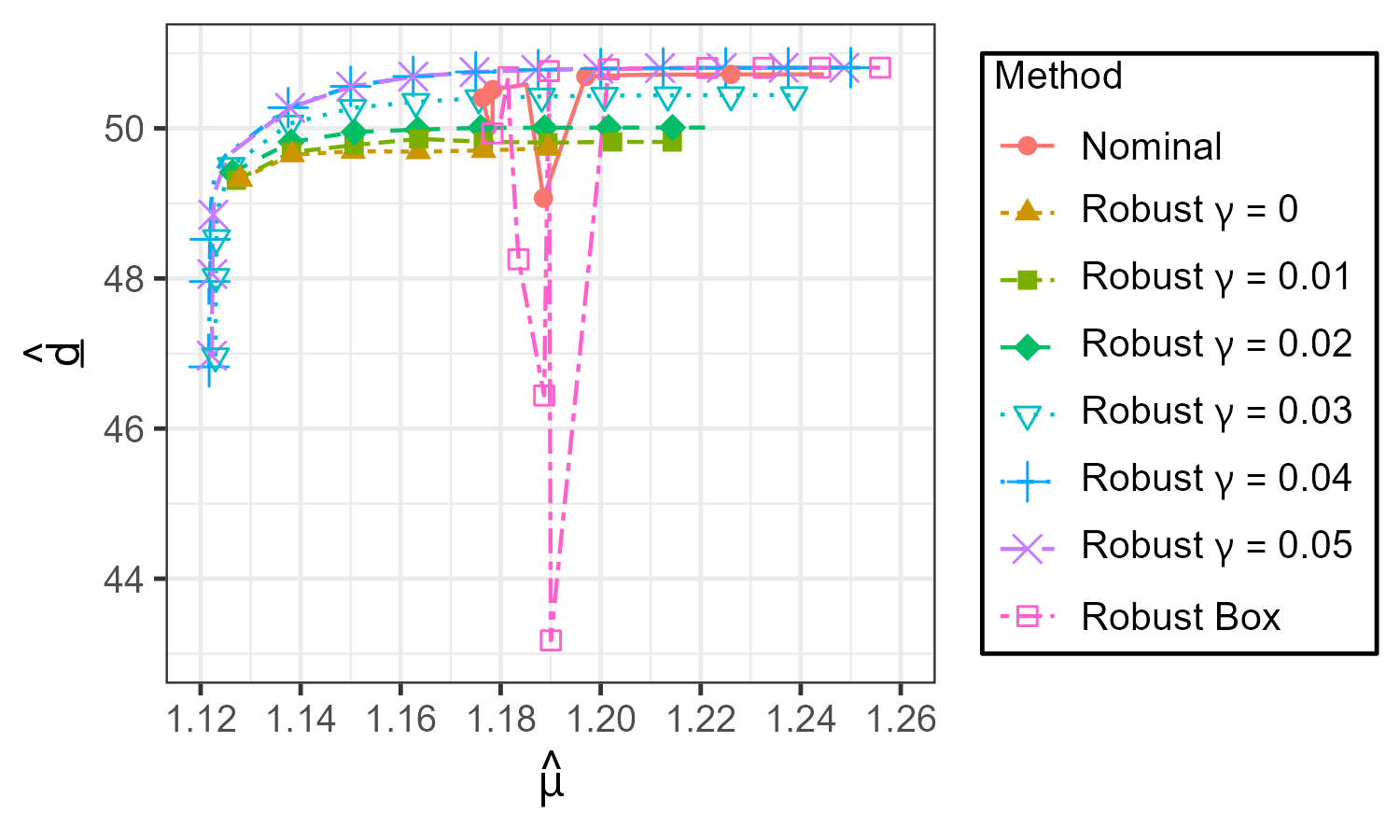}
   \vskip-10pt
   \caption{%$\hat{\ubar{d}}$ vs. $\hat\mu$, 
   ``True model" with $\gamma=0.04$\label{fig:preformance_gamma0.04_delta0.1}}
  \end{subfigure}
  
   \vskip-5pt
   \caption[]{Worst-case {biologically-adjusted} performance ($\hat{\mu}$, $\hat{\ubar{d}}$) of solutions optimal to {model}~\eqref{prob:SR_robust} with misspecified $\gamma$, under the assumption that the ``true" $\USB$ has {$\delta=0.1$} and $\gamma$ as indicated. The performance of the nominal solution is also shown for comparison.}
   \label{fig:performance_gamma_delta0.1}
\end{figure}

Finally, to evaluate the cost of misspecifying the uncertainty set parameters,  Figures~\ref{fig:performance_gamma_delta0.1} and~\ref{fig:performance_diffdelta} 
%\ref{fig:preformance_delta0.14_gamma0.04} 
depict the worst-case performance of the various solutions under the assumption that the correct uncertainty models {differ with respect to the parameter values assumed in  formulation~\eqref{prob:SR_robust}. 
In Figure~\ref{fig:performance_gamma_delta0.1}} the robust solutions use the ``true value'' of {$\delta=0.1$}, so the misspecification is only in terms of $\gamma$. {Specifically,} $\gamma=0.02$ and   $\gamma=0.04$ in Figure~\ref{fig:preformance_gamma0.02_delta0.1} and Figure~%\ref{fig:preformance_delta0.14_gamma0.04}
\ref{fig:preformance_gamma0.04_delta0.1}, respectively. %For all 
Note that while the worst-case performance deteriorates as the value of $\gamma$ used to derive the solution deviates from the true value, the %decline in 
dose declines by %does not exceed 
at most (and is generally much less than) {$1.1$} Gy. Importantly, the spatially robust models may achieve homogeneity values that are close to the homogeneity of the correct model, %providing more 
improving {the homogeneity} by more than $0.05$ relative to the nominal and robust box solutions. This is evident for most values of $\delta$ (see Figures \ref{fig:preformance_gamma0.02} and \ref{fig:preformance_gamma0.04} in Appendix~\ref{appx:numerical}).  {Moreover, although the robust box and nominal model appear to attain a similar biologically-adjusted dose as the true model for high values of $\hat \mu$, the performance of these models appears much less stable for desirable (lower) values of $\hat\mu$.  Notably, the robust box obtains doses significantly lower than all models considered including the nominal one for some values of $\hat\mu$.} Hence, even when $\gamma$ is misspecified in~\eqref{prob:SR_robust}, it may still yield nearly the same homogeneity, $\hat\mu$, perhaps at a slight cost to $\hat{\ubar d}$. 

\begin{figure}[t]
     \centering
      \begin{subfigure}[t]{0.45\textwidth}  \includegraphics[scale=0.75]{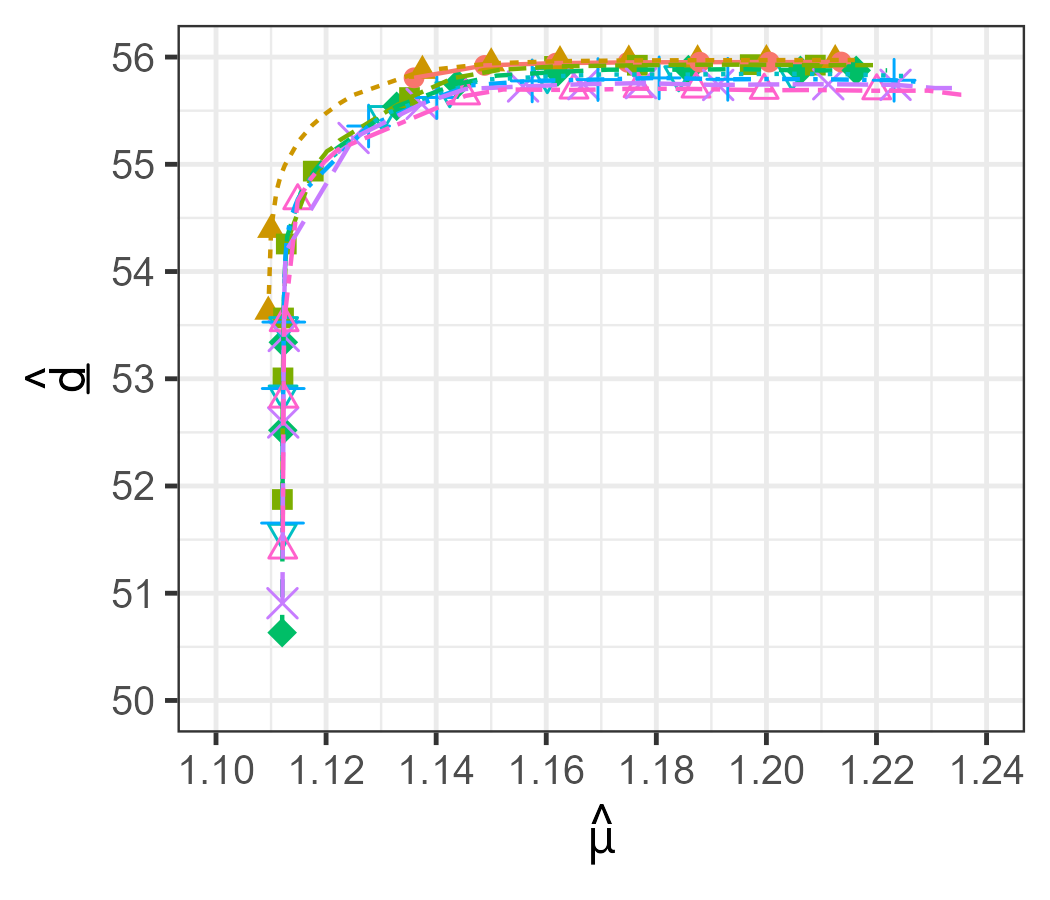}
     \vskip-10pt
   \caption{``True model'' with {$\delta=0.02$}% and $\gamma=0.04$
   }\label{fig:preformance_delta0.02_gamma0.04}
   \end{subfigure}
     \begin{subfigure}[t]{0.5\textwidth}
     \includegraphics[scale=0.75]{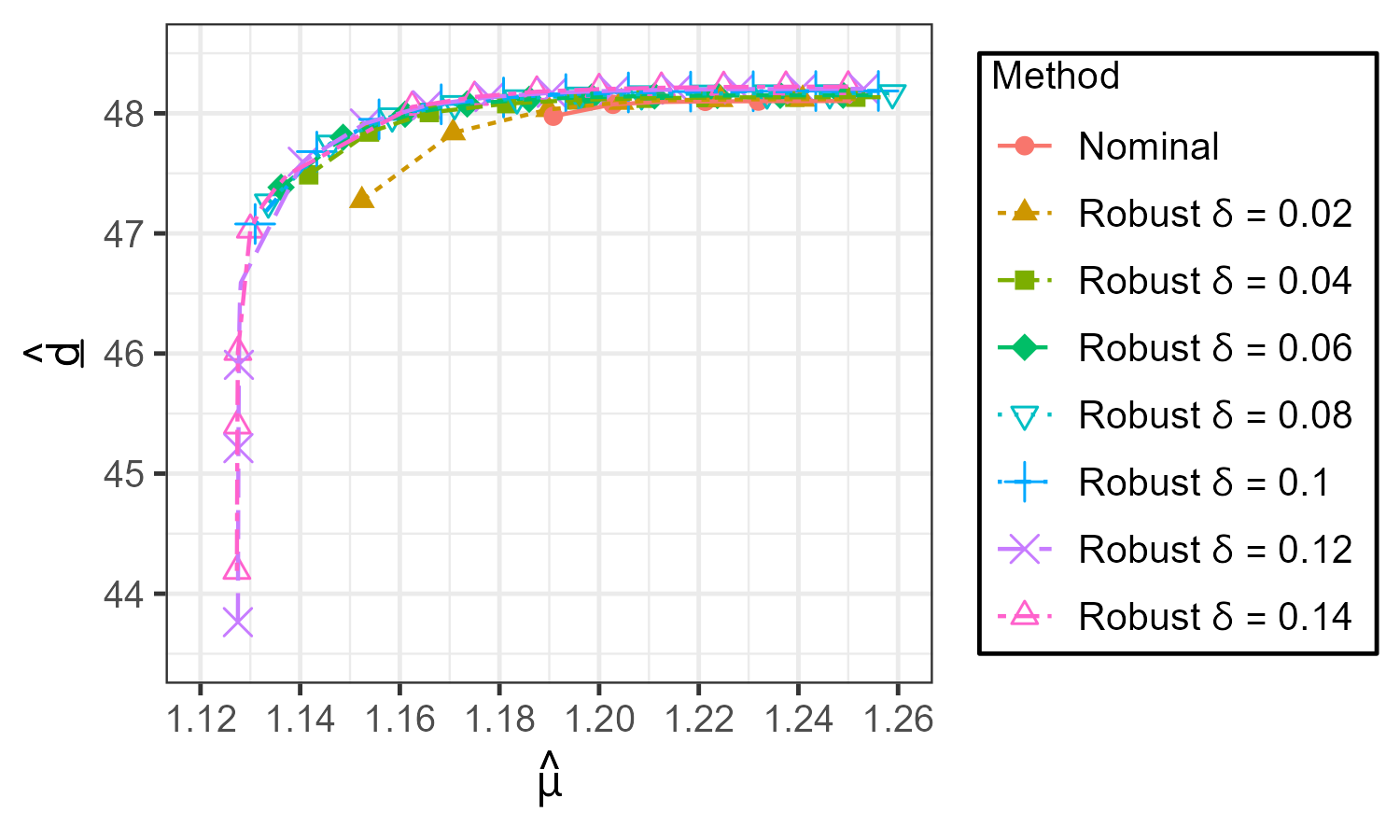}
     \vskip-10pt
   \caption{``True model'' with {$\delta=0.14$}%and $\gamma=0.04$
   }\label{fig:preformance_delta0.14_gamma0.04}
   \end{subfigure}
    \vskip-5pt
  \caption[]{Worst-case {biologically-adjusted} performance ($\hat{\mu}$, $\hat{\ubar{d}}$) of solutions optimal to~\eqref{prob:SR_robust} with misspecified $\delta$,  under the assumption that the ``true'' $\USB$ has $\gamma=0.04$ and $\delta$ as indicated}
   \label{fig:performance_diffdelta}
\end{figure}

Figure~\ref{fig:performance_diffdelta} depicts experimental results for evaluating the effect of misspecifying $\delta$ in the uncertainty model. In this set of experiments, the true uncertainty set, as well as the uncertainty sets assumed in solving  
\eqref{prob:SR_robust}, all share the same value of {$\gamma=0.04$}, but differ in the values of $\delta$. While the misspecification of $\delta$ seems to cause only a modest decrease in the value of $\hat{\ubar{d}}$ (of up to {$0.5$} Gy) compared with the value obtained under the true $\delta$, 
while the homogeneity, $\hat \mu$, appears to be more significantly affected. Specifically, underestimating $\delta$ 
%causes a decrease in homogeneity (i.e., an 
increases $\hat\mu$ of up to $0.03$ (for a similar dose), while overestimating $\delta$ may lead to little or no improvement in the homogeneity (a small decrease in $\hat\mu$). {Further, the $\hat\mu$ values attained for the nominal model are significantly worse (larger than the spatially robust models by up to 0.06) for the higher true value of $\delta=0.14$, compared to being approximately 0.03 larger than all spatially robust models when the true $\delta=0.02$.}

\subsection{Dose-Volume Experiments}
In this section, we conduct experiments using  Algorithm~\ref{alg:PARAMETRICschem}, introduced in Section~\ref{sec:tuning_beta} for handling dose-volume constraints.
r{The experiments in this segments were run using Julia 1.3.0, and Gurobi solver 9.0.0, on an Intel Xeon Gold 6254 CPU with a 3.10GHz base frequency, 24.75 MB cache, and 376 RAM.}

In order to speed up the computation, we may first narrow down the search interval further by performing a preliminary bisection search, and then proceed by running Algorithm~\ref{alg:PARAMETRIC} given this reduced interval. Note that invoking the bisection search may sacrifice some of the theoretical properties guaranteed by the initial lower bound, thus serving only as an effective heuristic. 
Further, in order to deal with possible degeneracy in Algorithm~\ref{alg:PARAMETRICschem} we determine adjacency of basic feasible solutions only based on the $x$ and $y$ variables.  
In these experiments, {we focus on the dose-volume constraint for the brain (OAR 2 of Patient 4)}. Specifically, following the ICRU guidelines~\citep{ICRU2010}, $99\%$ of the brain should receive less than $\bar{d}_2=60$ Gy %and 
while $\hat{d}_2=62$ Gy; that is, $\alpha=0.01$ in constraint~\eqref{DVCONSTRAINT}. For these experiments we consider both the nominal {and spatially robust settings with the same parameters as in Section~\ref{sec:compperform}}.%(that is, $\delta=0$ and $\gamma=1$) and with homogeneity parameter $\mu=1.125$, and the spatially robust case for r{$\delta=0.08$}, $\gamma=0.04$, with r{$\mu=1.1875$},which obtains a similar nominal homogeneity and minimal nominal dose without the dose volume constraints.  
Recall that Algorithm~\ref{alg:PARAMETRIC} determines the minimal $\beta$ value for which the solution satisfies this constraint. Intuitively, a small $\beta$ and corresponding penalty term in the objective should correspond to a small change in the objective compared with the non-penalized problem (where $\beta=0$). Thus, the constraint is expected to be satisfied % should yield only 
with a slight decline in the optimal objective value relative to that achieved when $\beta=0$. %This prediction can be evaluated empirically in the current experiment. 

Figure~\ref{fig:patient4_beta_functions} depicts the value of the adjusted dose $\ubar{d}$, and the $\ell_1$ and $\ell_0$ norms of the deviation vector $y$, as a function of $\beta$, for the different values of $\beta$ obtained by the algorithm for both cases. In order to run the algorithm, we first computed $\bar{\beta}\approx {2.22}\times 10^{-2}$ {for the nominal case, and 
$\bar{\beta}\approx {2.58}\times 10^{-2}$ for the specially robust case. {We used $\ubar{\beta}=0$ in both cases, while the true lower bounds are {$\ubar{\beta}\approx 1.73\times 10^{-6}$} for the nominal setting and {$\ubar{\beta}\approx 1.36\times 10^{-6}$} for the robust setting.}}
The algorithm determined that {$\beta^*=r{2.72}\times 10^{-5}$ and $\beta^*={2.4}\times 10^{-5}$ are the smallest $\beta$ values that satisfy $\norm{y}_0\leq \lfloor\alpha|H_2|\rfloor=13,601$ for the nominal and the spatially robust case, respectively (thus satisfying the dose-volume constraint). 

\begin{figure}[t]
    \centering
        \begin{subfigure}[b]{0.45\textwidth}
    \includegraphics[scale=0.55]{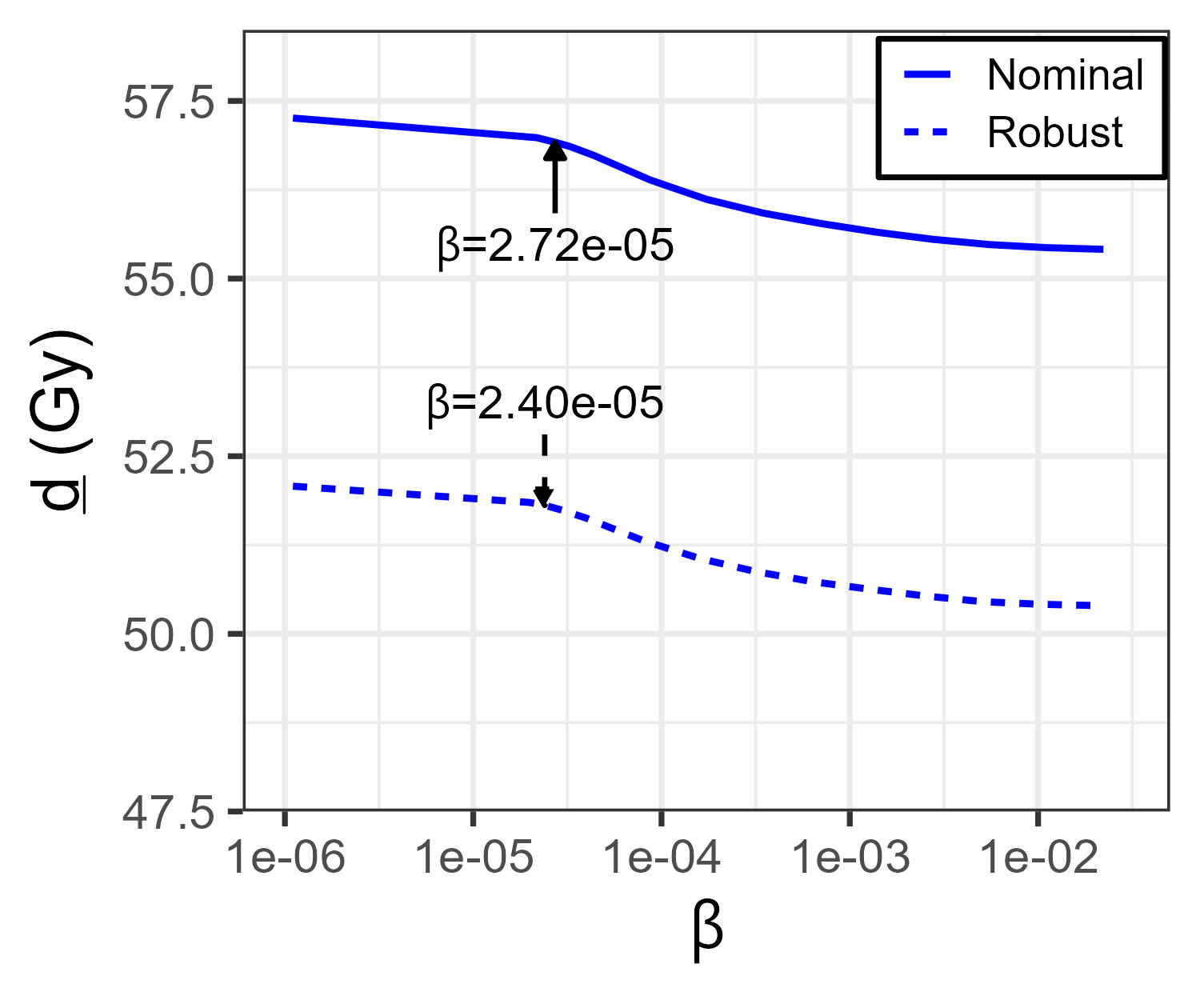}
    %\vskip-10pt
    \caption{$\ubar{d}$ vs. $\beta$}
    \end{subfigure}
     \begin{subfigure}[b]{0.50\textwidth}
    \includegraphics[scale=0.55]{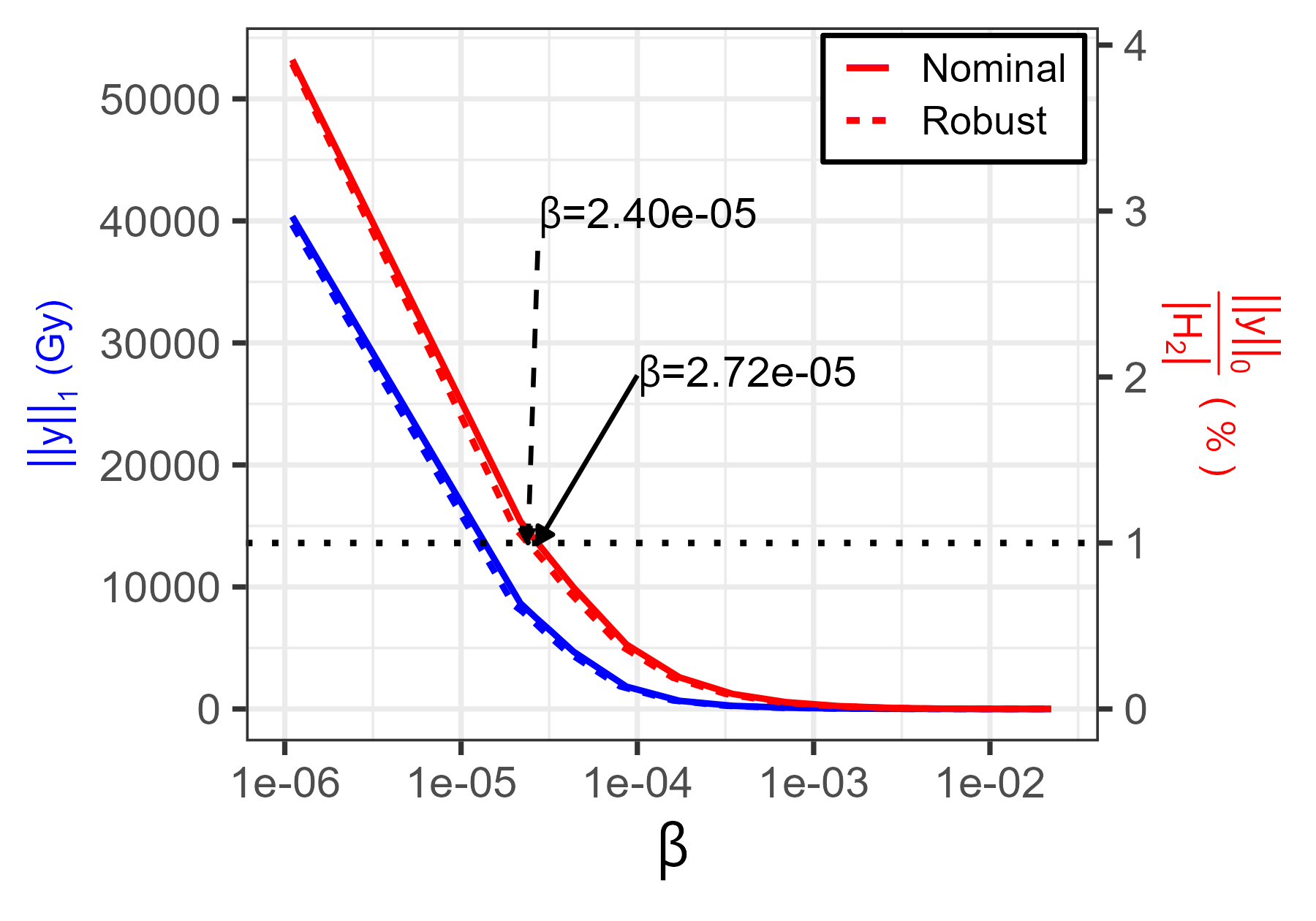}
    %\vskip-10pt
        \caption{$\vectornorm{y}_1$ and  $\vectornorm{y}_0/\card{H_2}$   vs. $\beta$}
    \end{subfigure}
    %\vskip-5pt
    \caption{%(A) Nominal calculated 
    	 {Minimum biologically-adjusted dose} $\underbar d$ and %(B) $\ell_1$ and $\ell_0$ 
    	%norms of the 
    	deviation {vector} $y$ norms %plotted 
    	vs. $\beta$ values enumerated by Algorithm~\ref{alg:PARAMETRICschem} for the nominal and spatially robust settings. \label{fig:patient4_beta_functions}}
\end{figure}

{We compared the solutions obtain by Algorithm 5 in both of the spatially robust and nominal settings, with a solution obtained %with the augmentation of a CVaR constraint to 
from problem~\eqref{prob:general_robust} with an additional CVAR constraint, as suggested by \cite{Romeijn2003}. 
Table~\ref{tbl:cmp_CVaR} shows these results, % of this comparison, %as well as a comparison with the results in 
comparing also with the unresricted case, where the dose-volume constraints are not imposed (%which is associated 
having $\beta=0$ in \eqref{prob:PEN_ROB_FORM}).}

\begin{table}
\caption{\label{tbl:cmp_CVaR} Comparison of Algorithm~\ref{alg:PARAMETRIC} with adding a CVaR constraint}
    \centering
    %\resizebox{0.9\textwidth}{!}{
    {
    \begin{tabular}{l|l|r|r|r|r}
    \textbf{Setting} & \textbf{Method} & \textbf{$\ubar{d}$} & \textbf{EUD} & \textbf{\% } & \multicolumn{1}{l}{\textbf{Running}}\\
    &&&&
    \textbf{Violation}
    & \multicolumn{1}{l}{\textbf{Time (sec)}}\\
    \hline
    \hline
    \multirow{3}{*}{Nominal}
    & Unrestricted & 57.26 & 64.85 & 3.91\%  & 2904 \\
    &Algorithm~\ref{alg:PARAMETRIC} & 56.92 & 64.50 & 1.00\% & 55784\\
    & CVaR & 56.46 & 64.03& 0.41\% & 87192 \\%- trivial, 14902 - iterative\\
    %& &  55.41 & & 0\% &\\
    \hline
    %\multirow{3}{50pt}{Robust ($\gamma=0.04$, $\delta=0.08$)} & Unrestricted & 51.12 & 63.95 & 4.03\%  & 15578 \\
    %&Algorithm~\ref{alg:PARAMETRIC} & 50.37 & 63.00 & 1.00\% & 174988\\
    %& CVaR & 50.10 & 62.65& 0.42\% & 360616 %- trivial, 79571 - iterative\\
    %& & 49.47 & & 0\% &\\
    \multirow{3}{50pt}{Robust} & Unrestricted & 52.08 & 64.73 & 3.89\%  &  4327\\
    &Algorithm~\ref{alg:PARAMETRIC} & 51.81 & 64.33 & 1.00\% & 68980\\
    & CVaR & 51.38 & 63.80& 0.41\% & 105366 %- trivial 20615 
    \end{tabular}}
    %}
\end{table}
{Algorithm~\ref{alg:PARAMETRIC} only reduces the dose by $0.59\%$ and $0.52\%$ compared to the unrestricted case, in the nominal and robust settings, respectively. Also, note that CVaR %provides lower values of both 
results in lower values of both minimal biologically-adjusted dose and EUD, compared %to the solution of A
with Algorithm~\ref{alg:PARAMETRIC}.
%and 
This is %more significant 
especially evident in the nominal setting. This is due to the CVaR constraint %providing 
being a safe approximation of the dose-volume constraint, %thus it is expected to provide 
resulting in more conservative solutions to the planning problem. Indeed, the percent of brain
voxels exceeding the $\bar d_2=60$Gy bound is approximately $0.4\%$ using the CVaR approach, which is lower than the required $1\%$. Finally,  the CVaR constraint requires a variable for each of the brain voxels, %thus, 
which may be computationally intensive; %expensive to apply. %To enable a fair comparison we attempted two implementation of this constraint: (1) a trivial implementation where all the variables are added at once (2) an iterative implementation, where variables are added iteratively based on whether their value is nonzero. 
%We found that 
without initializing to the unrestricted solution, we found that CVaR would not solve even after 72 hours.  With this particular initialization, we are able to solve for CVaR but it takes about 1.5 times as long as Algorithm~\ref{alg:PARAMETRIC}, as displayed in Table~\ref{tbl:cmp_CVaR}. %, while the iterative implementation is significantly faster. 
Thus, when Algorithm~\ref{alg:PARAMETRIC} can be applied, it provides an alternative to the CVaR approach that is both superior in terms of performance and in terms of the running time.
} 

%Moreover,  in the nominal case, the optimal solution \noam{determined} for  $\beta^*$ has a minimal adjusted dose $\ubar{d}=56.59$ Gy, which is only $0.54\%$ smaller than the adjusted dose of $56.90$ Gy obtained for $\beta=0$. Similarly, in the spatially robust case, the optimal value of $\ubar{d}$ decreases from $52.35$ Gy
%to $52.1$ Gy as $\beta$ increases from $0$ to $\beta^*$, which is even a smaller decline of approximately $0.48\%$.}

\section{Conclusions and Future Work}

We proposed a novel spatially based uncertainty set to model \shimrit{uncertainty in tumor radiosensitivity of radiation treatment based on measured  biomarkers. }% uncertainty that influences the radiosensitivity of radiation treatment.
\shimrit{Naively} incorporating this uncertainty set in robust voxel-based \shimrit{FMO} %fluence map optimization 
models \shimrit{would} lead {to an unmanageable increase in the already large problem size}. \shimrit{Alternatively, we compactly reformulate the robust problem by proving that at most a few uncertainty set extreme points need to be considered for each constraint.} %that reformulate it compactly using a polynomial number of constraints by proving that at most a few uncertainty set extreme points need to be considered for each constraint.
While theoretically tractable, \shimrit{the new compact formulation may still contain over a billion constraints}. We first show that although the compact formulation cannot be solved directly, it can be solved within a reasonable running time \noamr{and consumption of memory,} by carefully generating constraints as needed. Extensive experiments demonstrate advantages of the   proposed uncertainty set in terms of the adjusted homogeneity-dose tradeoff over the nominal and robust box-uncertain models.

\noamr{In this work, we also develop a method for satisfying a single dose-volume constraint  
that outperforms the
standard CVaR modeling approach to this problem.
%while outperforming a standard CVaR modeling approach to this problem. 
Although handling dose volume constraints remains computationally challenging
for clinical use, this can be addressed by future research focused on voxel aggregation techniques for solving large scale instances within reasonable running times.}  

This work has focused on spatial aspects of biomarker related uncertainty in a voxel-level planning model (also known as ``dose painting'').
%(dose painting) 
In future work, we would like to extend the proposed model to include temporal aspects,  taking into account fractionation of treatments as well as a possible additional image scan midtreatment. This leads to a challenging two-stage adaptive robust problem, \shimrit{where the spatial uncertainty changes over time}. Finally, it is worth noting that the study of robust and adaptive planning, under the condition that \noamr{the uncertainty of neighboring elements} \shimrit{is connected}, has application beyond radiation therapy. Communicable disease,  traffic congestion, and even agriculture are all examples of situations where uncertainty  among neighbors in space and time \shimrit{is connected}. 

%\section*{Acknowledgements}
%The authors are grateful to Aditya P. Apte for his help in using the CERR software to process and export the imaging datasets used in this paper. Noam Goldberg acknowledges the computational resources provided by the Data Science Institute at Bar-Ilan University.

%\section*{Disclosure/Declaration of Interest Statement}
 %The authors report there are no competing interests to declare.

\section*{Acknowledgments} The authors are grateful to Aditya P. Apte for his help in using the CERR software to process and export the imaging datasets used in this paper. Noam Goldberg acknowledges the computational resources provided by the Data Science Institute at Bar-Ilan University. %Finally, the authors would like to thank the area editor and anonymous referees for many comments and corrections that improved the quality of the manuscript.

\section*{Notes on the Contributors}

\paragraph{Noam Goldberg} is a senior lecturer in the Department of Management at Bar-Ilan University and an incoming associate professor at the Department of Industrial Engineering and Management at Ben-Gurion University. He completed a PhD in operations research at Rutgers University (RUTCOR), New Brunswick, NJ, in 2010. His graduate studies followed an early non-academic system and software engineering career in telecommunications. His PhD research on sparse data classification models was awarded a first prize in a contest of the INFORMS NJ chapter in 2009. He held several postdoctoral positions during 2009-2013 at the Technion-Israel Institute of Technology, Argonne National Laboratory and Carnegie Mellon University. He joined Bar-Ilan University in 2013 as a tenure-track lecturer, where he became a tenured senior lecturer since 2018. In fall 2019 he was a visiting associate professor at the Rutgers Business School. He has more than 30 peer reviewed publications including top journals in optimization and conference proceedings in computer science. In current work he continues to be motivated by high-impact %healthcare and societal 
applications of OR, data driven optimization, game models and solution methods.

\paragraph{Mark Langer, MD} is Professor of Clinical Radiation Oncology at Indiana University School of Medicine,
Indianapolis, IN USA. His research interests lie in mathematical programming and optimization for
radiation treatment planning and design, and clinical strategies to improve outcomes in head and neck
cancers. He is the recipient of grants from the NIH, NSF and the Indiana 21 st Century Fund and has
served on review panels for the NIH, NSF, BSF and the DOD. He served as chair for the Operations
Research and Radiotherapy Collaborative Working Group (ORART) sponsored by the NIH and was a
subawardee on the ORART Collaborative Working Group’s Tool Kit program funded by the NSF. He has
authored publications in Ann. Operations Research, Mathematical Programming B, Linear Algebra and
its Applications, J. Operat. Res. Soc., Journal of Mechanisms and Robotics, Journal of Applied Clinical
Medical Physics, Dysphagia, Int. J. Radiat. Oncol. Biol. Phys., Radiotherapy and Oncology, IIE
Transactions on Health Care Systems Engineering, Phys. Med. Biol, and Medical Physics.

\paragraph{Shimrit Shtern} is a senior lecturer in the Faculty of Data and Decision Sciences at the Technion - Israel Institute of Technology. She completed a PhD in optimization at the Technion in 2015. She was post-doctoral fellow at the Operation Research Center at MIT in 2015-2017 and joined the Technion as a faculty member in 2018. Her research focuses on theory and applications of optimization under uncertainty as well as development and convergence analysis of first order methods for structured optimization problems. Her work has been published in top operation research and optimization journals including Management Science, Operations Research, Mathematics of Operations Research, SIAM Journal on Optimization, and Mathematical Programming.  

\section*{Data Availability Statement}
All data used in this paper is based on data available from public repositories whose references are indicated and with processing as described in Section~\ref{sec:computational}. The processed data and code are available at \url{https://github.com/ShimShtern/RobustRTPStatic}. %in its final processed form will be made available by the authors upon reasonable request. 

\bibliographystyle{apalike} %informs2014}
\bibliography{refs2}

\newpage
%\begin{APPENDICES}

\appendix

\section{Detailed Algorithm Descriptions}

\subsection{Row Generation Algorithm}\label{appx:algo}

In the following, let $\Phi(S,\hat H_1,\ldots,\hat H_K)$ denote an instance of~\eqref{prob:SR_robust} with constraints~\eqref{constr:ROB_HOMOGEN1} and~\eqref{constr:ROB_HOMOGEN2} defined only for subsets $S\subseteq T^2$ and constraints~\eqref{constr:OAR_UB} defined only for subsets $\hat H_1\subseteq H_1,\ldots,\hat H_K\subseteq H_K$. 
For a given solution $s=(x,\ubar{d})$, we denote by $\Infs_{const}(s,q,\epsilon)$ the set of the $q$ most violated constraints of type $const$, where each of these constraints is violated by at least $\epsilon$. Further, the set of constraints is assumed to be sorted  in descending violation order. Accordingly, $\Infs_{\eqref{constr:OAR_UB}_k}(s,n_H,\epsilon)$ and $\Infs_{\eqref{constr:ROB_HOMOGEN1}, \eqref{constr:ROB_HOMOGEN2}}(s,n_S,\epsilon)$ denote sets of at most  $n_H$ and $n_S$ most violated constraints (%whose left-hand side exceeds the right-hand side by
each with violation of at least $\epsilon$) 
out of constraints \eqref{constr:OAR_UB} for OAR $k$, and constraints \eqref{constr:ROB_HOMOGEN1} and~\eqref{constr:ROB_HOMOGEN2}, respectively. Specifically, if  
these sets are all empty, then none of these constraints is violated by more than $\epsilon$. \noam{Algorithm~\ref{alg:RGA} is our row generation method for solving~\eqref{prob:SR_robust}, and is a fully detailed version of the sketch given as Algorithm~\ref{alg:RGAschem}. In all our experiments, the parameters  $\tau_{\Infs,3}=10^{-2}$, $n_H=2000$, and in Section~\ref{sec:compperform} different values of $n_0$ and $n_S$ are evaluated as indicated in Tables~\ref{table:runningtime}-\ref{table:runningtime2}.}

\begin{algorithm}
\caption{Row generation algorithm}\label{alg:RGA}
\LinesNumbered
\SetKwInOut{Input}{Input}
\SetKwInOut{Output}{Output}
\SetKwInOut{Init}{Initialization}
\Input{Initial beamlet configuration $x_0$, number of initial constraints for each OAR $n_0$, number of homogeneity constraints per iteration $n_S$, number of OAR constraints per iteration $n_H$, thresholds $\tau_{\Infs,1}>\tau_{\Infs,2}\geq 0,\tau_{\Infs,3}\geq0,\tau_{z}>0$ }
\Init{$S\leftarrow\emptyset$, $\hat H_1,\ldots,\hat H_K$ based on $x_0$ with  $n_0$ constraints in each OAR, 
$z_{\text{prev}}\leftarrow \infty$, $m_S\leftarrow 1$, $m_{H_k}\leftarrow 0, \; \forall k$\;}
\While{$ \sum_{k=1}^K m_{H_k}+m_S>0$}{
\If{$m_S>0$}
{Solve $\Phi(S,\hat H_1,\ldots,\hat H_K)$ and obtain its optimal solution $s$ and optimal value $z$\;}
\While{$|\Infs_{\eqref{constr:OAR_UB}}(s,1,\tau_{\Infs,1})|>0$ or $\frac{z_{\text{prev}}-z}{z_{\text{prev}}} > \tau_{z}$ or ($m_S=0$ and $ \sum_{k=1}^K m_{H_k}>0$)}{
$z_{\text{prev}}\leftarrow z$\label{INLOOPSTART}\;
\For{$k=1,\ldots,K$}{
Set $\tilde{H}\leftarrow\Infs_{\eqref{constr:OAR_UB}_k}(s,n_H,\tau_{\Infs,2})$, $m_{H_k}=|\tilde{H}|$, $\hat{H}_k\leftarrow \hat{H}_k \cup\tilde{H}$
}
\If{$\sum_{k=1}^K m_{H_k}>0$}{
Solve $\Phi(S,\hat H_1,\ldots,\hat H_K)$ and obtain its 
optimal solution $s$ and optimal value $z$\;}\label{INLOOPEND}}
Set $\tilde S \leftarrow \Infs_{  \eqref{constr:ROB_HOMOGEN1},\eqref{constr:ROB_HOMOGEN2}}(s,n_S,\tau_{\Infs,3})$, $m_S\leftarrow|\tilde{S}|$, $S\leftarrow S \cup\tilde{S}$\;
}
\Output{$s$, $z$}
\end{algorithm}

\subsection{Row and Column Generation Algorithm}\label{sec:RCGenAlg}

Here, $\Phi(S,\hat H_1,\ldots,\hat H_K,\beta)$  %as the 
denotes an instance of problem~\eqref{prob:PEN_ROB_FORM}  with constraints~\eqref{constr:ROB_HOMOGEN1} and~\eqref{constr:ROB_HOMOGEN2} defined only for a subset $S\subseteq T^2$, constraints~\eqref{constr:Hard_HK_BOUNDS2} defined only for subsets $\hat H_1\subseteq H_1,\ldots,\hat H_{K-1}\subseteq H_{K-1}$, and constraints \eqref{constr:SOFT_HK_BOUNDS},\eqref{constr:BOUND_DELTA} defined for subset $\hat H_K\subseteq H_K$. 
 
Accordingly, solutions of this problem are of the form $s=(x,\ubar{d},y_{\hat{H}_K})$. For solutions that are optimal to this problem, note that, for all $v\in \hat H_K$, $y_{v}=\max\{d_v(x)-\bar d_K,0\}$. For all  $v\in H_K\setminus \hat H_K$, we similarly define, for the given solution $s$,   $y_{v}=\max\{d_v(x)-\bar d_K,0\}$. 
We also define $$\Cols(\hat H_K, s,n_H,\epsilon)\in
\argmax_{H\subseteq H_K\setminus \hat H_K}\set{\vectornorm{y_{H}}_1}{y=((Dx)_{H_K}-\bar d_K\mathbbm{1})_+,\quad \card{H}\leq n_H,\quad \min_{v\in H}\{y_{v}\}\geq\epsilon}, 
$$ 
which is a subset of the voxels of OAR $K$ with the largest (up to) $n_H$ components of $y_{H_K\setminus \hat H_K}$ (those excluded from the current subproblem).
\begin{algorithm}[]
\caption{Row and column generation algorithm}\label{alg:RCGA}
\LinesNumbered
\SetKwInOut{Input}{Input}
\SetKwInOut{Output}{Output}
\SetKwInOut{Init}{Initialize}
\Input{Initial mode $\Phi(S,\hat H_1,\ldots,\hat H_K,\beta)$ with value $z$, number of homogeneity constraints per iteration $n_S$, number of OAR constraints per iteration $n_H$, thresholds $\tau_{\Infs,1}>\tau_{\Infs,2}\geq 0,\tau_{z}>0$ }
\Init{$m_S\leftarrow 1$, $m_{H_k}\leftarrow 0, \; \forall k$\;}

\While{$ \sum_{k=1}^K m_{H_k}+m_S>0$}{
\If{$m_S>0$}
{Solve $\Phi(S,\hat H_1,\ldots,\hat H_K,\beta)$ and obtain its optimal solution $s=(x,\ubar{d},y_{\hat{H}_K})$ and optimal value $z$\;}
\While{$|\Infs_{\eqref{constr:ALL_ROBUST}}(s,1,\tau_{\Infs,1})|>0$ or $\frac{z_{\text{prev}}-z}{z_{\text{prev}}} > \tau_{z}$ or ($m_S=0$ and $ \sum_{k=1}^K m_{H_k}>0$)}{
$z_{\text{prev}}\leftarrow z$\label{INLOOPSTART2}\;
\For{$k=1,\ldots,K-1$}{
Set $\tilde{H}\leftarrow\Infs_{\eqref{constr:OAR_UB}_k}(s,n_H,\tau_{\Infs,2})$, $m_{H_k}=|\tilde{H}|$, $\hat{H}_k\leftarrow \hat{H}_k \cup\tilde{H}$
}
Set $\tilde{H}\leftarrow \Cols(s,n_H,\tau_{\Infs,2})$, $m_{H_K}\leftarrow |\tilde{H}|$, $\hat{H}_K\leftarrow \hat{H}_K \cup \tilde{H}$\;
\If{$\sum_{k=1}^K m_{H_k}>0$}{
Solve $\Phi(S,\hat H_1,\ldots,\hat H_K,\beta)$ and obtain its 
optimal solution $s=(x,\ubar{d},y_{\hat{H}_K})$ and optimal value $z$\;}\label{INLOOPEND2}}
Set $\tilde S \leftarrow \Infs_{  \eqref{constr:ROB_HOMOGEN1},\eqref{constr:ROB_HOMOGEN2}}(s,n_S,\tau_{\Infs,2})$, $m_S\leftarrow|\tilde{S}|$, $S\leftarrow S \cup\tilde{S}$\;
}
\Output{$\Phi(S,\hat H_1,\ldots,\hat H_K,\beta)$, $s$, $z$}
\end{algorithm}

\subsection{Parametric Algorithm\label{appx:paramalgo}}

Suppose $s=(\ubar{d},x,y_{\hat{H}_K})$ and $s'=(\ubar{d}’,x’,y_{\hat{H}_K’}’)$ are basic feasible solutions, optimal to 
$\Phi(S,\hat H_1,\ldots,\hat H_K,\beta)$ and 
$\Phi(S',\hat H_1',\ldots,\hat H_K',\beta')$, respectively, where without loss of generality, $S\subseteq S'$ and $\hat H_{k}\subseteq \hat H_{k}'$, for all $k=1,\ldots, K$. In this context, we say that $s$ and $s’$ are adjacent if once extending subvector $y_{\hat H_K'\setminus\hat H_K}$ with $y_{\hat H_K}$ (appending the absent variables), the extended solution $(\ubar{d},x,y_{\hat{H}_K’})$ is an adjacent basic feasible solution to $s'$ in problem 
$\Phi(S',\hat H_1',\ldots,\hat H_K',\beta')$ (that is, each solution can be reached from the other through a single Simplex pivot). We denote this adjacency by the Boolean function $\adjacent(s,s')$, returning true if $s$ and $s'$ are adjacent, and false otherwise.

 \begin{algorithm}
    \caption{Parametric penalty algorithm\label{alg:PARAMETRIC}}
\LinesNumbered
\SetKwInOut{Input}{Input}
\SetKwInOut{Output}{Output}
\SetKwInOut{kwInit}{Initialize}
\Input{Initial bounds $0\leq \beta_l < \beta_u$, $s^{(1}=s^{(0)}=(\ubar{d}^{(0)},x^{(0)},y_{\hat{H}_K}^{(0)})$ optimal for $\beta_l$ (and where $\emptyset=\Cols(\hat H_K, s, 1,\tau_{\Infs,2})$), $\epsilon>0$, 
%where $\hat{H}_K=\Cols(\hat H_K, s,\tau_{\Infs,2}/2)$.
}
\kwInit{$\beta^{(1)}\leftarrow \beta_l$,
$\beta^{(0)}\leftarrow \beta-\epsilon$.}
%Set $s_{\text{prev}}\leftarrow (0,0,\emptyset)$, $S\leftarrow\Infs_{  \eqref{constr:ROB_HOMOGEN1}, \eqref{constr:ROB_HOMOGEN2}}(s,\tau_{\Infs,2}/2)$, for $k\in\{1,\ldots,K-1\}$ set $\hat{H}_k\leftarrow\Infs_{ \eqref{constr:Hard_HK_BOUNDS2}_k}(s,\tau_{\Infs,2})$\;
\For{$\ell = 1,\ldots$}{%{$\vectornorm{y_{\hat{H}_K}}_0 {>} \theta$ or $\beta\leq \beta_{\text{prev}}$ or $\neg\adjacent(s,s_{\text{prev}})$}{
Determine objective coefficient interval $[\beta^{B}_l,\beta^{B}_u]$ that maintains the optimality of current partial basis $B_{s^{(\ell)}}$\label{algstep:BETAINTVL}\;
%$\beta_{\text{prev}}\leftarrow\beta$\;
\uIf{%$\beta^B_u \geq \beta_{\text{prev}}-\epsilon$ and 
$\vectornorm{y_{\hat{H}_K}^{(\ell)}}_0 >  \theta$}
{ $\beta^{(\ell)} \leftarrow \min\{\beta^B_u + \epsilon,\beta_u\}$\; \label{algstep:BETAINC}}
\uElseIf{$\beta^{(\ell)}>\beta^{(\ell-1)}$ and $s^{(\ell)}$ and $s^{(\ell-1)}$ once extended are adjacent BFS of $\Phi(T,H_1,\ldots,H_K)$}{break\;}
\Else{%$\beta_l\leftarrow\beta$\;
$\beta^{(\ell)}\leftarrow\max\{\beta^B_l-\epsilon,\beta_l\}$ \label{algstep:BETADEC}\; }
 %$s_{\text{prev}}\leftarrow s$\;
{Solve subproblem  $\Phi(S,\hat H_1,\ldots,\hat H_K,\beta^{(\ell)})$ to obtain 
a solution $s^{(\ell+1)}$ %with objective value $z$ 
and  updated constraint sets $(S,\hat H_1,\ldots,\hat H_K)$\label{RCInvokeStep}}\;}
\Output{$\beta^{(\ell)}$ and solution $s^{(\ell)}$.}
\end{algorithm}

\section{Proofs}

\subsection{Proof of Lemma \ref{lemma:monotone}}
\label{appx:proof_lemma_monotome}

\begin{proof}%{Proof.}
Since both $(x,\ubar{d},y)$ and $(\tilde x,\tilde{\ubar{d}},\tilde y)$ are feasible for problem~\eqref{prob:PEN_ROB_FORM}, $(x,\ubar{d},y)\in \mathcal{S}(\beta_1)$, $\beta_2>\beta_1$, and $y\geq 0$; thus, the following inequalities hold:
\begin{equation}\label{eq:opt_beta1}
 \ubar{d}-\beta_1\sum_{v\in H_K}y_v\geq \tilde{\ubar{d}}-\beta_1\sum_{v\in H_K}\tilde y_v\geq \tilde{\ubar{d}}-\beta_2\sum_{v\in H_K}\tilde y_v.
\end{equation}
Thus, we have obtained the third inequality in the lemma. 
Additionally, since $(\tilde x,\tilde{\ubar{d}},\tilde y)\in \mathcal{S}(\beta_2)$, the following inequality holds:
\begin{equation}\label{eq:opt_beta2}
\tilde{\ubar{d}}-\beta_2\sum_{v\in H_K}\tilde y_v\geq
 \ubar{d}-\beta_2\sum_{v\in H_K}y_v.
\end{equation}
Combining the first inequality of \eqref{eq:opt_beta1} with \eqref{eq:opt_beta2} leads to
\begin{equation*}
\beta_1(\sum_{v\in H_K}y_v-\sum_{v\in H_K}\tilde y_v)\leq \ubar{d}-\tilde{\ubar d}\leq \beta_2(\sum_{v\in H_K}y_v-\sum_{v\in H_K}\tilde y_v),
\end{equation*}
Which, due to $\beta_2>\beta_1$, implies that both $\sum_{v\in H_K}y_v\geq \sum_{v\in H_K}\tilde y_v$ and $\ubar{d}\geq\tilde{\ubar d}$. %\Halmos
\end{proof}

%\subsection{Proof of Lemma \ref{lemma:monotone}\label{appx:proof_lemma_monotome}}

\subsection{Penalty Parameter Bounds Lemma\label{appx:proof_lem_bounds}}

The next lemma establishes bounds on the value of $\beta^*$.

\begin{lemma}\label{lem:bounds}~
Consider an optimal solution of formulation~\eqref{prob:BOUND_ROB_FORM} with $\Theta =  0$, and let $\lambda_u$ be an optimal dual variable value corresponding to constraint~\eqref{constr:BUDGET}. 
Next, consider an optimal solution of formulation~\eqref{prob:BOUND_ROB_FORM} with $\Theta =  \lfloor\alpha |H_K|\rfloor(\hat d_K -\bar d_K)+\epsilon$ for every $\epsilon>0$, and let $\lambda_l$ be an optimal dual variable value corresponding to constraint~\eqref{constr:BUDGET}.
Then, $\lambda_l\leq \beta^*\leq \lambda_u$.
\end{lemma}
%The proof of this lemma is provided in Appendix~\ref{appx:proof_lem_bounds}.  The next proposition establishes the correctness of our algorithm in determining $\beta^*$ under a uniqueness assumption. 
\begin{proof}
For the following proofs, we use notation $\mathcal{X}$ to denote the feasible solution set of \eqref{prob:PEN_ROB_FORM}. 
Also, for every $\Theta\geq 0$, $Z_C(\Theta)$ denotes the optimal value of problem \eqref{prob:BOUND_ROB_FORM} with the right-hand side of constraint~\eqref{constr:BUDGET} given by $\Theta$. 
Finally, consider %$q(\Theta,\cdot)$ denote 
the dual function of~\eqref{prob:BOUND_ROB_FORM}, 
obtained by dualizing constraint~\eqref{constr:BUDGET} with dual multiplier $\beta$,  
$$q(\Theta,\beta)\equiv \beta \Theta + \max_{(x,\ubar d,y)\in \mathcal{X}}\{\ubar d-\beta\sum_{v\in H_k}y_v\}=\beta \Theta+z(\beta).$$
First, to prove the upper bound, let $s^*=(x^*,\ubar d^*,y^*)$ denote an optimal solution of formulation~\eqref{prob:BOUND_ROB_FORM} with $\Theta = 0$, and let $\lambda_u$ be an optimal dual variable value corresponding to constraint~\eqref{constr:BUDGET}.
Strong Lagrangian duality implies that $(x^*,\underbar d^*,y^*)\in \mathcal{S}(\lambda)$ and 
$q(0,\lambda)= 
\lambda \cdot 0 + z(\lambda)= \ubar{d}^* =Z_C(0)$,  
where the second equality follows from feasibility, which requires $y^*=\mathbf{0}$. Evidently, there exists a finite  
$\lambda$ such that %$z(\lambda)=Z_C(0)=\ubar{d}^*$ with 
$y(\lambda)=y^*=\mathbf{0}$, implying also that $\vectornorm{y(\lambda)}_0=\vectornorm{y(\lambda)}_1=0$. Thus, by Lemma~\ref{lemma:monotone}, $\beta^*\leq \lambda_u$.

To prove the lower bound, let $\epsilon>0$, let $s^*=(x^*,\underbar d^*,y^*)$ denote an optimal solution of formulation~\eqref{prob:BOUND_ROB_FORM} with $\Theta =  \lfloor\alpha |H_K|\rfloor(\hat d_K -\bar d_K)+\epsilon$, and let $\lambda_l$ be an optimal dual variable value corresponding to constraint~\eqref{constr:BUDGET}.
 By definition of $\lambda_l$ and from strong duality, we have 
 $(x^*,\ubar{d}^*,y^*)\in \mathcal{S}(\lambda_l)$, and from the feasibility of $y^*$ for problem \eqref{prob:BOUND_ROB_FORM}, it follows that 
 $\norm{y^*}_1\leq \Theta$. Moreover, as a result of strong duality,
 $$q(\Theta,\lambda_l)=\lambda_l\Theta+
 z(\lambda_l)=\min_{\beta\geq 0}q(\Theta,\beta)=Z_C(\Theta)=\ubar{d}^*,$$
 and either $\norm{y^*}_1=\Theta$ or $\lambda_l=0$. If $\lambda_l=0$, then, by feasibility of $\beta^*$ for the dual problem, $\beta^*\geq \lambda_l=0$. Otherwise 
 (when $\norm{y^*}_1=\Theta$), by definition of $\beta^*$, $\vectornorm{y(\beta^*)}_0\leq \lfloor \alpha\card{H_K}\rfloor$, which implies that $\vectornorm{y(\beta^*)}_1\leq  \lfloor\alpha\card{H_K}\rfloor(\hat d_K-\bar d_K)=\Theta-\epsilon=\norm{y^*}_1-\epsilon$. Thus, it follows from Lemma~\ref{lemma:monotone} that $\beta^*\geq \lambda_l$. %\Halmos 
% \shimrit{[unclear why we need $\epsilon$ here]}
\end{proof}

\subsection{Proof of Proposition~\ref{prop:param_finite_steps}\label{appx:proof_param_finite}}
\begin{proof}

First observe that at the start of each iteration of~Algorithm~\ref{alg:PARAMETRIC},  $\beta>\beta_{\text{prev}}$ unless, in the previous iteration, $\vectornorm{y}_0 \leq  \lfloor\alpha\card{H_k}\rfloor$ and $\adjacent(s,s_{\text{prev}})$ does not hold (so a basic feasible solution has been ``skipped''). Let $\beta_{i-2},\beta_{i-1},\beta_{i}$ denote a sequence of such penalty parameters $\beta$ in iterations $i-2$, $i-1$ and $i$, respectively,
where $\beta_{i-2} < \beta_i < \beta_{i-1}$ and where the bases corresponding to $\beta_{i-2}$ and $\beta_{i-1}$ are not adjacent.   
This implies that once all active constraints have been generated for the basic feasible solution corresponding to $\beta_{i-1}$, a distinct basic feasible solution can be found in $[\beta_{i-2},\beta_{i-1}]$. 
The finiteness of the number of steps for determining the returned $\beta$ follows from the existence of a finite number of basic feasible solutions. Then, by the uniqueness assumption, the fundamental theorem of LP, and the fact that the algorithm traverses all basic optimal solutions for the penalty parameter values in $[\beta_l,\beta]$, it follows that the algorithm terminates with $\beta=\beta^*$.
\end{proof}

\section{Conversion of PET Image SUV to OMF}\label{appx:suv}

In our experiments, imaging standard uptake values (SUV) are read directly from an FMISO-PET scans. The SUV values are typically normalized with respect to well oxygenated region, and in our case the reference value {of $pO2=26$ is fixed as the the median over the entire brain (excluding the tumor)~\citep{Mckeown2014defining}}. %95\%-tile SUV reading in the image. 
These normalized {SUV} readings are converted to oxygen pressure $pO2$ values using the formula
\[
pO2(SUV) = \frac{(A-SUV)\times C}{SUV - A+B}, 
\] where $A=10.9$, $B=10.7$ and $C=2.5$~\shimritr{(\citealp{toma2012} and \citealp{lindblom2017defining})}.
Then, the OER is given by
\[
OER(pO2) = \frac{m\times pO2+K}{pO2+K}; 
\] see for example~\cite{Titz2008}. The values assumed here are $K=3$ and $m=3$ following~\cite{Saka2014}. 
Finally, the OMF is simply OER normalized as a proportion of the maximum OER value in the image. Finally, we apply interpolation to deduce voxel OMF values for the CT images that may have a higher resolution of voxels than the corresponding PET images. 

{Note that the values of the constants in this appendix appear to be highly uncertain according to the scientific literature. In particular, for the conversion of $pO2$ to $OER$,~\cite{Titz2008} consider a range of possible values $K=2,3,4$. Note that for $pO2=15$ it implies that the corresponding $OMF$ value could vary in the interval $(0.875,0.930)$ (assuming a well oxygenated reference point has $pO2=151$) due only to parameter $K's$ uncertainty, while significant uncertainty may be associated with the values of other parameters as well. Based on such a sensitivity analysis, experimenting with values of $\delta$ in $[0,0.14]$ would be reasonable.}

\section{Additional Numerical Results}\label{appx:numerical}

Figure~\ref{fig:OPTVSDELTA} shows the optimal solution value as a function of $\delta$ for different values of $\mu$. In this figure, points that can only be satisfied by the zero beamlet intensity configuration ($x=0$) are omitted from the graph (this applies to all figures throughout the paper). In particular, for each value of $\mu$, only the zero solution is feasible below a certain $\delta$ threshold. 
For example, $\mu=1.125$ 
cannot be satisfied with any positive dose (accordingly, positive beamlet intensity) when $\delta\geq 0.03$. This is because $\mu$ 
constrains the biological homogeneity,  
which is directly affected by the ``size'' of the uncertainty set corresponding to the magnitude of the parameters $\delta$ and $\gamma$.

\begin{figure}[H]
    \centering
    \includegraphics[scale=0.8]{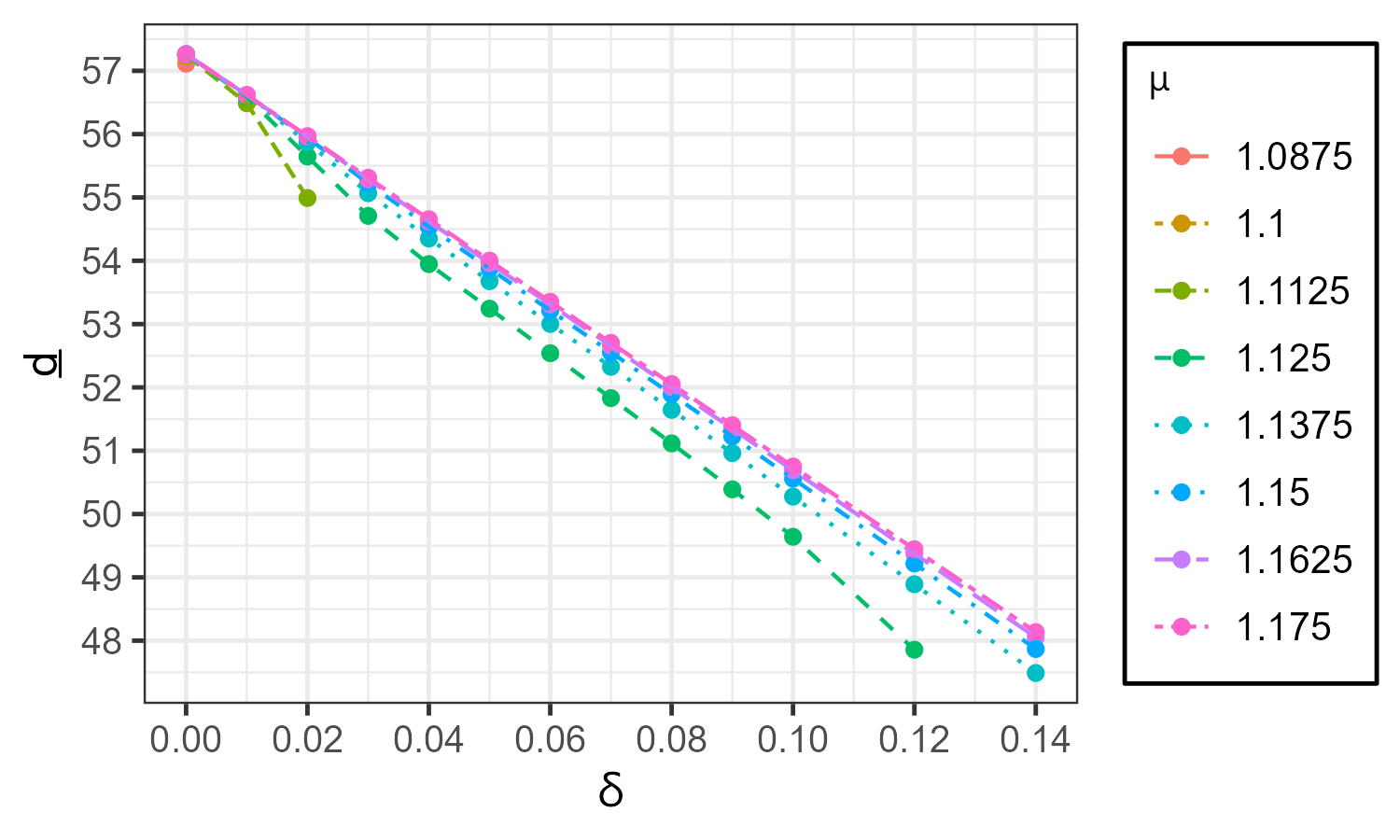}
     \vskip-10pt
    \caption{Optimal solution value of~\eqref{prob:SR_robust} for $\gamma=0.04$ and different values of $\delta$.}
    \label{fig:OPTVSDELTA}
\end{figure}

\begin{figure}[H]
    \centering
    \includegraphics[width=\textwidth]{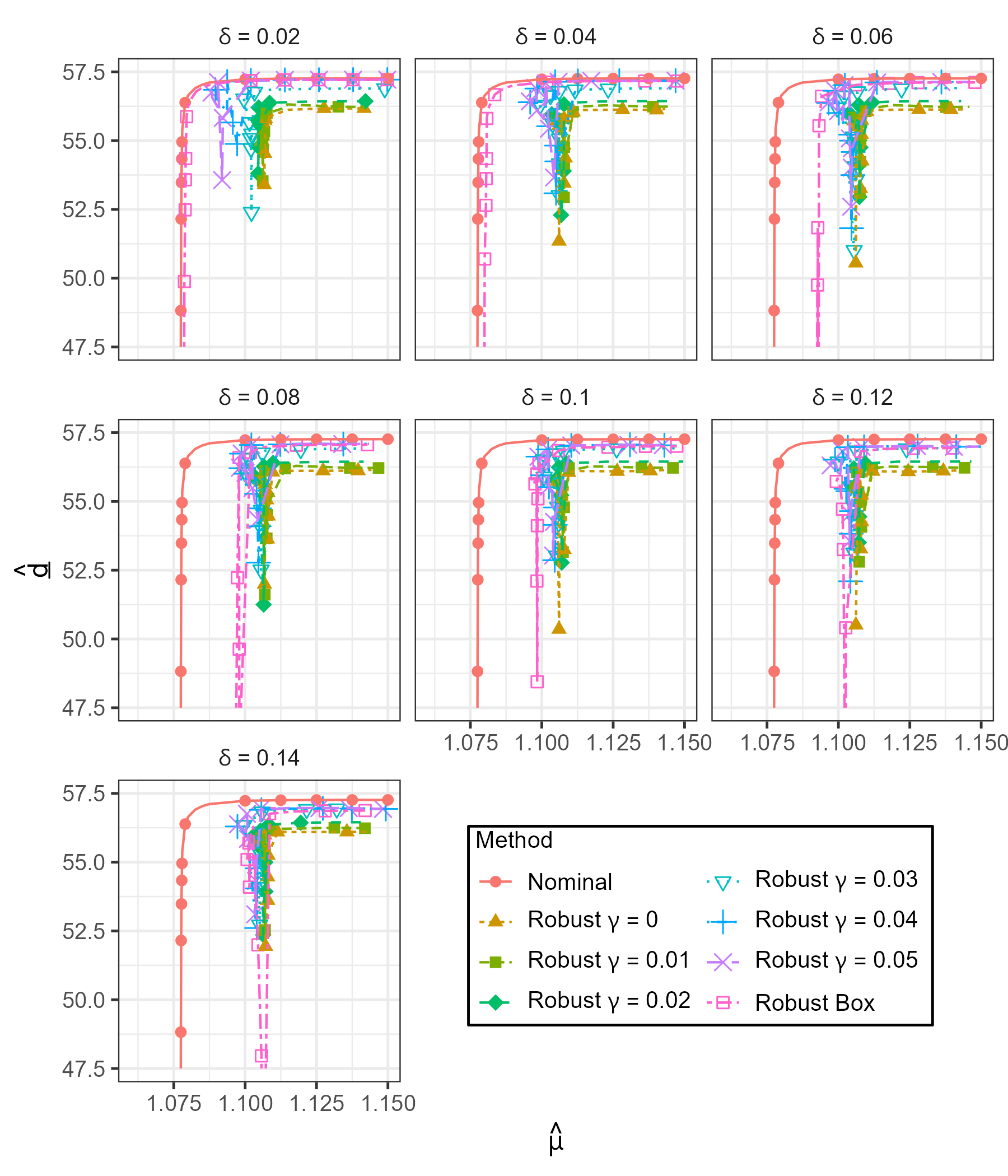}
    \caption{Nominal performance of optimal solutions
   to~\eqref{prob:SR_robust} with different values of $\gamma$ and $\delta$, compared with the nominal solution.}
    \label{fig:nominal_deltaall}
\end{figure}
\begin{figure}[H]
    \centering
    \includegraphics[width=0.95\textwidth]{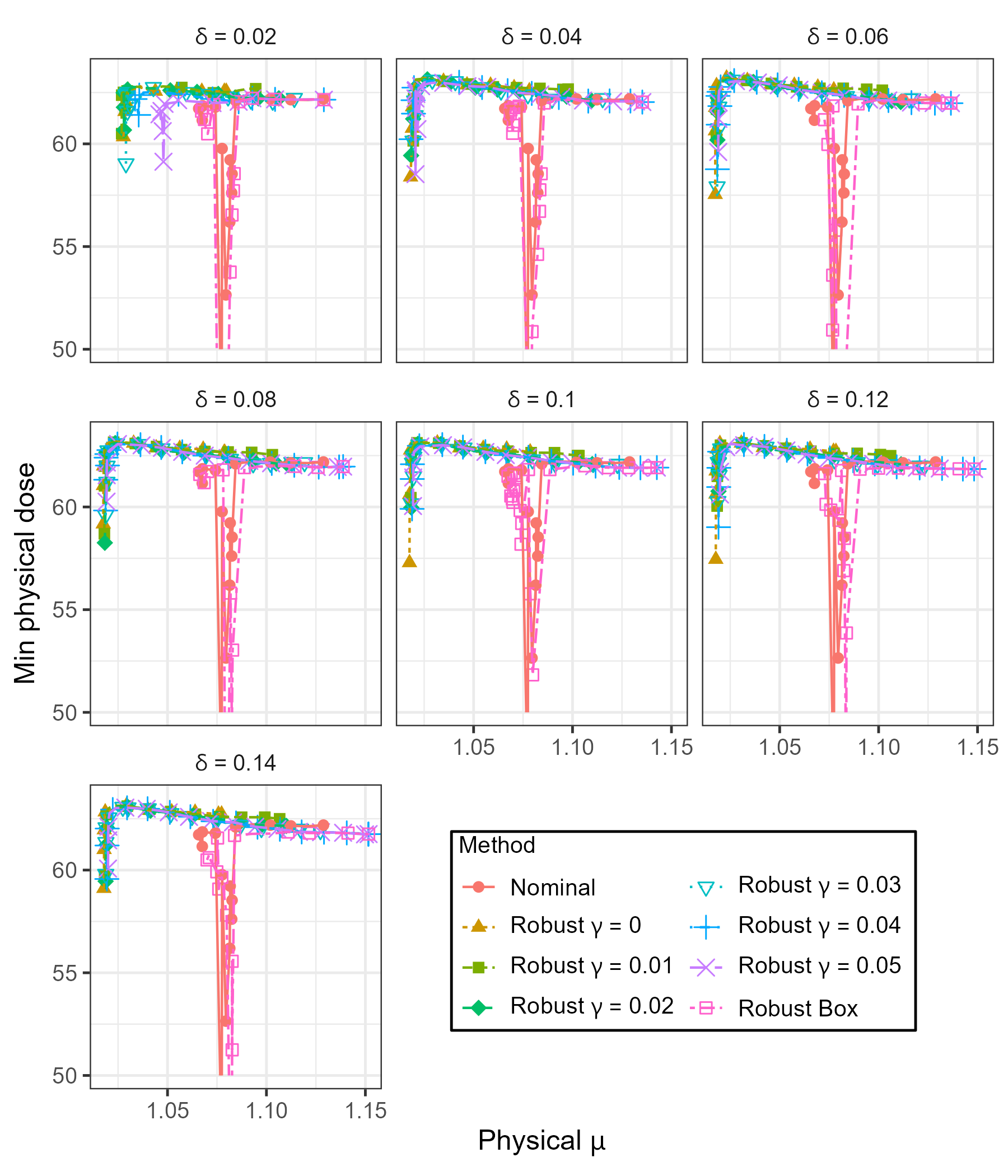}
    \caption{Physical performance of optimal solutions
   to~\eqref{prob:SR_robust} with different values of $\gamma$ and $\delta$, compared with the nominal solution.}
    \label{fig:physical_deltaall}
\end{figure}
\begin{figure}[H]
   \centering
     \includegraphics[width=\textwidth]{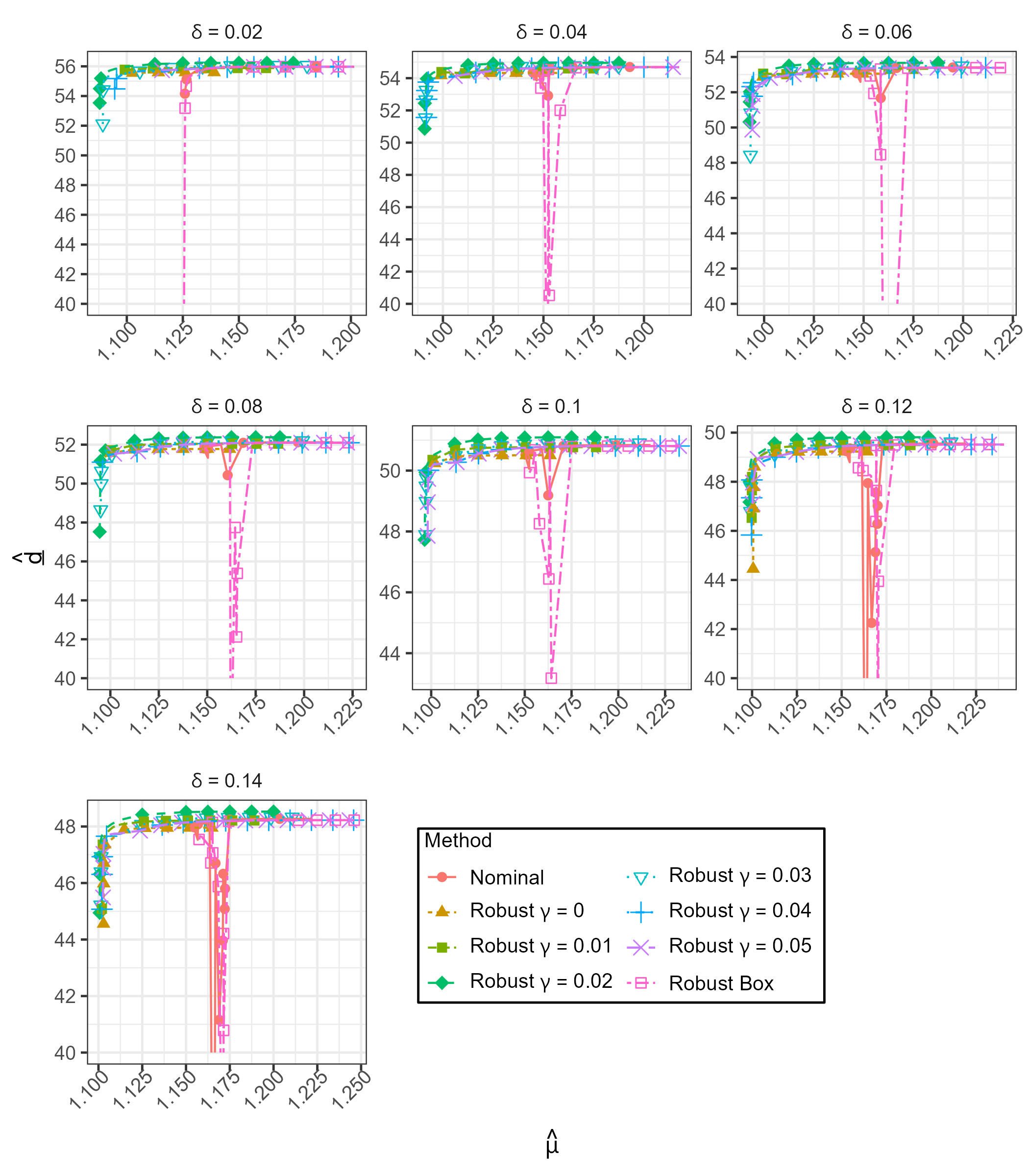}
   \caption{Worst-case performance ($\hat{\mu}$, $\hat{\underbar{d}}$) of solutions optimal to~\eqref{prob:SR_robust} with misspecified $\gamma$, under the assumption that the ``true" $\USB$ has $\gamma=0.02$ and $\delta$ as indicated. The performance of the nominal solution is also shown for comparison.}
   \label{fig:preformance_gamma0.02}
\end{figure}

\begin{figure}
    \includegraphics[width=\textwidth]{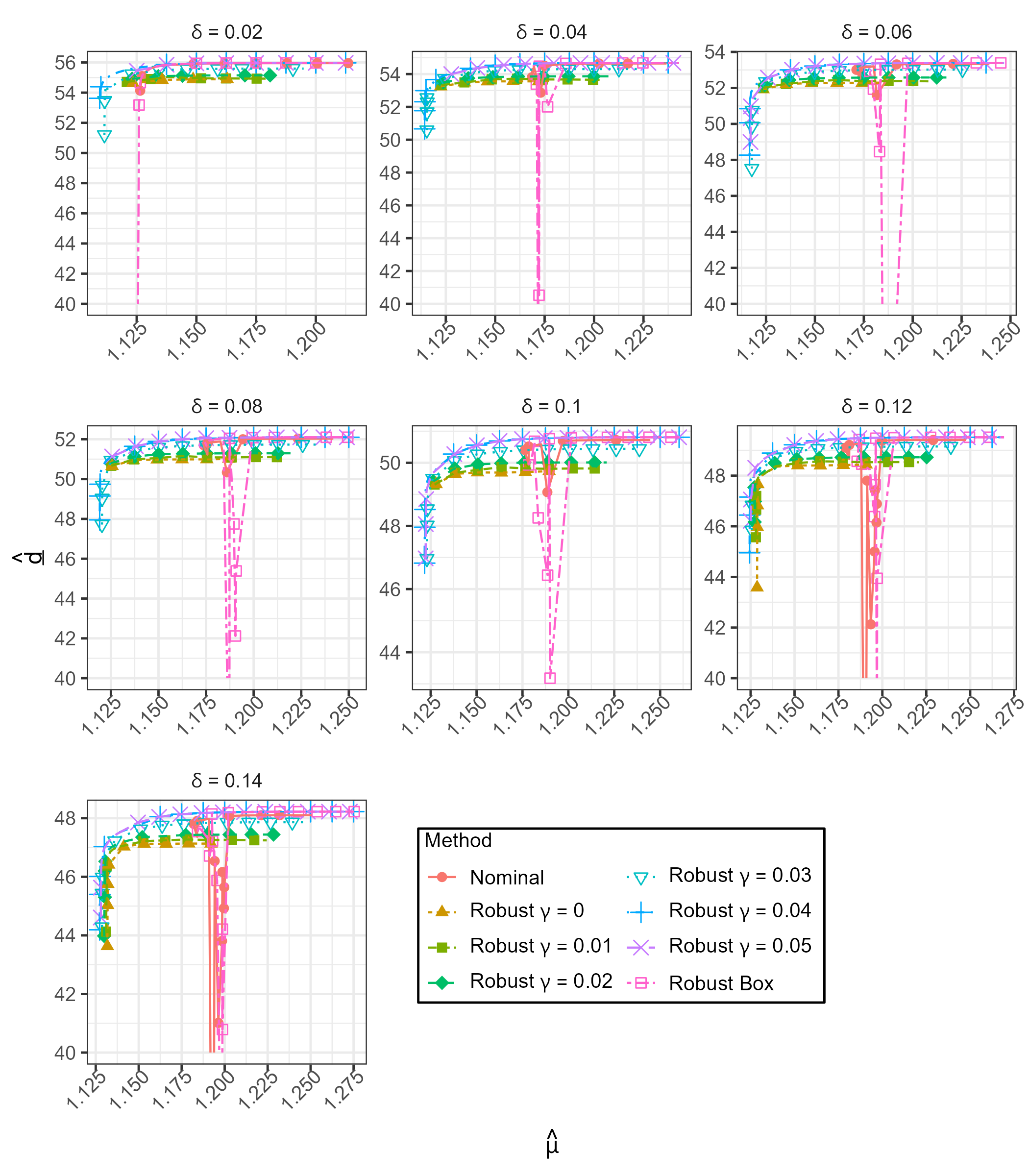}
   \caption{Worst-case performance ($\hat{\mu}$, $\hat{\underbar{d}}$) of solutions optimal to~\eqref{prob:SR_robust} with misspecified $\gamma$, under the assumption that the ``true" $\USB$ has $\gamma=0.04$ and $\delta$ as indicated. The performance of the nominal solution is also shown for comparison.}
\label{fig:preformance_gamma0.04}
\end{figure}

%\end{APPENDICES}
\end{document}